\documentclass [12pt] {amsart}
\usepackage{times}
\usepackage{latexsym,amssymb, amsmath, xcolor}
\usepackage{bm}
\usepackage{setspace}
\usepackage{geometry}
\usepackage{amsmath, amsfonts, amssymb, amsthm}
\newtheorem{thm}{Theorem}
\newtheorem{cor}[thm]{Corollary}
\newtheorem{lemma}[thm]{Lemma}
\newtheorem{remark}[thm]{Remark}

\newtheorem{example}{Example}

\numberwithin{thm}{section}

\DeclareMathOperator{\Ric}{Ric}

\DeclareMathOperator{\diver}{div}

\begin{document}
	
	\title[Local radial rigidity]{On a Local Radial Rigidity Phenomenon of  Elliptic Systems on Riemannian Manifolds}
	
	\author{Yifei Pan}
	\address{Department of Mathematical Sciences\\
		Purdue University Fort Wayne\\
		Fort Wayne, Indiana 46805}
	\email{pan1@pfw.edu}

	\author{ Yu Yan}
	\address {Department of Mathematics and Computer Science\\ Biola University\\La Mirada, California 90639}
	\email {yu.yan@biola.edu}

	\author{Yuan Zhang}
\address{Department of Mathematical Sciences\\
	Purdue University Fort Wayne\\
	Fort Wayne, Indiana 46805}
\email{zhan1313@pfw.edu}

\begin{abstract}
 
In this paper, we study local radial rigidity phenomena for a broad class of general-order elliptic systems on Riemannian manifolds via a reduction to   singular ordinary differential equations of Euler type. These rigidity properties have natural applications to several geometric problems, including prescribed curvature problems of Yamabe and Paneitz type in conformal geometry, higher-order analogues arising from harmonic maps, and Calabi-type problems on K\"ahler manifolds. The method applies also to higher-order nonlinear Poisson systems and, more generally, to weighted divergence-form elliptic systems, yielding local uniqueness and existence results for solutions with prescribed initial jets at the singularity. 

\end{abstract}

\renewcommand{\thefootnote}{\fnsymbol{footnote}}

\maketitle

\section{Introduction}
 Radial symmetry has played a central role in the qualitative theory of elliptic equations since the celebrated work of Gidas, Ni, and Nirenberg \cite{GNN}. They proved that for any Lipschitz continuous function $f$, every positive solution $u \in C^2 (\overline{B_R})$ to the Dirichlet problem
\begin{equation*}
	\left \{
	\begin{aligned}
		&	\Delta u = f(u) \hspace{.2in} &&\text{on } \,\, B_R; \\
			&	u  > 0 \hspace{.42in} &&\text{on } \,\,  B_R; \\
		&	u  = 0 \hspace{.42in} &&\text{on } \,\, \partial B_R
		\end{aligned} \right.
	\end{equation*}
	is necessarily radially symmetric. This symmetry is a global consequence of the boundary condition and positivity, obtained via the maximum principle and the method of moving planes.

On a smooth Riemannian manifold $(M, g)$, the Euclidean notion of radial symmetry is naturally formulated in terms of geodesic distance.  More specifically, a function $u$ defined on a normal neighborhood $ U_p$ of a point ${p}\in M$ is said to be  geodesically radially symmetric in $  U_p$ about $p$ if it depends only on the geodesic distance $r=d_g(\cdot, {p})$ from ${p}$.  In this paper, we study another  aspect of local radial symmetry for functions satisfying a class of  higher-order elliptic systems on $(M, g)$. Rather than deriving radiality from global hypotheses, we investigate whether radial symmetry  can enforce rigidity near the center.

Our main result establishes such a local radial rigidity phenomenon for functions satisfying certain higher-order Poisson-type differential inequalities.

\begin{thm}
	\label{thm:manifold_tri}
	Let $(M, g)$ be a Riemannian manifold, $p \in M$, and $m \in \mathbb{Z}^+$.   Suppose $u = (u_1, \dots, u_{\nu})  $ is a $C^{2m}$ vector-valued function defined on a normal neighborhood $U_p$ of $p$  satisfying  the differential inequality
	\begin{equation}
		\label{eqn:Delta_g^m_inequality}
		|\Delta_g^m u| \leq C \sum_{j=0}^{2m-1} |\nabla_g^{j} u| \quad \text{on } U_p
	\end{equation}
	for some constant $C > 0$. If $u$ is    geodesically radially symmetric in $U_p$ about ${p}$, then exactly one of the following holds:
	\begin{itemize}
		\item There exists some $j_0 \in \{0, \dots, m-1\}$ such that $\partial_r^{2j_0} u(p) \neq 0$; 
		\item $u \equiv 0$ on $U_p$.
	\end{itemize}
	\end{thm}

\noindent
Here $\mathbb{Z}^+$ denotes the set of positive integers, while $\nabla_g$ and $\Delta_g$ denote, respectively, the covariant derivative and the Laplace--Beltrami operator associated with $g$. Moreover, $\Delta_g^m$ denotes the $m$-fold iteration of $\Delta_g$, and $\nabla_g^j u$ denotes the $j$-th covariant derivative of $u$, viewed componentwise as a $(0,j)$-tensor field. The norm $|\nabla_g^j u|$ is defined as $\left(\sum_{i=1}^{\nu} |\nabla_g^j u_i|^2\right)^{\frac{1}{2}}$, where each norm $|\nabla_g^j u_i|$ is computed by contracting the tensor indices using the metric $g$. If $u$ is geodesically radially symmetric about $p$, so that $u=\phi(r)$ for a radial profile $\phi$, then $\partial_r^j u(p)$ is understood as $\phi^{(j)}(0)$.

 In contrast to classical results such as \cite{GNN},   our approach to the radial rigidity is entirely local. The proof of the dichotomy in Theorem \ref{thm:manifold_tri} proceeds by  reducing the problem to a uniqueness theorem (Theorem \ref{gron}) for a singular Euler-type ordinary differential inequality.  This method circumvents the use of global boundary conditions and sign assumptions, and extends naturally to higher-order operators.   The associated ordinary differential equations are analyzed in detail in Sections \ref{section:factor} -- \ref{section:iodea_proof}, where both existence and uniqueness are established.

Despite the linear format of the differential inequality \eqref{eqn:Delta_g^m_inequality}, Theorem \ref{thm:manifold_tri} is meant to be applied to differences of  solutions to nonlinear systems of the form $$ \Delta_g^m u   = f(\cdot, u, \nabla_g u, \dots, \nabla_g^{2m-1}u),$$ 
where the nonlinearity $f$ is sufficiently smooth in its arguments. The local existence results for such systems were established in   \cite{PZ, PY} for  arbitrary   admissible sets of jets at ${p}$ up to order $2m-1$. 
The abundance of such local solutions  naturally gives rise to  the question of whether additional structural assumptions could imply uniqueness.

Motivated by this perspective, we study pairs of local solutions whose difference is radially symmetric and analyze the resulting uniqueness consequences. As a direct consequence of Theorem \ref{thm:manifold_tri},   the following result shows that, when the difference of two solutions to  \eqref{eqn:Delta^m_IVP_Riem} is   radial, the corresponding radial germ is uniquely determined by its even-order  jets up to order $2m-2$ at the base point. In particular, radial symmetry, combined with  sufficiently high-order  contact at a point, imposes strong local constraints on solutions. To the best of our knowledge, this provides a new form of local uniqueness for higher-order elliptic equations/systems under minimal structural assumptions.

	\begin{thm}
		\label{thm:application_manifold_2m-1}
		Let $(M, g)$ be a Riemannian manifold, $p \in M$, and $m \in \mathbb{Z}^+$.   Suppose $u$ and $v$   are two $C^{2m}$ vector-valued functions defined on a normal neighborhood $U_p$ of $p$, each satisfying the nonlinear Poisson system
		\begin{equation}\label{eqn:Delta^m_IVP_Riem}
			\Delta_g^m u = f(\cdot, u, \nabla_g u, \dots, \nabla_g^{2m-1}u) \quad \text{on } U_p, 
		\end{equation}  
		where the vector-valued function $f$ is continuous in all its arguments and locally Lipschitz continuous in the variables $(u, \nabla_g u, \ldots, \nabla_g^{2m-1}u)$. If the difference $u - v$ is geodesically radially symmetric on $U_p$ about $p$, and the even-order covariant derivatives of $u$ and $v$ agree at $p$ up to order $2m-2$:
		\[
		\nabla_g^{2j} u(p) = \nabla_g^{2j} v(p) \qquad \text{for all } j=0,\ldots,m-1,
		\]
		then $u \equiv v$ on $U_p$.
	\end{thm}

In the Euclidean setting,  Theorem \ref{thm:application_manifold_2m-1} immediately yields the following corollary for the standard Laplacian and gradient operators.

\begin{cor} \label{thm:application_f_x_u_nabla_u_2m-1}
	Let   $m, n \in \mathbb{Z}^+$ and $n\ge 2$. Suppose $u$ and $v$ are two $C^{2m}$ vector-valued functions on an open ball $B$  in $\mathbb R^n$  centered at the origin, each satisfying the nonlinear Poisson system
	\[
	\Delta^m u = f(\cdot, u, \nabla u, \dots, \nabla^{2m-1}u) \quad \text{in } B,
	\]
	where the vector-valued function $f$ is continuous in all its arguments and locally Lipschitz continuous in the variables $(u, \dots, \nabla^{2m-1}u)$. If the difference $u - v$ is radially symmetric in $B$ about the origin, and the even-order derivatives of $u$ and $v$ agree at $0$ up to order $2m-2$:
	\[
	\nabla^{2j} u(0) = \nabla^{2j} v(0) \qquad \text{for all } j = 0, \ldots, m-1,
	\]
	then $u \equiv v$ in $B$.
\end{cor}

	As a byproduct of the singular ODE theory, we  generalize the Laplace operator to a class of weighted divergence-form elliptic operators defined near a singular point $0 \in \mathbb R^n $. These operators take the form
$$\Delta_{(c_1, c_2)} u := r^{-c_2} \diver \left( r^{c_2-c_1} \nabla (r^{c_1} u) \right),$$
where $(c_1, c_2) \in \mathbb{C}^2$, $u$ is a vector-valued  $ C^2 $ function near $0$, and $r = |x|$. 
For $m \in \mathbb{Z}^+$ and a parameter vector $(c_1, \ldots, c_{2m}) \in \mathbb{C}^{2m}$, we define the iterated operator as the composition $\Delta_{(c_{2m-1}, c_{2m})} \cdots \Delta_{(c_1, c_2)}.$ A detailed discussion of the properties of these operators, along with several illustrative examples, is provided in Section \ref{last}. Despite their seemingly  involved structure,   every singular  operator of the form $   \Delta^{m}      + \sum_{{  0\le j+2k\le 2m-1}} \frac{d_{j, k}}{r^{2m-2k}} (x\cdot \nabla)^j\Delta^k ,$  where $  d_{j, k} \in \mathbb C $, acts on  a $C^{2m}$ radial function  $u$ according to the identity 
        $$
\Delta^{m} u
+
\sum\nolimits_{{  0 \le j+2k \le 2m-1}}
\frac{d_{j,k}}{r^{2m-2k}} (x \cdot \nabla)^j \Delta^k u
=
\Delta_{(c_{2m-1}, c_{2m})} \cdots \Delta_{(c_1, c_2)} u 
$$
for suitable complex constants $c_1, \ldots, c_{2m}$. 
     
     By establishing the uniqueness of radial solutions to a  linear differential inequality involving the operator $\Delta_{(c_{2m-1}, c_{2m})} \cdots \Delta_{(c_1, c_2)}$ in Theorem \ref{lpdeg}, we obtain a rigidity phenomenon for nonlinear systems of the form
        	\begin{equation*}
				\Delta_{(c_{2m-1}, c_{2m})} \cdots \Delta_{(c_1, c_2)} u  = f(\cdot, u , \nabla u , \dots, \nabla^{2m-1} u )   
				\end{equation*}
 under suitable conditions on the parameters $(c_1,\dots, c_N)$. 
See    Corollary \ref{gthm:application_f_x_u_nabla_u_2m-1}. 

On the other hand, existence of radial solutions to the above nonlinear system  
is established in Theorem \ref{pe}, provided that the nonlinearity $f$ depends only on the  radial variables, namely  $ (r, u, \partial_r  u, \dots, \partial_r^{2m-1} u) $. In particular, the theorem yields $C^{2m}$ radial solutions to the above system with any properly prescribed initial values at the singularity.   See also Corollaries \ref{pel} and \ref{re2}. These results extend  readily to  general Riemannian manifolds,  and we leave the details to the  interested reader.

    \vspace{.2in}

	\section{Geometric applications}\label{geo}

       Owing in particular  to the general form of the differential system under consideration,  Theorem \ref{thm:application_manifold_2m-1} and Theorem \ref{gron}  have substantial applications to  a broad class of geometric elliptic problems and  yield rigidity phenomena for such problems.   In this section, we present several  examples in which our local rigidity results apply.   These examples   arise from geometric PDEs, including curvature prescription problems, variational problems such as harmonic maps, and fully nonlinear equations.

    \subsection{Yamabe  and Paneitz problems}
     
     Let $(M, g)$ be an $n$-dimensional Riemannian manifold with $n \geq 3$. A classic problem in geometric analysis is to find a conformal metric \[
    \tilde g=u^{\frac{4}{n-2}}g,\qquad u>0,
\]
whose scalar curvature is a prescribed function $K$ on $M$. The conformal
factor $u$ satisfies the Yamabe equation (\cite{LP})
	\begin{equation*}
		\Delta_g u = c_n R_g\, u - c_n K\, u^{\frac{n+2}{n-2}},
	\end{equation*}
	where $c_n = \frac{n-2}{4(n-1)}$ and $R_g$ is the scalar curvature of the background metric $g$.
	
    Because the exponent $\frac{n+2}{n-2}$ is strictly greater than $1$, the nonlinearity in the above equation is locally Lipschitz continuous with respect to $u$. This enables us to apply Theorem \ref{thm:application_manifold_2m-1} in the case $m=1$, which yields the following rigidity result 
    for osculating conformal deformations.
    
\begin{thm}\label{yamabe}
	Let $(M, g)$ be a Riemannian manifold of dimension $n \geq 3$. Suppose $\tilde{g}_1 = u^{\frac{4}{n-2}} g$ and $\tilde{g}_2 = v^{\frac{4}{n-2}} g$ are two conformal metrics defined on a normal neighborhood $U_p$ of a point $p\in M$ that both prescribe the same local scalar curvature $K$ on $U_p$. If the difference  $u - v$ of the two conformal factors is geodesically radially symmetric in $U_p$ about $p$  and $u(p) = v(p)$,
    then the conformal metrics are identical, i.e., $\tilde{g}_1 \equiv \tilde{g}_2$ in $U_p$. 
    
\end{thm}

	 This phenomenon becomes more pronounced in higher-order conformal geometry, where maximum principles are typically unavailable. The conformal Paneitz operator  $P_g^4$  is a fourth-order conformally covariant differential operator that plays a central role in higher-order conformal geometry (see \cite{Paneitz, Bran, CY, Chang, GM, HY} and the references therein). Its leading term is the bi-Laplacian,  together with   lower-order curvature corrections:
	\begin{equation*}
		P_g^4 u = \Delta_g^2 u + \operatorname{div}_g \Big( a_n R_g\,\nabla_g u + b_n \operatorname{Ric}_g(\nabla_g u) \Big) + \frac{n-4}{2} Q_g\,u,
	\end{equation*}
	where $a_n$ and $b_n$ are dimensional constants, $Q_g$ is the $Q$-curvature of the metric, and $\operatorname{Ric}_g(\nabla_g u)$ denotes the vector field obtained by raising the index of the 1-form $\operatorname{Ric}_g(\nabla_g u, \cdot)$.
	For $n\neq 4$, under the conformal change
\[
    \tilde g=u^{\frac{4}{n-4}}g,\qquad u>0,
\]
the Paneitz operator satisfies the conformal covariance property
\[
    P_{\tilde g}^4(\psi)
    =
    u^{-\frac{n+4}{n-4}}\,P_g^4(u\psi)
\]
for every smooth function $\psi$.

    More generally, the GJMS   (Graham--Jenne--Mason--Sparling) operator $P_g^{2m}$ (\cite{GJMS}) is a higher-order generalization of the Paneitz operator for $m\in \mathbb Z^+$.  For a manifold of dimension $n > 2m$, it takes the form
	$$P_g^{2m} = \Delta_g^m + \text{lower-order geometric terms involving derivatives up to order } 2m-1.$$
Under the conformal change
\[
    \tilde g=u^{\frac{4}{n-2m}}g,
\]
the prescribed $2m$-th order $Q$-curvature equation, after possibly absorbing the dimensional constant into the $Q$-curvature $Q$, takes the form
\[
    P_g^{2m}u
    =
    Q \, u^{\frac{n+2m}{n-2m}}.
\]

Moving the lower-order geometric terms of $P_g^{2m}$ to the right-hand side,
the equation assumes the form of \eqref{eqn:Delta^m_IVP_Riem}. Since  the exponent satisfies
$    \frac{n+2m}{n-2m}>1,$
the nonlinear term is locally Lipschitz with respect to $u$. Therefore
Theorem~\ref{thm:application_manifold_2m-1} applies and leads to

\begin{thm}
	\label{paneitz}
	Let $(M, g)$ be a Riemannian manifold of dimension $n > 2m$, where $m\in \mathbb Z^+$. Suppose $\tilde{g}_1 = u^{\frac{4}{n-2m}} g$ and $\tilde{g}_2 = v^{\frac{4}{n-2m}} g$ are two conformal metrics  defined on a normal neighborhood $U_p$ of a point $p\in M$ that prescribe the same $Q$-curvature  $Q $  on $U_p$. If the difference $u-v$ is geodesically radially symmetric in $U_p$ about $p$ and  the even jets of $u$ and $v$ at ${p}$ agree up  to order $2m-2$,  then the two  metrics $\tilde{g}_1  $ and $\tilde{g}_2$ coincide in $U_p$.
\end{thm}
	
	These results impose a strong geometric constraint on bifurcation phenomena: if a radially symmetric conformal factor solves a prescribed curvature problem locally, then any distinct solution with the same initial data at the pole cannot remain radially symmetric. In particular, any bifurcation must arise through the breaking of local rotational symmetry.

\subsection{Harmonic maps} 

 Introduced by Eells and Sampson \cite{ES}, harmonic maps serve as a natural geometric counterpart to semilinear elliptic systems and have been extensively studied in both local and global settings. 
Let $(M, g)$ and $(N, h)$ be two Riemannian manifolds. A smooth map  $u:(M,g)\to (N,h)$ is called harmonic if it is a critical point  of the Dirichlet energy
\[
E(u)=\frac{1}{2}\int_M |\nabla u|^2\, dV_g.
\]
Consequently, $u$ satisfies the Euler--Lagrange equation $$\tau(u)=0,$$ where $\tau(u)$ is the tension field. In local coordinates on the target manifold $(N, h)$, this equation takes the form
\[
\Delta_g u^\alpha + \Gamma^\alpha_{\beta\gamma}(u)\,\langle \nabla_g u^\beta, \nabla_g u^\gamma\rangle_g = 0,
\]
where the functions $(u^\alpha)$ are the coordinate components of $u$, and \(\Gamma^\alpha_{\beta\gamma}\) are the Christoffel symbols of the Levi--Civita connection on \(N\), evaluated at \(u\). We use the Einstein summation convention, so repeated indices are summed over. This formulation exhibits a quasilinear elliptic structure with a quadratic nonlinearity in the gradient. See also \cite{Jost, EL} for background.

Higher-order analogues arise from higher-order variational energies. For example, biharmonic maps are critical points of
\[
E_2(u)=\frac{1}{2}\int_M |\tau(u)|^2\, dV_g,
\]
and satisfy the Euler--Lagrange equation
\[
\tau_2(u)=\Delta_g \tau(u) + \text{lower-order curvature terms} = 0,
\]
which can be viewed schematically as a fourth-order elliptic system with leading term $\Delta_g^2 u$.   More generally, polyharmonic maps of order $m$ satisfy equations of the form
\[
\Delta_g^m u = f(\cdot,u,\nabla_g u,\dots,\nabla_g^{2m-1}u),
\]
where the nonlinearity $f$ is determined by the geometry of the target manifold  and is smooth in all its arguments $(\cdot,u,\nabla_g u,\dots,\nabla_g^{2m-1}u) $. See, for instance, \cite{Jiang, MO, GS}.  These systems therefore fall within the framework developed in this paper.

To apply our results, we fix an isometric embedding $(N,h)\subset \mathbb{R}^{\nu}$ by Nash's theorem, and regard maps $u$ and $v$ as $\mathbb{R}^{\nu}$-valued functions. In this setting, the difference $u-v$ is well-defined, and radial symmetry is understood with respect to this fixed embedding. We note that, while radial symmetry of a single map is an  intrinsic notion, the radiality of $u-v$ is extrinsic, since it may depend on the chosen embedding  of $(N,h)$ into $\mathbb{R}^{\nu}$. 
Under these considerations, the local rigidity result stated below follows directly from Theorem \ref{thm:application_manifold_2m-1}.

   \begin{thm}
\label{harm}
Let $(M,g)$ be a Riemannian manifold and $(N,h)\subset \mathbb{R}^{\nu}$ an isometrically embedded Riemannian manifold. Let   $m \in \mathbb{Z}^+$. Suppose $u,v:U_p \to N$ are two polyharmonic maps of order $m$, where  $U_p$ is a normal neighborhood $U_p$ of a point $p\in M$.  If $u-v$ is geodesically radially symmetric in $U_p$ about $p$ and the even jets of $u$ and $v$ at $p$ agree up to order $2m-2$, then $u \equiv v$ in $U_p$.
\end{thm}
 
In the case   $m=1$, since   constant maps are  trivially   harmonic and radial, Theorem \ref{harm} implies that any locally geodesically radially symmetric harmonic map must also  be constant. 

 \begin{cor}
\label{harm1}
Let $(M,g)$ be a Riemannian manifold and $(N,h)\subset \mathbb{R}^{\nu}$ an isometrically embedded Riemannian manifold.  Let $u :U_p \to N$ be a harmonic map, where $U_p$ is a normal neighborhood of a point $p\in M$. If  $u$ is geodesically radially symmetric in $U_p$ about $p$, then $u$ is constant in $U_p$.
\end{cor}

   \subsection{Calabi problem} The  Calabi--Yau equation is a fundamental example of geometric fully nonlinear elliptic equations  arising in K\"ahler geometry. It arises as the complex Monge--Amp\`ere equation associated with prescribing the Ricci form of a K\"ahler metric within a fixed K\"ahler class.
Calabi initiated this program by reducing the problem to an  equation of the form
 \begin{equation*}
(\omega+\sqrt{-1}\,\partial\bar\partial\varphi)^n=e^f\omega^n,
\end{equation*}
and formulated what is now known as the Calabi conjecture, predicting the existence and uniqueness of a K\"ahler metric in a given cohomology class with prescribed volume form, or equivalently, prescribed Ricci form. More precisely, if \((M,\omega)\) is a K\"ahler manifold of complex dimension \(n\) and \(f\) is a given smooth function, the Calabi--Yau equation seeks a K\"ahler potential \(\varphi\) satisfying
    \begin{equation}\label{calabi}
(\omega+\sqrt{-1}\,\partial\bar\partial\varphi)^n=e^f\omega^n, 
\qquad
\omega+\sqrt{-1}\,\partial\bar\partial\varphi>0.
\end{equation}
On compact K\"ahler manifolds, the global existence of solutions to this equation was established by Yau \cite{Yau1978} through  a priori estimates. Moreover, the global solution is unique up to an additive constant.

While this result provides a complete understanding in the global setting, the local behavior of solutions remains more flexible and less rigid. 
  Concerning local solutions to the Calabi--Yau equation,  since  jets of the potential $\varphi$ up to order one do  not yield distinct K\"ahler forms, we study the existence of local potentials    by prescribing  admissible second-order jets of the potential in Theorem \ref{cale}.     Consequently,  there exist    infinitely many K\"ahler forms  with the prescribed Ricci form in a neighborhood of a point. With regard to uniqueness,   classical  results  are typically global in nature and rely on the maximum principle or energy methods, often combined with compactness, boundary conditions, or suitable decay assumptions. 
  While Theorem \ref{thm:manifold_tri} does not directly apply to the fully nonlinear  equation,  the linearization of the equation allows us to apply the  singular ODE theory (Theorem \ref{gron})   to establish  a local rigidity result for K\"ahler forms with prescribed Ricci form under a radial symmetry assumption.  In summary, we obtain in Section \ref{cas} the following local existence and uniqueness  for the Calabi-Yau equation.
  
\begin{thm}\label{thm:local-kahler-form-rigidity}
	
	Let $(M,\omega)$ be a K\"ahler manifold and let $p\in M$.   
	
	1. There exist infinitely many K\"ahler forms in the same local K\"ahler class as $\omega$ with the prescribed Ricci form in a neighborhood of $p$.  
	
	2. If  $\omega_0$ and $\omega_1$ are two K\"ahler forms in the same local K\"ahler class as $\omega$ with the prescribed Ricci form in a normal neighborhood $U_p$ of $p$, where 
	$$ \omega_\ell=\omega+\sqrt{-1}\,\partial\bar\partial\varphi_\ell, \quad \ell=0,1, $$
	and if the potential difference $\varphi_1-\varphi_0 $ is geodesically radially symmetric in $U_p$ about $p$, then $\omega_0 \equiv \omega_1$ in $U_p$.

\end{thm}

In particular, as shown in the proof of Theorem  \ref{thm:local-kahler-form-rigidity}, the potentials corresponding to  the  K\"ahler forms constructed in part \(1\)  have distinct second-order jets.  Part $2$  then implies that none of  these potentials  differs by a locally geodesically radially symmetric function near \(p\).

\vspace{.1in}

The rest of the paper is organized as follows. Sections \ref{section:factor} -- \ref{section:iodea_proof} develop the uniqueness and existence theory for Euler-type ordinary differential equations, which provides the main analytic tool for the subsequent PDE arguments. In detail,   Section \ref{section:factor} introduces an algebraic factorization of Euler-type operators. The uniqueness of solutions to Euler-type ODEs  is proved in Sections \ref{section:ee_proof} and \ref{section:gron_proof}, respectively. See Theorems \ref{ee} for smooth solutions and \ref{gron} for solutions with finite regularity. Building  on these results, Section \ref{section:iodea_proof} develops the existence and uniqueness theory for the associated singular initial value problems (Theorem \ref{iodea}). Beginning with Section \ref{sec8},  we turn to the PDE setting. There, we apply  the  ODE reductions developed earlier  to prove the local radial rigidity dichotomy for higher-order Poisson systems on Riemannian manifolds (Theorems \ref{thm:manifold_tri} and  \ref{thm:application_manifold_2m-1}). 
The local existence and uniqueness result (Theorem \ref{thm:local-kahler-form-rigidity}) for the Calabi-Yau equation is carried out in Section \ref{cas}. Finally, Section \ref{last}   extends this approach to a broader class of weighted divergence-form elliptic operators, establishing the corresponding uniqueness and existence results. 

\vspace{.2in}

\section{Decomposition of Euler-type operators}
\label{section:factor}
A classical model for singular ordinary differential operators near the origin is the Euler-type operator
\begin{equation}\label{euler}
w^{(N)} + \frac{a_1}{r} w^{(N-1)} + \cdots + \frac{a_N}{r^N} w,
\end{equation}
where $N \in \mathbb{Z}^+$   and $(a_1,\dots,a_N)\in\mathbb C^N$.
Such operators exhibit an intrinsic scale-invariant structure and play a fundamental role in the analysis of singular differential equations and inequalities.

A key structural property of Euler-type operators is their factorization into first-order singular operators. To formalize this, let $b \in \mathbb{C}$ and $r_0 > 0$. For a scalar function $w \in C^1((0, r_0))$, we define the first-order operator $\mathcal{D}_b$ by
\begin{equation}\label{1st}
    \mathcal{D}_b w(r) := r^{-b} \frac{d}{dr} \left( r^b w(r) \right) = w'(r) + \frac{b}{r} w(r).
\end{equation}
Here the complex power is defined by $r^b = e^{b \log r}$, where $\log$ denotes the branch of the complex logarithm with a cut along the negative real axis, corresponding to $\arg z \in (-\pi ,  \pi)$.
More generally, for $N \in \mathbb{Z}^+$ and a parameter vector $(b_1, \dots, b_N) \in \mathbb{C}^N$, we consider the $N$-th order operator obtained by composition:
$$\mathcal{D}_{b_N} \cdots \mathcal{D}_{b_1} : C^N((0, r_0) ) \to C((0, r_0) ).$$
For vector-valued functions $w = (w_1, \dots, w_{\nu}): (0, r_0) \to \mathbb{C}^{\nu}$, the operator acts component-wise.

In this section, we show that any Euler-type   operator of the form  \eqref{euler} can be decomposed into a product of first-order operators of the form $\mathcal{D}_{b_N} \cdots \mathcal{D}_{b_1}$, with parameters $b_1, \ldots, b_N\in \mathbb C$.

 \begin{lemma}
 	\label{ele2}
 	For any $(a_1, \dots, a_N) \in \mathbb{C}^N$, there exist constants $(b_1, \dots, b_N) \in \mathbb{C}^N$ such that
 	\begin{equation}
 		\label{cor}
 		w^{(N)} + \frac{a_1}{r}w^{(N-1)} + \cdots + \frac{a_N}{r^N} w = \mathcal{D}_{b_N} \cdots \mathcal{D}_{b_1} w.
 	\end{equation}
 \end{lemma}

Before proving the lemma in the most general form, we first provide a concrete example to clarify the idea of the proof.

 \vspace{0.1cm}
 
 \begin{example}
 	We seek constants $b_1, b_2, b_3$ such that
 	$$w''' -\frac{2}{r}w'' +\frac{3}{r^2}w' -\frac{3}{r^3}w = \mathcal D_{b_3} \mathcal D_{b_2}\mathcal D_{b_1} w. $$
 \end{example}

 \noindent
First, we determine constants $c_1, c_2, c_3$ such that
\begin{align*}
	 w''' -\frac{2}{r}w'' +\frac{3}{r^2}w' -\frac{3}{r^3}w & = \mathcal{D}_{c_3} \left(w''+ \frac{c_1}{r}w'+\frac{c_2}{r^2}w  \right) \\
	 & =  \frac{d}{dr}\left(w''+ \frac{c_1}{r}w'+\frac{c_2}{r^2}w  \right) +\frac{c_3}{r}\left( w'' +\frac{c_1}{r}w'+\frac{c_2}{r^2}w \right).
\end{align*}

 \vspace{.05in}
 \noindent
 By expanding the right-hand side and matching coefficients, we obtain the system
 	\begin{equation*}
 		\begin{cases}
 			 c_1+c_3 = -2;\\
 			 c_{2} + c_1(c_3-1)= 3;\\
 			 c_2(c_3-2) = -3.
 		\end{cases}
 	\end{equation*}
 From the first equation, we have $$c_1 = -2 - c_3.$$ Substituting this into the second equation yields $$c_2 = 3 - (-2 - c_3)(c_3 - 1) = c_3^2 + c_3 + 1.$$
 Finally, substituting this expression for $c_2$ into the last equation, we obtain
 $\left( c_3^2 + c_3 + 1 \right)(c_3 - 2) = -3,$
 which simplifies to $$  (c_3-1)^2(c_3+1)=0. $$  This gives two possible solutions:
  $$(c_1, c_2, c_3) = (-3, 3 ,1) \ \text{or}\   (-1, 1,-1).$$

 Setting $b_3 = c_3$, we now look for constants $b_1, b_2$ such that the second-order component factors as
  $$w'' +\frac{c_1}{r}w'+\frac{c_2}{r^2}w = \mathcal D_{b_2}\mathcal D_{b_1} w =  \frac{d}{dr}\left(w'+\frac{b_1}{r}w\right)+\frac{b_2}{r} \left(w'+\frac{b_1}{r}w\right) . $$
  This requires
 	\begin{equation*}
 	  			  b_2+b_1= c_1,\ \ \ \  			  b_1(b_2-1) = c_2,
  	\end{equation*}
 	  	For the first case $(c_1, c_2, c_3) = (-3, 3 ,1)$, we have $(b_1, b_2) = (-1,  -2 )$ or $(-3, 0)$. 	For the second case $(c_1, c_2, c_3) = (-1, 1 ,-1)$, we have $(b_1, b_2) = (1, 0)$.
 		 Altogether, the possible choices are
 	$$(b_1, b_2, b_3) =(-1, -2, 1), \text{or}\  (-3, 0, 1),\text{or} \ (1, 0, -1) .$$
\medskip

 \begin{proof}[\textbf{Proof of Lemma \ref{ele2}}: ]
 	
 	We proceed by induction on $N$. For the base case $N=2$, the lemma can be verified by a direct computation:
 	$$\mathcal{D}_{b_2}\mathcal{D}_{b_1} w = \left( \frac{d}{dr} + \frac{b_2}{r} \right) \left( w' + \frac{b_1}{r}w \right) = w'' + \frac{b_1+b_2}{r} w' + \frac{b_1(b_2-1)}{r^2} w.$$
 	Thus, given $(a_1, a_2) \in \mathbb{C}^2$, it suffices to choose $b_1, b_2 \in \mathbb{C}$ satisfying the system
 	$$b_1 + b_2 = a_1, \qquad b_1(b_2 - 1) = a_2.$$
 	
    Now, assume the lemma holds for an operator of order $N-1$. To prove it for order $N$, we first seek coefficients $(c_1, \ldots, c_N) \in \mathbb{C}^N$ such that
 	 			\begin{align*}
 	&	w^{(N)} + \frac{a_1}{r}w^{(N-1)} +\cdots + \frac{a_N}{r^N} w \\= 	&\mathcal{D}_{c_N} \left( w^{(N-1)} + \frac{c_1}{r}w^{(N-2)} + \cdots + \frac{c_{N-1}}{r^{N-1}} w \right) \\
 		 = &\frac{d}{dr}\left( w^{(N-1)} + \frac{c_1}{r}w^{(N-2)} + \cdots + \frac{c_{N-1}}{r^{N-1}} w \right)  +  \frac{c_N}{r}  \left( w^{(N-1)} + \frac{c_1}{r}w^{(N-2)} + \cdots + \frac{c_{N-1}}{r^{N-1}} w \right) .
 	\end{align*}
 	Matching coefficients on both sides leads to the system
 	\begin{equation}
 		\label{sys}
 		\begin{cases}
 			c_1 + c_N = a_1, \\
 			c_{j+1} + c_j(c_N - j) = a_{j+1}, & 1 \leq j \leq N-2, \\
 			c_{N-1}(c_N - N + 1) = a_N.
 		\end{cases}
 	\end{equation}

 	The first $N-1$ equations allow us to express $c_1, \ldots, c_{N-1}$ recursively in terms of $c_N$. Specifically, they take the form
 	\begin{equation}
 		\label{sys2}c_{j} = a_{j} - P_{j}(a_1, \ldots, a_{j-1}, c_N), \qquad 1\le j \le N-1,
 		\end{equation}
 		where $P_{j}$ is a polynomial in $c_N$ of degree $j$. Substituting the resulting expression for $c_{N-1}$ into the final equation of \eqref{sys} yields a single polynomial equation in $c_N$ of degree $N$. By the Fundamental Theorem of Algebra, this equation admits at least one solution in $\mathbb{C}$. For any such $c_N$, the coefficients $c_1, \ldots, c_{N-1}$ are uniquely determined by \eqref{sys2} and
 		the operator factors as
 		\begin{equation*}
 			w^{(N)} + \frac{a_1}{r}w^{(N-1)} + \cdots + \frac{a_N}{r^N} w = \mathcal{D}_{c_N} \left( w^{(N-1)} + \frac{c_1}{r}w^{(N-2)} + \cdots + \frac{c_{N-1}}{r^{N-1}} w \right).
 			\end{equation*}
 			
 			By the inductive hypothesis, there exist constants $(b_1, \dots, b_{N-1}) \in \mathbb{C}^{N-1}$ such that the $(N-1)$-th order operator in the parentheses satisfies
 			$$w^{(N-1)} + \frac{c_1}{r}w^{(N-2)} + \cdots + \frac{c_{N-1}}{r^{N-1}} w = \mathcal{D}_{b_{N-1}} \cdots \mathcal{D}_{b_1} w.$$
 			
 			\noindent
 			Setting $b_N = c_N$, we conclude that
 			$$w^{(N)} + \frac{a_1}{r}w^{(N-1)} + \cdots + \frac{a_N}{r^N} w = \mathcal{D}_{b_N} \mathcal{D}_{b_{N-1}} \cdots \mathcal{D}_{b_1} w,$$
 			which completes the induction and the proof.

 \end{proof}

\vspace{.2in}

\section{Uniqueness of smooth solutions to Euler-type ODEs }
\label{section:ee_proof}
The goal of this section is to study uniqueness of smooth solutions to Euler-type singular ordinary differential inequalities. As shown in \cite{PW} by the first author and Wang, the precise Euler-type singularity is  essential for uniqueness. In fact,  for any $\varepsilon>0$, there exists a nontrivial function $w_\epsilon \in C^\infty([0,r_0))$ such that $$
| w_\epsilon^{(N)}(r) | \le C_\epsilon \sum_{j=0}^{N-1}
\frac{|w_\epsilon^{(j)}(r)|}{r^{N-j+\varepsilon}}
$$ for some constant $C_\epsilon>0$, and $w_\epsilon^{(j)}(0)=0 $ for all $j \ge 0$.  Here $|\cdot|$  is the Euclidean norm. This example shows that any weakening of the Euler-type singularity may lead to a failure of uniqueness. On the other hand, the following unique continuation result  holds for Euler-type inequalities.

\begin{thm}
\label{thm:Pan_Wang_flatness}
(\cite{PW}, Theorem 5)
Let $0 \in [a,b]$ and $N\in \mathbb Z^+$. Suppose  $w =(w_1 , \dots, w_{\nu} ): [a, b]\rightarrow \mathbb C^{\nu}$ is  smooth on $[a, b]$ and satisfies the differential inequality
\[
| w^{(N)}(r) | \le C \sum_{j=0}^{N-1} \frac{|w^{(j)}(r)|}{|r|^{\,N-j}},
\qquad r \in [a,b]\setminus \{0\}
\]
for some constant $C>0$. Then
\[
w^{(j)}(0)=0 \quad \text{for all } \,\, j \ge 0
\quad \Longrightarrow \quad
w \equiv 0 \,\, \text{ on } \,\, [a,b].
\]
\end{thm}

The   flatness assumption  $w^{(j)}(0)=0 $ for all $j \ge 0$ cannot, in general, be weakened to finite-order vanishing. Indeed, for any fixed $k \in \mathbb{Z}^+$, one can construct a standard Euler equation that admits a non-trivial solution  vanishing to order higher than $k$ at the origin. For instance, the smooth function  $w= r^{k+1}$ vanishes to order $k$ at the origin and satisfies $ w' - \frac{k+1}{r}w =0$.
Thus the vanishing of only finitely many  jets is generally insufficient to force triviality.  A related result \cite{PW2} shows, however, that finite-order vanishing does imply triviality, provided the constant $C$ is sufficiently small.

In this section, we establish a   uniqueness result for smooth solutions to the corresponding ordinary differential inequality using the  factorized singular operator $\mathcal{D}_{b_N} \cdots \mathcal{D}_{b_1}$,  under a   mild condition on $(b_1, \ldots, b_N)\in \mathbb C^N$.  
 In contrast to  \cite{PW2},  no smallness assumptions are imposed on the constant $C$ or the parameters $|b_j|$.

\begin{thm}
	\label{ee}
     Let $r_0>0, N\in \mathbb Z^+$ and $(b_1, \ldots, b_N)\in \mathbb C^N$.  Suppose that   $w = (w_1, \dots, w_{\nu}): [0, r_0) \to \mathbb{C}^{\nu}$ is a vector-valued smooth function on $[0, r_0) $ (i.e., each component $w_i \in C^\infty([0, r_0))$) and satisfies the differential inequality
      \begin{equation}\label{gode}
     	   | \mathcal{D}_{b_N} \cdots \mathcal{D}_{b_1} w(r) | \leq h(r)\sum_{j=0}^{N-1} \frac{| w^{(j)}(r) |}{r^{N-j}}, \quad r \in (0, r_0) 
\end{equation}
     for some nonnegative function $h\in C([0, r_0))$ with  $  h(0) =0$.  If the parameters $b_j$ satisfy
     \begin{equation}
     	\label{bj}
         j -  b_j  - N \notin \mathbb{Z}^+, \ \  j = 1, \ldots, N,
     \end{equation}
and $w$ satisfies the vanishing conditions$$w^{(j)}(0) = 0, \ \  j = 0, \dots, N-1,$$  then $ w \equiv 0 $ on $[0, r_0)$.

\end{thm}

In particular,   as a consequence of Theorem \ref{ee}, any nonconstant function satisfying \eqref{gode} under condition \eqref{bj} and vanishing to order $N-1$ at $0$ cannot be smooth at $0$.  
\medskip

Rather than proving Theorem \ref{ee} directly, we establish a  uniqueness result under  a weaker differential inequality, from which Theorem \ref{ee} follows immediately.  
To this end, we first note that if $(b_1, \ldots, b_N)\in \mathbb C^N$ satisfies \eqref{bj}, then
\begin{equation}\label{Mb}
 M_b: = \sup_{p\in \mathbb Z, p\ge N}   \dfrac{ \prod_{k=1}^{N-1} \left |p-(k-1) \right| }{ 	\prod_{j=1}^N |p-(j-1)+b_j|} <\infty.
\end{equation}
 To justify this, recall that condition \eqref{bj} requires $ j-   b_j  -N \notin \mathbb{Z}^+  $  for each $j = 1, \dots N$.  Thus for any integer $p \ge N$,  $$p  - j + 1 +  b_j  = (p  -N+1) - \left (j-  b_j  -N \right )\ne 0. $$
  Consequently,
  \begin{equation}\label{pne}
      \prod_{j=1}^N |p-(j-1)+b_j| \neq 0.
  \end{equation}	 Moreover,
 	\begin{align*}
 	&	\lim_{p \to \infty}  \dfrac{ \prod_{k=1}^{N-1} \left |p-(k-1) \right| }{ 	\prod_{j=1}^N |p-(j-1)+b_j|} \\
 	  = 	&\lim_{p \to \infty} \left | \dfrac{p}{p+b_1} \right | \cdot  \left | \dfrac{p-1}{p-1+b_2} \right | \dots \left | \dfrac{p-(N-2)}{p-(N-2)+b_{N-1}} \right | \cdot \dfrac{1}{\left |p-(N-1)+b_{N} \right | }  =0.
  	\end{align*}
It follows that the supremum  $M_b$ is finite. Note that  $M_b$ depends only  on $(b_1, \ldots, b_N)$.

\medskip

\begin{thm}
	\label{ees}
	Let $r_0 > 0$, $N \in \mathbb Z^+$, and $(b_1, \ldots, b_N) \in \mathbb{C}^N$ satisfy \eqref{bj}. Let $M_b$ be  defined as   in \eqref{Mb}.
      Suppose that   $w = (w_1, \dots, w_{\nu}): [0, r_0) \to \mathbb{C}^{\nu}$ is smooth on $[0, r_0) $ and satisfies the differential inequality
\begin{equation}
     	\label{godewe}
   | \mathcal{D}_{b_N} \cdots \mathcal{D}_{b_1} w(r) | \leq C\sum_{j=0}^{N-1} \frac{| w^{(j)}(r) |}{r^{N-j}}, \quad r \in (0, r_0),
\end{equation}
     where the constant $C$ satisfies \begin{equation}\label{cl1}
         C< \frac{1}{M_b N}.
     \end{equation}
    If $$w^{(j)}(0) = 0, \quad   j = 0, \dots, N-1,$$
   then $ w \equiv 0 $ on $[0, r_0)$.
   \end{thm}

Our approach is based  on the observation that, for  smooth functions subject to  operators satisfying \eqref{bj}, vanishing to   order $N-1$ implies the flatness at the origin. This reduction allows us to apply Theorem \ref{thm:Pan_Wang_flatness} directly to conclude the triviality of solutions. 

\begin{proof}[\textbf{Proof of Theorem \ref{ees}:}] Since the differential inequality \eqref{godewe} can be rewritten as
$$|w^{(N)}(r) | \leq \tilde C \sum_{k=0}^{N-1} \frac{|w^{(k)}(r)|}{r^{N-k}} $$
for some $\tilde C$ dependent only on $C$ and $b_j, j=1, \ldots, N$, in view of Theorem \ref{thm:Pan_Wang_flatness}  it suffices to show  that   $w $ vanishes to infinite order at the origin.
	Suppose this is not the case, and let $p_0 \geq N$ be the smallest integer such that $w^{(p_0)}(0) \neq (0, \dots, 0).$
	It follows that there exists a vector $c  \neq (0, \dots, 0)$ such that
  $$w(r) = c r^{p_0} + O(r^{p_0+1}) \quad \text{as } \,\, r \to 0^+.$$
  Since $w$ is $C^{\infty}$, as proved in the Appendix to \cite{PWY}, we may formally differentiate this Taylor expansion of each component $w_i$   consecutively  to obtain
  \allowdisplaybreaks
  \begin{align*}
  	w'(r) & = c p_0 r^{p_0-1} + O(r^{p_0}); \\
  	\vdots & \\
  	w^{(N)} (r) & = c\prod_{j=1}^N \left (p_0-(j-1) \right )  \cdot r^{p_0-N} + O( r ^{p_0-N+1}).
  \end{align*}

  \noindent
   Substituting these expansions into \eqref{godewe}, a direct computation yields
 \begin{equation}\label{fti}
     \begin{split}
 	&	\ |c|\prod_{j=1}^N |p_0-(j-1)+b_j|  \cdot r^{p_0-N} + O( r ^{p_0-N+1})   \\
 \quad  \leq&\  C|c| \sum_{j=0}^{N-1} \prod_{k=1}^j \left |p_0-(k-1) \right|\cdot r ^{p_0-N} + O( r ^{p_0-N+1})  \\
 \quad \leq & \ C|c|N \prod_{k=1}^{N-1} \left |p_0-(k-1) \right|\cdot r ^{p_0-N} + O( r ^{p_0-N+1})
     \end{split}
 \end{equation}
as $  r \to 0^+.$
In view of \eqref{pne}, we can  divide both sides of \eqref{fti} by $ |c|\prod_{j=1}^N |p_0-(j-1)+b_j| \cdot r^{p_0-N}$. Applying the definition of $M_b$ in \eqref{Mb}, we get
\begin{equation*}
		  1 < CNM_b + O( r )\quad \text{as} \,\,  r \to 0^+ .
\end{equation*}
 Taking the limit as $r \to 0^+$, we obtain $1 \leq CNM_b$. This contradicts the assumption \eqref{cl1}, thereby proving that $w^{(p)}(0)=0$ for all $p \geq 0$.

\end{proof}

Theorem \ref{ee} is now an immediate consequence of Theorem \ref{ees}.

\begin{proof}[\textbf{Proof of Theorem \ref{ee}}: ]
Since $h\in C([0, r_0))$ and  $  h(0) =0$, there exists some positive   $r_1<r_0$  such that
$$   h(r)< \frac{1}{NM_b}, \ \
 \quad r \in (0, r_1).$$
Thus $w$ satisfies \eqref{godewe} on $(0, r_1)$. Applying Theorem \ref{ees} yields $w\equiv 0$ on $[0, r_1)$.   When $r\ge r_1$, since the weights $r^{-(N-j)}$ are bounded on the interval $[r_1, r_0)$, Corollary 1.3 in \cite{PW2} further ensures that $w \equiv 0$ on the entire  domain $[0, r_0)$.

\end{proof}

It is worth noting that condition \eqref{bj} is quite mild, requiring only that each parameter $b_j$ avoids a certain discrete set of the real axis. As the following example shows, this condition  is also necessary:    the conclusion of Theorem \ref{ee} fails whenever  \eqref{bj} is violated.

\begin{example}
	\label{ex1}
	\upshape	Let $(b_1, \dots, b_N) \in \mathbb{C}^N$ be such that $j_0 -  b_{j_0}  - N \in \mathbb{Z}^+$ for some $j_0 \in \{ 1, \dots, N \}$. Let $w(r) = r^l$, where $$l: = j_0 - b_{j_0} - 1.$$ Then the assumption implies that   $  l (\in \mathbb Z^+) \geq N$. Therefore  $w \in C^{\infty}([0,1])$ and
		$$w^{(j)}(0) = 0, \quad   j = 0, \dots, N-1.$$ Moreover, by a direct computation, we have
	\begin{equation*}
		\mathcal{D}_{b_{j_0}} \dots \mathcal{D}_{b_1} w = (l + b_1)  \dots \big( l - (j_0-1) + b_{j_0} \big) r^{l-j_0} = 0.
		\end{equation*}
		  Thus inequality \eqref{gode} is satisfied with $h(r)\equiv 0$.
        	This example demonstrates that condition \eqref{bj} is necessary.

	\end{example}
	
	On the other hand, Example \ref{ex2} below demonstrates that the conclusion in Theorem \ref{ee} may fail if $w$ has only finite regularity, indicating that the smoothness assumption on solutions is essential. 
    
\begin{example}
	\label{ex2}
 \upshape   Consider the simple case $N=1$ and $b_1=-\frac{3}{2}$, and the function  $w(r)=r^{\frac{3}{2}}$. Clearly,    $w(0) = 0$ and satisfies the inequality \eqref{gode} with $h(r)\equiv 0$ as
\[
\mathcal D_{b_1} w(r)
= w'(r) - \frac{\frac{3}{2}}{r} w(r)
= 0.
\]
Furthermore, condition \eqref{bj} is satisfied since
\[
j - b_j - N
= 1 - \left(-\frac{3}{2}\right) - 1
= \frac{3}{2} \notin \mathbb{Z}^+.
\]
Noting $w \in C^1([0,1)) \setminus C^{\infty}([0,1))$, this case shows that Theorem \ref{ee} fails without the smoothness assumption on $w$.

 \end{example}

  Theorem \ref{ee} can be used to study uniqueness of solutions to certain  singular linear differential equations with nonconstant coefficients. In detail, given $(a_1, \ldots, a_N)\in \mathbb C^N$, we first define the set
$P_{a_1, \ldots, a_N}$ as
\begin{equation}\label{P}
  P_{a_1, \ldots, a_N}:  = \left\{(b_1, \ldots, b_N)\in \mathbb C^N:   w^{(N)} + \frac{a_1}{r}w^{(N-1)} + \cdots + \frac{a_N}{r^N} w = \mathcal{D}_{b_N} \cdots \mathcal{D}_{b_1} w    \right\} .
\end{equation}
According to the proof of Lemma \ref{ele2}, the set $ P_{a_1, \ldots, a_N}$ has finite cardinality. Let
\begin{equation}\label{B_1}
    I_1: = \{(b_1, \ldots, b_N)\in \mathbb C^N: j -  b_j  - N \notin \mathbb{Z}^+, j = 1, \ldots, N\}.
   \end{equation}
 That is, $I_1$ is the collection of all parameters such that \eqref{bj} in Theorem \ref{ee} is satisfied.

\begin{cor}
	\label{eenl}
    Let $r_0>0$ and $N \in \mathbb Z^+$.  Suppose that   $w = (w_1, \dots, w_{\nu}): [0, r_0) \to \mathbb{C}^{\nu}$ is smooth on $[0, r_0) $ and       satisfies the differential system
\begin{equation*}
        	w^{(N)} + \frac{h_1(r)}{r}w^{(N-1)} + \cdots + \frac{h_N(r)}{r^N} w =0, \quad r \in (0, r_0) 
\end{equation*}
    for some functions $h_j\in C([0, r_0)), j =1, \ldots, N$. Let $P_{h_1(0), \ldots, h_N(0)}$ and $ I_1$ be defined as in \eqref{P} and \eqref{B_1},  respectively. If
    $$ P_{h_1(0), \ldots, h_N(0)} \cap I_1\ne \emptyset$$
    and
     $$w^{(j)}(0) = 0, \ \  j = 0, \dots, N-1,$$
   then $ w \equiv 0 $ on $[0, r_0)$.
\end{cor}

\begin{proof}
  We first observe that $w$ satisfies
    $$  	\left|w^{(N)} + \frac{h_1(0)}{r}w^{(N-1)} + \cdots + \frac{h_N(0)}{r^N} w \right|\le \sum_{j=0}^{N-1}|h_{N-j}(r)-h_{N-j}(0)|\cdot \frac{| w^{(j)}(r) |}{r^{N-j}}, \quad r \in (0, r_0). $$
    Since $h_j\in C([0, r_0)), j =1, \ldots, N$, there exists some function $h\in  C([0, r_0))$ with $h(0)=0$, such that for all $j =0, \ldots, N-1$,
    $$ |h_{N-j}(r)-h_{N-j}(0)|\le h(r), \quad r \in (0, r_0).$$
    On the other hand, let $(b_1, \ldots, b_N)\in P_{h_1(0), \ldots, h_N(0)} \cap B_1$. Then $(b_1, \ldots, b_N) $ satisfies \eqref{bj}, and
    $$ |\mathcal{D}_{b_N} \cdots \mathcal{D}_{b_1} w(r) |=\left| w^{(N)} + \frac{h_1(0)}{r}w^{(N-1)} + \cdots + \frac{h_N(0)}{r^N} w\right|\le h(r)\sum_{j=0}^{N-1}   \frac{| w^{(j)}(r) |}{r^{N-j}}.$$
      By Theorem \ref{ee}, $w\equiv 0$.
      
\end{proof}

\vspace{.2in}

\section{Uniqueness of non-smooth solutions to Euler-type ODEs}
\label{section:gron_proof}

Despite the   counterexamples in the previous section, it is natural to ask whether uniqueness for \eqref{gode} still holds under limited regularity,  specifically when $w \in C^N$, since the inequality and the vanishing condition at the origin   involve only derivatives up to order $N$. In this section, we show  that  uniqueness can still be expected, provided that a stronger condition \eqref{bjN} below is imposed on the coefficients $(b_1, \dots, b_N)$.

\begin{thm}
	\label{gron}
       Let $r_0>0, N\in \mathbb Z^+$ and $(b_1, \ldots, b_N)\in \mathbb C^N$. Suppose $w = (w_1, \dots, w_{\nu}): [0, r_0) \to \mathbb{C}^{\nu}$ is a vector-valued $C^N$ function on $[0, r_0) $ that satisfies the differential inequality
       \begin{equation}\label{gode1}
     	   | \mathcal{D}_{b_N} \cdots \mathcal{D}_{b_1} w(r) | \leq h(r)\sum_{j=0}^{N-1} \frac{| w^{(j)}(r) |}{r^{N-j}}, \quad r \in (0, r_0)
\end{equation}
     for some nonnegative function $h\in C([0, r_0))$ with  $  h(0) =0$.  If
\begin{equation}
	\label{bjN}
   j - \mathrm{Re}(b_j) - N<1, \quad j =1, \ldots, N,
\end{equation}
and $$w^{(j)}(0)=0,\ \  j=0, \ldots,  N-1,$$ then $ w \equiv 0 $ on $[0,r_0)$.

\end{thm}

As will be shown below, the stronger condition  \eqref{bjN}  on the parameters $(b_1, \ldots, b_N)$  enables an estimate that controls all lower-order derivatives of $w$ in terms of the quantity $\mathcal{D}_{b_N} \cdots \mathcal{D}_{b_1} w$. When combined with the differential inequality, this estimate forces the solution to vanish identically. As in Theorem \ref{ee}, we   establish  uniqueness by proving it under a weaker differential inequality. 

\begin{thm}
	\label{gronw}
       Let $r_0>0$ and $ N\in \mathbb Z^+$, and let $(b_1, \ldots, b_N)\in \mathbb C^N$ satisfy \eqref{bjN}. There exists a constant $C>0$ dependent only on $(b_1, \ldots, b_N)$, such that   every $C^N$ function $w = (w_1, \dots, w_{\nu}): [0, r_0) \to \mathbb{C}^{\nu}$    satisfying
     \begin{equation}
     	\label{gode_finite}
    | \mathcal D_{b_N} \cdots\mathcal D_{b_1} w(r)|\le  C\sum_{j=0}^{N-1}\frac{|w^{(j)}(r)|}{r^{N  -j}},
   \ \ r\in   (0, r_0)
\end{equation}
and   $$w^{(j)}(0)=0, \ \ j=0, \ldots,  N-1$$ must necessarily vanish identically.

\end{thm}

\begin{proof} 	We begin by considering the case $N=1$, where $w$ satisfies the differential inequality
\begin{equation*}
	\left| \mathcal{D}_{b_1}w(r) \right| \le C \frac{|w(r)|}{r}.
\end{equation*}
Let $g := \mathcal{D}_{b_1} w$. By definition, we have
\begin{equation}
	\label{g1}
	g(r)  = r^{-b_1} \frac{d}{dr} \left( r^{b_1} w(r) \right) = w'(r)+ \frac{b_1}{r} w(r).
	\end{equation}
In particular, $g\in C((0, r_0))$ and the original inequality becomes
	\begin{equation}
		\label{n1}
		|g(r)| \le \frac{C|w(r)|}{r}.
		\end{equation}
		
		\noindent
		Given that $w(0)=0$ and $w \in C^1([0, r_0))$, applying L'H\^opital's rule and setting $ g(0) = (1+b_1)w'(0)$   extends      $g$ continuously to the origin.  Consequently, the maximum function
		\begin{equation*}
			M(r) := \max_{s \in [0, r]} |g(s)|
			\end{equation*}
			is well-defined and finite, and monotone non-decreasing for all $r \in [0, r_0)$.

            The condition $\mathrm{Re}(b_1) > -1$ ensures that $$\lim_{r \to 0^+} r^{b_1} w(r) = 0.$$ In fact this follows immediately from the continuity of $w$ if $\mathrm{Re}(b_1) \ge 0$, or from L'H\^opital's rule  if $-1 < \mathrm{Re}(b_1) < 0$. Multiplying both sides of \eqref{g1} by $r^{b_1}$ and integrating it from $0$ to $r$ then yields
			\begin{equation*}
				w(r) = r^{-b_1} \int_0^r s^{b_1} g(s) \, ds.
				\end{equation*}
				
	\noindent
	Using the monotonicity of $M(r)$ and the fact that $\mathrm{Re}(b_1) > -1$, one gets
	\begin{align}
	\label{w11}
		|w(r)| &\le r^{-\mathrm{Re}(b_1)} \int_0^r s^{\mathrm{Re}(b_1)} |g(s)| \, ds     \le M(r) r^{-\mathrm{Re}(b_1)} \int_0^r s^{\mathrm{Re}(b_1)} \, ds   = \frac{r M(r)}{\mathrm{Re}(b_1) + 1}.
	\end{align}
					
		\noindent
		Substituting this estimate back into \eqref{n1}, we obtain
			\begin{equation*}
				|g(r)| \le \frac{C M(r)}{ \mathrm{Re}(b_1) + 1 }  , \quad r \in [0, r_0).
		\end{equation*}
	Now, for each $r_1<r_0$, take the maximum on both sides over $[0, r_1]$ and make use of the monotonicity of $M(r)$ to get
	\begin{equation*}
		M(r_1) \le \frac{C M(r_1)}{ \mathrm{Re}(b_1) + 1 }.
	\end{equation*}
If \begin{equation}\label{C1}
    C<\mathrm{Re}(b_1) + 1,
\end{equation} this would be impossible unless $M(r_1)\equiv 0$. It follows from \eqref{w11} and the arbitrariness of $r_1<r_0$ that $w \equiv 0$ on $[0, r_0)$.

\vspace{.05in}
In the case $N=2$, the inequality is given by
\begin{equation*}
\left|\mathcal D_{b_2} \mathcal D_{b_1} w(r)\right|\le C\left(\frac{|w(r)|}{r^{2}}+\frac{|w'(r)|}{r}\right).
\end{equation*}

\vspace{.05in}
\noindent
Let $g: =  \mathcal D_{b_2} \mathcal D_{b_1} w$. The inequality then becomes
\begin{equation}
		\label{eqn:m=2_first_integration}
	\left | g(r) \right | \leq C\left(\frac{|w(r)|}{r^{2}}+\frac{|w'(r)|}{r}\right).
\end{equation}
 By the definition of the operators, $g(r)$ can be expressed as
\begin{equation}
	\label{g2}
	g(r) = r^{-b_2}\frac{d}{dr}\left ( r^{b_2-b_1} \frac{d}{dr}\left(r^{b_1} w(r)  \right)   \right ),\ \  r\in (0, r_0).
\end{equation}
Given $w(0) = w'(0) = 0$ and $w \in C^2([0, r_0))$, $g$ admits a continuous extension to the origin, so $g \in C([0, r_0))$.

Under the conditions $\mathrm{Re}(b_1) > -2$ and $\mathrm{Re}(b_2) > -1$,  it follows that
$$ \lim_{r\to 0^+} r^{b_1} w(r)=0 $$
and
$$ \lim_{r\to 0^+} r^{b_2-b_1} \frac{d}{dr}\left(r^{b_1} w(r)  \right) =\lim_{r\to 0^+} r^{b_2}\left( b_1\frac{w(r)}{r}+w'(r)\right)= 0. $$

\noindent
Thus, we can multiply both sides of \eqref{g2} by $r^{b_2}$ and perform successive integration from $0$ to $r$ to obtain
\begin{align*}
	w(r) &= r^{-b_1} \int_0^r s_2^{b_1-b_2} \int_0^{s_2} s_1^{b_2} g(s_1) ds_1 ds_2.
\end{align*}
Differentiating  the above further gives
\begin{equation*}
    \begin{split}
        w'(r) = -b_1 r^{-b_1-1}\int_0^rs_2^{b_1-b_2}\int_0^{s_2} s_1^{b_2} g(s_1) \,ds_1 \, ds_2 + r^{-b_2}\int_0^r s_1^{b_2} g(s_1) \,ds_1.
    \end{split}
\end{equation*}

Let $$M(r) := \max_{s \in [0, r]} |g(s)|.$$ Using the monotonicity of $M(r)$, we estimate $|w(r)|$:
\begin{align}
	\label{wb}
        |w(r)| & \le M(r) r^{-{\mathrm{Re}(b_1)}}\int_0^rs_2^{{\mathrm{Re}(b_1)}-{\mathrm{Re}(b_2)}}\int_0^{s_2} s_1^{{\mathrm{Re}(b_2)}}  \,ds_1 \, ds_2  \nonumber \\
        &   = \frac{ r^2M(r)}{\left ({\mathrm{Re}(b_2)}+1 \right )({\mathrm{Re}(b_1)}+2)}.
    \end{align}
   Similarly, for the derivative $|w'(r)|$:
    \begin{align*}
        |w'(r)| & \le   \frac{|b_1|M(r)}{{\mathrm{Re}(b_2)}+1}r^{-{\mathrm{Re}(b_1)}-1}\int_0^r s_2^{\mathrm{Re}(b_1)+1} \, ds_2  + \frac{rM(r)}{ {\mathrm{Re}(b_2)}+1 } \nonumber \\
        & = \left(\frac{|b_1|}{({\mathrm{Re}(b_2)}+1)({\mathrm{Re}(b_1)}+2)}  + \frac{1}{ {\mathrm{Re}(b_2)}+1 }\right)rM(r).
\end{align*}

\noindent
Substituting these estimates into \eqref{eqn:m=2_first_integration}, the inequality reduces to:
\begin{align*}
        |g(r)| \le C \left(\frac{|b_1|+3 + {\mathrm{Re}(b_1)}}{({\mathrm{Re}(b_2)}+1)({\mathrm{Re}(b_1)}+2)}    \right) M(r), \quad \ r\in(0, r_0).
\end{align*}
As before, for each $r_1<r_0$, taking the maximum over $[0, r_1]$ yields
$$ M(r_1)\le  C \left(\frac{|b_1|+3 + {\mathrm{Re}(b_1)}}{({\mathrm{Re}(b_2)}+1)({\mathrm{Re}(b_1)}+2)}    \right) M(r_1). $$
If \begin{equation}\label{C2}
    C< \frac{({\mathrm{Re}(b_2)}+1)({\mathrm{Re}(b_1)}+2)} {|b_1|+3 + {\mathrm{Re}(b_1)}} ,
\end{equation} then $M(r_1) = 0$, which in turn forces $g \equiv 0$  on $[0, r_1)$.  Consequently, by \eqref{wb} we have $w \equiv 0$ on $[0, r_1]$.

\vspace{.05in}
For the general case $N \geq 3$, define
\begin{align}
	\label{eqn:gN}
	g(r) &:= \mathcal{D}_{b_N} \cdots \mathcal{D}_{b_1} w(r) \nonumber \\
	&=  r^{-b_N} \frac{d}{dr} \left ( r^{b_N-b_{N-1}} \frac{d}{dr} \left(  r^{b_{N-1}-b_{N-2}}   \frac{d}{dr} \left (\cdots r^{b_2-b_1} \frac{d}{dr}\left(r^{b_1} w(r) \right)   \right )\cdots \right)  \right).
\end{align}

\noindent
 The inequality is given by
\begin{equation}
	\label{wn}
         \left| g(r)  \right|
\le  C\sum_{j=0}^{N-1}\frac{|w^{(j)}(r)|}{r^{N  -j}}.
\end{equation}

\noindent
Given the vanishing conditions $w^{(k)}(0) = 0$ for $k = 0, \dots, N-1$ and the constraints \eqref{bjN} on $b_j$, a recursive application of the L'H\^opital's rule ensures that the following limits vanish at the origin:
$$ \lim_{r \to 0^+} r^{b_{j}-b_{j-1}} \frac{d}{dr} \left (      r^{b_{j-1}-b_{j-2}}\frac{d}{dr} \left (\cdots  r^{b_2-b_1} \frac{d}{dr}\left(r^{b_1} w(r)   \right)   \right )\cdots \right ) =0 \quad \text{for} \,\, j=1, \dots, N. $$
Thus performing $N$ successive integration on \eqref{eqn:gN} yields the integral representation of $w(r)$:
\begin{align}\label{we}
   		w(r)
		 =   r^{-b_1} \int_0^r   s_N^{b_{1}-b_{2}}	 \int_0^{s_{N}} s_{N-1}^{b_{2}-b_{3}} \cdots \int_0^{s_3}s_2^{b_{N-1}-b_N}\int_0^{s_2} s_1^{b_N}g(s_1)\, ds_1 \,ds_2 \cdots ds_{N}.
\end{align}
Differentiating this expression inductively, the $j$-th order derivative   satisfies
\begin{equation}\label{wj}
    \begin{split}
         w^{(j)}(r) = &r^{-b_{j+1}}
\int_0^{r} s_{N-j}^{\,b_{j+1}-b_{j+2}}
\int_0^{s_{N-j}} s_{N-j-1}^{\,b_{j+2}-b_{j+3}}
\cdots
\int_0^{s_2} s_1^{b_N} g(s_1)\, ds_1\, ds_2 \cdots ds_{N-j} \\
&+ \sum_{m=0}^{j-1}
\frac{C_{j,m}}{r^{\,j-m}}\, w^{(m)}(r),\ \ j=1, \ldots, N-1,
    \end{split}
\end{equation}
  where the constants $C_{j,m} $ are  dependent only on  $(b_1, \ldots, b_N)$.

  Let $$M(r) := \max_{s \in [0, r]} |g(s)|.$$
Utilizing the monotonicity of $M(r)$, we obtain from \eqref{we} that
\allowdisplaybreaks
\begin{align}
	\label{es}
            |w(r)| \leq & M(r)  r^{-\mathrm{Re}(b_1)} \int_0^r   s_N^{\mathrm{Re}(b_{1})-\mathrm{Re}(b_{2})}	 \cdots \int_0^{s_3}s_2^{\mathrm{Re}(b_{N-1})-\mathrm{Re}(b_N)}\int_0^{s_2} s_1^{\mathrm{Re}(b_N)}\, ds_1 \,ds_2 \cdots ds_{N} \nonumber \\
             = &\frac{M(r)}{\mathrm{Re}(b_N)+1}  r^{-\mathrm{Re}(b_1)} \int_0^r   s_N^{-\mathrm{Re}(b_{1})-\mathrm{Re}(b_{2})}	 \cdots \int_0^{s_3}s_2^{\mathrm{Re}(b_{N-1})-\mathrm{Re}(b_N)+\mathrm{Re}(b_N)+1} \,ds_2 \cdots ds_{N} \nonumber \\
            &  \vdots \nonumber  \\
             =& \frac{r^NM(r)}{\prod_{j=1}^N\left(\mathrm{Re}(b_j)+N-j+1\right)}.
    \end{align}

\vspace{.05in}
    \noindent
 Similarly, through induction on $j$, it follows that there exist constants $C_j > 0$ dependent only on $(b_1, \ldots, b_N)$ such that
    \begin{equation}
    	\label{wjn}
        |w^{(j)}(r)|\le C_j r^{N-j}M(r), \quad j=0, \ldots, N-1 .
    \end{equation}

\noindent
Substituting the estimates  \eqref{wjn} into the right-hand side of \eqref{wn}, the inequality reduces to
$$ |g(r)|\leq C \left (\sum_{j=0}^{N-1}C_j \right )  M(r).$$
Finally, taking the maximum over $[0, r_1]$ for each $r_1<r_0$, we find
\begin{equation*}
	M(r_1) \leq C \left (\sum_{j=0}^{N-1}C_j \right )   M(r_1).
\end{equation*}
If \begin{equation}\label{CN}
    C <\frac{1}{ \sum_{j=0}^{N-1}C_j    },
\end{equation} then
  $M(r_1) = 0$ necessarily. Consequently, $w \equiv 0$ on the interval $[0, r_1]$. The proof is complete.

\end{proof}

\noindent
\textbf{Remark}:
  As demonstrated in the proof, the upper bound for the constant $C$ in Theorem \ref{gronw} can be explicitly determined in terms of the parameters $b_j$,  $j=1, \ldots, N$. The specific value for $C$ is given in \eqref{C1} for $N=1$, \eqref{C2} for $N=2$, and \eqref{CN} for the general cases $N \ge 3$.

\vspace{.1in}

\begin{proof}[\textbf{Proof of Theorem \ref{gron}}: ] The proof follows from Theorem \ref{gronw} by an argument similar to that used in Theorem \ref{ee}, and is therefore omitted.

\end{proof}

Using a construction similar to that in Example \ref{ex1}, the following example shows that condition \eqref{bjN} is necessary for Theorem \ref{gron} to hold.

\begin{example}
	\label{ex6}
	\upshape
Let $(b_1, \dots, b_N) \in \mathbb{C}^N$ be such that $$j_0 - \mathrm{Re}(b_{j_0}) - N =1 +\alpha$$  for some $j_0 \in \{ 1, \dots, N \}$ and   $\alpha \geq 0$. Let $w(r) = r^l$, where $$l: = j_0 - b_{j_0} - 1.$$ Then our assumption implies that  $ \mathrm{Re}(l) = j_0 - \mathrm{Re}(b_{j_0}) - 1 = N + \alpha\ge N.$  Therefore $w \in C^{N}([0,1])$ and   $$w^{(j)}(0) = 0, \quad  j = 0, \dots, N-1.$$ Moreover, by a direct computation,
\begin{equation*}
	\mathcal{D}_{b_{j_0}} \dots \mathcal{D}_{b_1} w(r) = (l + b_1)  \dots \big( l - (j_0-1) + b_{j_0} \big) r^{l-j_0} = 0.
\end{equation*}
 Thus, inequality \eqref{gode1} is satisfied with $h(r)\equiv 0$. This example shows that condition \eqref{bjN} is necessary.
\end{example}

\begin{example}
	\label{remark:Re_bj_geq_0}
\upshape	One readily checks that assumption \eqref{bjN} is satisfied provided
	$\mathrm{Re}(b_j) >-1$ for all $j=1,\dots,N$.
 In particular, this holds for the radial component of the classic Laplace operator $\Delta$ in $\mathbb{R}^n$. Indeed, if $w$ is a radial function, $\Delta w$ reduces to the following form in the radial variable $r$:
 $$w'' +\frac{n-1}{r} w' = \mathcal D_{n-1}\mathcal D_{0} w, $$where $b_1 = 0$ and $b_2 = n-1$. Similarly, if $w$ is radial, the  operator $\Delta^m$ can be written as
 $$\Delta ^m w = \underbrace{
 	(\mathcal D_{n-1}\mathcal D_0)  \cdots
 	(\mathcal D_{n-1}\mathcal D_0)
 }_{m\, \text{times}}w = \left(\mathcal D_{n-1}\mathcal D_0\right)^m w, $$
so condition \eqref{bjN} is also satisfied.

\end{example}

Analogous to Corollary \ref{eenl}, the following uniqueness result holds for $C^N$ solutions  to a class of linear singular systems with nonconstant coefficients. The proof is similar and is therefore omitted. For the statement of this result, we define
\begin{equation}\label{B_2}
   I_2: =  \{(b_1, \ldots, b_N)\in \mathbb C^N: j -  \mathrm{Re} (b_j)  - N <1, j = 1, \ldots, N\}.
\end{equation}

\begin{cor}\label{eenl2}
    Let $r_0>0$ and $N\in \mathbb Z^+$.  Suppose that   $w = (w_1, \dots, w_{\nu}): [0, r_0) \to \mathbb{C}^{\nu}$ is $C^N$ on $[0, r_0)$ and       satisfies the differential system
\begin{equation*}
     	    	w^{(N)} + \frac{h_1(r)}{r}w^{(N-1)} + \cdots + \frac{h_N(r)}{r^N} w =0, \quad r \in (0, r_0)
\end{equation*}
    for some functions $h_j\in C([0, r_0)), j =1, \ldots, N$. Let $P_{h_1(0), \ldots, h_N(0)}$ and $ I_2$ be as defined in \eqref{P} and \eqref{B_2}, respectively. If
    $$ P_{h_1(0), \ldots, h_N(0)} \cap I_2\ne \emptyset$$
    and
     $$w^{(j)}(0) = 0, \quad   j = 0, \dots, N-1,$$
   then $ w \equiv 0 $ on $[0, r_0)$.
\end{cor}

\vspace{.2in}

\section{Existence of solutions to Euler-type ODEs}
\label{section:iodea_proof}
In this section, we establish the existence  of $C^N$ solutions to the following initial value problem for a singular nonlinear system involving the operator  $\mathcal{D}_{b_N} \cdots \mathcal{D}_{b_1}$.

\begin{thm}
	\label{iodea}
	Let $r_0 > 0$, $N\in \mathbb Z^+$, and let $(b_1, \ldots, b_N) \in \mathbb{C}^N$ satisfy \eqref{bjN}. Define the set of indices $\Lambda$ as
	\begin{equation}
		\label{Lam}
		\Lambda = \{ -b_{j} + j - 1 : j = 1, \ldots, N \} \cap A,
	\end{equation}
	where $A := \{0, \ldots, N-1\}$. Given $\zeta_j \in \mathbb{C}^\nu$ for each $j \in \Lambda$, there exists a solution $u =(u_1, \ldots, u_{\nu}) \in C^N([0, r_0), \mathbb{C}^\nu)$ to the nonlinear differential system
	\begin{equation}\label{indicial}
		\left\{ \begin{aligned}
			& \mathcal D_{b_N} \cdots\mathcal D_{b_1} u(r) =  f(r, u, \ldots, u^{(N-1)}), \quad r\in (0, r_0);\\
			&u^{(j)}(0)= \zeta_j, \quad j\in  \Lambda;\\
			&u^{(j)}(0)= 0, \quad j\in A\setminus \Lambda,
		\end{aligned}\right.
	\end{equation}
	where the vector-valued function $f:= (f_1, \ldots, f_{\nu})$ is continuous with respect to all its arguments. If in addition $f$ is locally Lipschitz continuous in $(u, \ldots, u^{(N-1)})$, then the solution is unique.
\end{thm}

 Due to the singular nature of the operator $ \mathcal D_{b_N} \cdots\mathcal D_{b_1}$ at the origin, one might typically expect homogeneous initial conditions at $r=0$ to be necessary to compensate for the singular behavior. A notable aspect of our existence result Theorem \ref{iodea} is that the initial jets may be freely prescribed whenever the corresponding index belongs to $\Lambda$.   As demonstrated in the lemma below,   the indices in $\Lambda$ are precisely the integer indicial roots of the operator $\mathcal{D}_{b_N} \cdots \mathcal{D}_{b_1}$ that lie in the range $\{0, \dots, N-1\}$.

\begin{lemma}
	\label{el}
    For each $\lambda\in A$, the power function $r^\lambda$ satisfies the homogeneous Euler initial value problem
    \begin{equation*}
    	\left\{
    	\begin{aligned}
          & \mathcal D_{b_N} \cdots\mathcal D_{b_1} u   = 0,    \quad  r\in   (0, r_0);\\
   &u^{(\lambda)}(0)  = \lambda !;   \\
    &u^{(j)}(0)  = 0, \quad j\in A\setminus \{\lambda\}
            \end{aligned}
            \right.
            \end{equation*}
if and only if $\lambda\in \Lambda$.
\end{lemma}

\begin{proof}
	For each $\lambda\in A$, $r^\lambda$ satisfies the initial conditions.  To determine if $r^\lambda$ is a solution, we substitute $u = r^\lambda$ into the homogeneous equation. A straightforward computation yields
	$$\mathcal{D}_{b_N} \cdots \mathcal{D}_{b_1} (r^\lambda) = \left[ \prod_{j=1}^N (\lambda - (j-1) + b_j) \right] r^{\lambda-N}.$$

\noindent
Consequently, $r^\lambda$ satisfies the homogeneous equation if and only if
$$\prod_{j=1}^N (\lambda - (j-1) + b_j) = 0,$$
i.e., if and only if $\lambda \in \{j - 1 - b_j : j = 1, \ldots, N\}$.  This completes the proof.

\end{proof}

 While the initial jets corresponding to indices in $\Lambda$ may be prescribed arbitrarily, as stated in Theorem \ref{iodea},  the remaining initial jets must vanish to ensure the existence of a solution, as explained below.

\begin{remark}
	\label{re}
	\upshape
	   The jets corresponding to indices not in $\Lambda$ must vanish in order to ensure compatibility and guarantee  the existence of a $C^N$ solution to \eqref{indicial}. Indeed, let $j_0$ be the smallest index in $A\setminus \Lambda$ such that the $j_0$-th jet of the solution in the initial condition of \eqref{indicial} is nonzero. Then  for any $C^N$ solution $u$ to the initial value problem,  according to the proof of Lemma \ref{el}  there exists a constant $c>0$ such that  
    $$  |\mathcal D_{b_N} \cdots\mathcal D_{b_1} u| = c r^{j_0-N} +o(r^{j_0-N} ), \ \ r<<1,$$
    which diverges to infinity as $r\rightarrow 0^+$ since $j_0\le N-1$.
     On the other hand,  the right-hand side $ | f(r, u, \ldots, u^{(N-1)}) |$ remains bounded near $0$, which leads to a contradiction. This argument does not apply to indices in $ \Lambda$, since any potential singular contribution arising from      jets of such order is annihilated by  $\mathcal D_{b_N} \cdots\mathcal D_{b_1}$.
\end{remark}

To prove Theorem \ref{iodea}, we first study the existence of $C^N([0, r_0))$ solutions to the  system with homogeneous initial conditions:
 \begin{equation}
 	\label{hom}
 	\left\{ \begin{aligned}
		&            \mathcal D_{b_N} \cdots\mathcal D_{b_1} u(r) =  f(r, u, \ldots, u^{(N-1)}),
		\quad  r\in (0, r_0)\\
		&u(0)= u'(0)= \ldots =u^{(N-1)}(0) =0.
	\end{aligned}\right.
\end{equation}

\vspace{.05in}

\begin{thm}
	\label{cont}
	Let $r_0 > 0$, $N\in \mathbb Z^+$, and let $(b_1, \dots, b_N) \in \mathbb{C}^N$ satisfy condition \eqref{bjN}. Suppose that  $f$ is a vector-valued function that is continuous with respect to all its variables. Then there exists a $C^N([0, r_0))$ solution to the nonlinear differential system \eqref{hom}.\end{thm}
	
	The proof of this existence result under the continuity assumption  on $f$ essentially follows the Carath\'eodory approach, combined with the  lemma below to handle the singularity at the origin. Recall that for $k\in \mathbb Z^+\cup\{0\}$, the $C^k$ norm of a function $u\in C^k([a, b])$ is defined by
     $$
  \|u\|_{C^{k}([a, b])} = \max_{0\le j\le k}\sup_{r\in [a, b]}|u^{(j)}(r)|.
$$

\begin{lemma}
	\label{tb}
Let $r_0 > 0$, $N\in \mathbb Z^+$, and   $(b_1, \dots, b_N) \in \mathbb{C}^N$ satisfy condition \eqref{bjN}. For any $g \in C([0, r_0))$, define the operator $\mathcal{T}$ by
\begin{equation*}
	(\mathcal{T}g)(r) :=
	\begin{cases}
	r^{-b_1} \int_0^r   s_N^{b_{1}-b_{2}}	 \int_0^{s_{N}} s_{N-1}^{b_{2}-b_{3}} \cdots \int_0^{s_3}s_2^{b_{N-1}-b_N}\int_0^{s_2} s_1^{b_N}g(s_1)\, ds_1 \cdots ds_N, &  r \in (0, r_0); \\
		0, &  r = 0.
	\end{cases}
\end{equation*}

\noindent
Then $\mathcal{T}g \in C^N([0, r_0))$ and satisfies the system
 \begin{equation}
 	\label{hom1}
 	\left\{
 	\begin{aligned}
   &            \mathcal D_{b_N} \cdots\mathcal D_{b_1} (\mathcal T g) =  g,
   \quad r\in   (0, r_0)\\
    &(\mathcal T g) (0)= \ldots =(\mathcal T g)^{(N-1)}(0) =0.
        \end{aligned}
        \right.
        \end{equation}

        \noindent
Furthermore, for any $r_1 < r_0$, the operator $\mathcal{T}$ is bounded from $C([0, r_1])$ into $C^N([0, r_1])$. In particular,  there exists a constant $C > 0$ independent of $r_1$ such that
\begin{equation}
	\label{Tb1}
    \|\mathcal T g\|_{C^{N-1}([0,r_1])}  \le Cr_1\, \|g\|_{C([0,r_1])}
\end{equation}
and
\begin{equation}\label{Tb2}
    \|(\mathcal T g)^{(N)}\|_{C([0,r_1])}   \le C  \| g \|_{C([0,r_1])}.
\end{equation}
\end{lemma}

\vspace{.05in}

\begin{proof}
As seen in the proof of Theorem \ref{gron}, the admissibility condition \eqref{bjN} ensures that the iterated integral in the definition of $\mathcal{T}$ is well-defined. Furthermore, $\mathcal{T}g \in C^N((0, r_1))$ and satisfies the system \eqref{hom1}. The bound \eqref{Tb1} follows directly from \eqref{wjn}. Consequently, it remains only to show that $(\mathcal{T}g)^{(N)}$ is continuous at $0$ and satisfies \eqref{Tb2}.

Recall that \eqref{wj} with $j=N-1$ yields
\begin{equation}
	\label{tgn-1}
           (\mathcal T g)^{(N-1)}(r) = r^{-b_{N}}
\int_0^{r} s_1^{b_N} g(s_1)\, ds_1 + \sum_{m=0}^{N-2}
\frac{C_{N-1,m}}{r^{\,N-1-m}}\, (\mathcal T g)^{(m)}(r).
\end{equation}
Before proceeding to the $N$-th derivative of $\mathcal Tg$, we first  claim that
\begin{equation}\label{cla}
    \lim_{r\rightarrow 0^+} \frac{(\mathcal T g)^{(m)}(r)}{{r^{\,N-m}}} = a_m g(0), \ \ m=0, \ldots, N-1,
\end{equation}
where
\begin{equation}\label{am}
     a_m: = \frac{(\mathcal T 1)^{(m)}(r)}{r^{N-m}}   = \frac{N!}{(N-m)!\cdot \prod_{j=1}^N\left( b_j +N-j+1\right)}.
\end{equation}
 Indeed, to see this, we  first estimate    $\mathcal T\left(g-g(0)\right)$ by replacing $g$ in the expression of \eqref{we} by $g-g(0)$. The estimate   \eqref{es} yields
 $$ \frac{|\mathcal T\left(g-g(0)\right)(r)|}{r^N}\le \frac{\max_{s\in [0,r]}|g(s)-g(0)|}{\prod_{j=1}^N\left(\mathrm{Re}(b_j)+N-j+1\right)}.$$
 Since $g$ is continuous at $r=0$, the right-hand side goes to zero as $r\rightarrow 0^+$. Hence
 $$ \lim_{r\rightarrow 0^+}\frac{\mathcal T g(r)}{r^N} =\frac{ \mathcal T(g(0))}{r^N} = g(0)\cdot \frac{\mathcal T1(r) }{r^N}= \frac{g(0)}{\prod_{j=1}^N\left( b_j +N-j+1\right)} =a_0g(0).$$
 A similar argument as above can further show that
$$\frac{|\left(\mathcal T\left(g-g(0)\right)\right)^{(m)}(r)|}{r^{N-m}} \rightarrow 0,   \ \ m=0, \ldots, N-1$$
as $r\rightarrow 0^+$ as well.  Thus
$$  \lim_{r\rightarrow 0^+} \frac{(\mathcal T g)^{(m)}(r)}{{r^{\,N-m}}} = \frac{(\mathcal T (g(0))(r))^{(m)}(r)}{{r^{\,N-m}}} = g(0) \cdot \frac{(\mathcal T 1(r))^{(m)}(r)}{{r^{\,N-m}}} = a_m g(0). $$
The claim \eqref{cla} is proved.

We  now   verify the continuity of the $N$-th derivative of $\mathcal{T}g$ at $0$. Since $ (\mathcal{T}g )^{(N-1)}(0)=0,$ we use the definition of derivative along with \eqref{tgn-1}, L'H\^opital's Rule, and the limit established in \eqref{cla} to obtain
\allowdisplaybreaks
\begin{align}
    \label{22}
           (\mathcal Tg)^{(N)}(0) & = \lim_{r\to 0^+}\frac{(\mathcal Tg)^{(N-1)}(r) - (\mathcal Tg)^{(N-1)}(0)}{r} \nonumber \\
        & = \lim_{r\to 0^+} \left (r^{-b_{N}-1}
\int_0^{r} s_1^{b_N} g(s_1)\, ds_1 + \sum_{m=0}^{N-2}
\frac{C_{N-1,m}}{r^{\,N-m}}\, (\mathcal T g)^{(m)}(r) \right )\nonumber \\
& = \frac{g(0)}{b_N+1} + \sum_{m=0}^{N-2}C_{N-1,m}a_mg(0).
\end{align}

\noindent
On the other hand, for $r\in   (0, r_0)$, differentiating \eqref{tgn-1} with respect to $r$ yields
\allowdisplaybreaks
\begin{align}
	\label{tN}
    (\mathcal T g)^{(N)}(r) & = g(r) -\frac{b_N}{r} \left( r^{-b_N} \int_0^r s_1^{b_N} g(s_1) \, ds_1 \right)  \nonumber \\
   & \quad + \sum_{m=0}^{N-2} \left (\frac{-(N-1-m) C_{N-1,m}}{r^{N-m}} (\mathcal T g)^{(m)}(r) + \frac{C_{N-1,m}}{r^{N-1-m}} (\mathcal T g)^{(m+1)}(r) \right ) \nonumber \\
      & = g(r) -\frac{b_N}{r} \left(  (\mathcal T g)^{(N-1)}(r) - \sum_{m=0}^{N-2} \frac{C_{N-1,m}}{r^{N-1-m}} (\mathcal T g)^{(m)}(r) \right) \nonumber \\
   & \quad + \sum_{m=0}^{N-2} \frac{-(N-1-m) C_{N-1,m}}{r^{N-m}} (\mathcal T g)^{(m)}(r) + \sum_{m=1}^{N-1}\frac{C_{N-1,m-1}}{r^{N-m}} (\mathcal T g)^{(m)}(r)  \nonumber \\
   &  = g(r) +\sum_{m=0}^{N-1}C_{N,m}
\frac{ (\mathcal T g)^{(m)}(r)}{r^{\,N-m}},
\end{align}
where the coefficients    $\{C_{N,m}\}_{0 \leq m \leq N-1}$ are determined by    $\{C_{N-1,m}\}_{0 \leq m \leq N-2}$, via the following  relations:
\begin{align}
	\label{cm}
    \begin{split}
        &C_{N,0}=\big (b_N-(N-1) \big )C_{N-1,0},\\
        &C_{N,m}= \big (b_N-(N-1-m) \big )C_{N-1,m}+C_{N-1,m-1},\quad 1\le m\le N-2,\\
       & C_{N,N-1}= C_{N-1,N-2}-b_N.
    \end{split}
\end{align}

\vspace{.05in}
\noindent
Taking the limit as $r \to 0^+$ in \eqref{tN} and applying the limit established in \eqref{cla}, we find
\begin{align}
	\label{11}
	        \lim_{r\to 0^+}(\mathcal T g)^{(N)}(r) & = g(0) + \sum_{m=0}^{N-1}C_{N,m}\lim_{r\to 0^+}\frac{ (\mathcal T g)^{(m)}(r)}{r^{\,N-m}} \nonumber \\
	& = g(0) + \sum_{m=0}^{N-1}C_{N,m}a_mg(0).
\end{align}

In view of \eqref{11} and \eqref{22}, the continuity of $(\mathcal{T} g)^{(N)}$ at $0$ is therefore equivalent to the identity
$$
\frac{1}{b_N+1}+\sum_{m=0}^{N-2} C_{N-1,m} a_m
=
1+\sum_{m=0}^{N-1} C_{N,m} a_m .
$$
By substituting the recurrence relations \eqref{cm} for $C_{N,m}$ and using the relation $$a_{m+1} = (N-m) a_m$$ from \eqref{am}, we observe that
\begin{align*}
\sum_{m=0}^{N-1} C_{N,m} a_m
=&\sum_{m=1}^{N-2}
\big((b_N-(N-1-m))C_{N-1,m}+C_{N-1,m-1}\big)a_m \\
&+\ (b_N-(N-1))C_{N-1,0}a_0 +(C_{N-1,N-2}-b_N)a_{N-1}\\
=&
\sum_{m=0}^{N-2}
C_{N-1,m}\bigg(\big(b_N-(N-1-m)\big)a_m+a_{m+1}\bigg)
-
b_N a_{N-1}  \\
=&
(b_N+1)\sum_{m=0}^{N-2} C_{N-1,m} a_m - b_N a_{N-1}.
\end{align*}

\noindent
Substituting this back into our equivalence condition, we see that the continuity holds if
$$  \frac{1}{b_N+1}+\sum_{m=0}^{N-2} C_{N-1,m} a_m
=1+ (b_N+1)\sum_{m=0}^{N-2} C_{N-1,m} a_m - b_N a_{N-1},  $$
which simplifies to
\begin{align}
	\label{eq}
	0 & =\frac{b_N}{b_N+1}	+
	b_N\sum_{m=0}^{N-2}C_{N-1,m}a_m	-	b_N a_{N-1} \nonumber \\
	& =	b_N\!\left(
	\frac{1}{b_N+1}	+
	\sum_{m=0}^{N-2}C_{N-1,m}a_m-a_{N-1}\right).
\end{align}

\noindent
On the other hand, from \eqref{tgn-1} and the definition of $a_m$ in \eqref{am} it follows immediately that
\begin{align*}
	a_{N-1} &= \frac{(\mathcal{T} 1)^{(N-1)}(r)}{r}   = \frac{r^{-b_N} \int_0^r s^{b_N} \, ds + \sum_{m=0}^{N-2} \frac{C_{N-1,m}}{r^{N-1-m}} (\mathcal{T} 1)^{(m)}(r)}{r} \\
	&= \frac{1}{b_N+1} + \sum_{m=0}^{N-2} C_{N-1,m} \frac{(\mathcal{T} 1)^{(m)}(r)}{r^{N-m}}  = \frac{1}{b_N+1} + \sum_{m=0}^{N-2} C_{N-1,m} a_m.
\end{align*}

\noindent
Therefore, \eqref{eq} holds and we
  thus conclude that   $(\mathcal T g)^{(N)}\in C([0, r_1])$. Finally, the estimate \eqref{Tb2} follows from \eqref{tN} and \eqref{wjn}.

\end{proof}

\begin{proof}[\textbf{Proof of Theorem \ref{cont}}: ]
Let the operator $\mathcal{T}$ and the constant $C$ be given as in Lemma \ref{tb}, satisfying \eqref{Tb1} and \eqref{Tb2}. Define
\begin{equation}\label{M}
    M: = \sup_{\substack{r\in [0, 1], |p_j|\leq 1, \\\\ j=0, \ldots, N-1 }}|f(r, p_0, \ldots, p_{N-1})|,
\end{equation}
and choose $0 < {r_1} < \min\{1, r_0\}$ such that $$MC{r_1} < 1.$$

\noindent
It suffices to show that there exists a $C^N([0, {r_1}])$ solution to the system \eqref{hom} on $[0,{r_1})$. Indeed, for $ {r_1}<r <r_0$, the system is non-singular. Thus  any local solution $u$ extends uniquely to a $C^N([0, r_0))$ solution by the standard continuation theorem for initial value problems of ordinary differential equations.

We construct a sequence of functions $(u_l)_{l \ge 1}$ inductively as follows. For each $l \in \mathbb{Z}^+$, define$$u_l(r) := 0, \quad r \in \left[0, \frac{{r_1}}{l}\right],$$and for each $j = 2, \ldots, l$, set
\begin{equation}
	\label{lim}
	u_l(r) := \mathcal{T}\left( f(\cdot, u_l, \ldots, u_l^{(N-1)})\right)\left(r - \frac{{r_1}}{l}\right), \quad r \in \left(\frac{(j-1){r_1}}{l}, \frac{j{r_1}}{l}\right].
	\end{equation}
	From \eqref{hom1}, it is straightforward to verify that $u_l \in C^{N-1}([0, {r_1}])$.

 We now estimate the $C^{N-1}$ norm of $u_l$. Clearly, $$\|u_l\|_{C^{N-1}([0, \frac{{r_1}}{l}])} = 0 < 1.$$ Proceeding inductively for each $j = 2, \ldots, l$, we apply \eqref{Tb1} to obtain
          $$ \|u_l\|_{C^{N-1}([\frac{{r_1}}{l} , \frac{j{r_1}}{l}])} \le C {r_1}\sup_{r\in [0, \frac{(j-1){r_1}}{l}]}|f(\cdot, u_l, \ldots, u_l^{(N-1)})|\le MC{r_1}<1.$$
       Combining these estimates, we find that
          $$  \|u_l\|_{C^{N-1}([0 , {r_1}])}<1. $$
           Thus, the sequence $(u_l)$ is  uniformly bounded in $C^{N-1}([0, {r_1}])$.

      Furthermore, by construction, $u_l^{(N)}$ is continuous on $[0, {r_1}]$ except possibly at $r = \frac{{r_1}}{l} $. By \eqref{Tb2}, we have
     \begin{equation*}
         \|u_l^{(N)}\|_{L^\infty([0, {r_1}])}\le C \sup_{r\in [0, {r_1}]}|f(\cdot, u_l, \ldots, u_l^{(N-1)})|\le C M.
     \end{equation*}
  The equicontinuity of $(u_l)$ and its derivatives up to order $N-1$ follows from the inequality
  \begin{align*}
  	|u^{(j)}_l(t_1) - u^{(j)}_l(t_2)| &\le \max\{\|u_l\|_{C^{N-1}([0, {r_1}])}, \|u_l^{(N)}\|_{L^\infty([0, {r_1}])}\} |t_1 - t_2| \\
  	&= \max\{1, CM\} \cdot |t_1 - t_2|, \quad j = 0, \ldots, N-1.
  \end{align*}

        Therefore, by the Arzel\`a-Ascoli theorem, there exists a subsequence, still denoted by $(u_l)$, such that $u_l \to u$ uniformly in $C^{N-1}([0, {r_1}])$ norm. Taking the limit as $l \to \infty$ in \eqref{lim} and invoking the continuity of $f$, we obtain
    $$  u = \mathcal T\left( f(\cdot, u, \ldots, u^{(N-1)})\right).$$

    \noindent
    Since $ f(\cdot, u, \ldots, u^{(N-1)})\in C([0,{r_1}])  $, Lemma \ref{tb} implies that $u\in C^N([0, {r_1}])$. In addition, by \eqref{hom1}, $u$ satisfies the initial value problem \eqref{hom} on $[0, {r_1}]$. The proof is complete.

\end{proof}

If $f$ is additionally locally Lipschitz continuous with respect to the variables $(u, \ldots, u^{(N-1)})$, the existence of a solution to \eqref{hom} can alternatively be established using a contraction mapping principle. For the reader's convenience, we provide the details of this alternative proof below. Recall that $f: \mathbb R^{l}\rightarrow \mathbb C^{\nu}$ is said to be locally Lipschitz if it is Lipschitz on every small neighborhood of each point. Equivalently, for every compact set $K\subset \mathbb R^l$, there exists some constant $L$ dependent on $K$ such that
$$ |f(x)-f(y) |\le L|x-y|  \ \ \text{for}\ \ x, y\in K.  $$

\begin{thm}
	\label{iode}
     Let $r_0>0, N\in \mathbb Z^+$ and   $(b_1, \ldots, b_N)\in \mathbb C^N$ satisfy  \eqref{bjN}. There exists a   $ C^{N}([0,  r_0))$  solution to the nonlinear differential system \eqref{hom}
provided that $f$ is continuous in all variables and locally Lipschitz continuous with respect to the variable $(u, \ldots, u^{(N-1)})$.

\end{thm}

\begin{proof}
Let $r_1\in (0, r_0)$ be a sufficiently small   number  to be determined later. As in the proof of Theorem \ref{cont}, we only need to show the existence of $C^N$ solutions on $[0, r_1] $.  Let $X$ be the closed subset of $ C^{N}([0,r_1]) $ defined by
\[
X := \left \{ u \in C^{N}([0,r_1]) : u^{(j)}(0)=0 \text{ for } j=0,\dots,N-1, \ \text{and}\ \|u\|_{C^{N-1}([0, r_1])}\le 1 \right \}.
\]
Then $X$ is a complete metric space with respect to the induced metric $$d(u, v): = \|u-v\|_{C^{N-1}([0, r_1])} \ \ \text{for}\ \   u, v\in X.$$  We shall establish the existence of $C^N$ solutions by applying the  contraction mapping principle (see \cite[Theorem 5.7]{Br}) to a suitably defined operator on $X$.

Given $u\in X$, define
$$
\Phi(u):=\mathcal T\big(f(\,\cdot\,,u,\ldots,u^{(N-1)})\big).
$$
By Lemma \ref{tb} and the fact that $f(\cdot, u, \dots, u^{(N-1)}) \in C([0, r_1])$, we know $\Phi(u) \in C^{N}([0,r_1])$ with $u^{(j)}(0)=0$ for  $j=0,\dots,N-1$. Moreover, by \eqref{Tb1},
$$ \|\mathcal T\big(f(\,\cdot\,,u,\ldots,u^{(N-1)})\big)\|_{C^{N-1}([0, r_1])}\le Cr_1\| f(\,\cdot\,,u,\ldots,u^{(N-1)}) \|_{C([0, r_1])}\le Cr_1M,$$
where $M$ is defined in \eqref{M}. Thus, by choosing $r_1 $ such that $Cr_1M\le 1 $, one has $\Phi$ maps $X$ into itself.

Furthermore, since $f$ is locally Lipschitz continuous with respect to     $(u, \dots, u^{(N-1)})$, for any $u, v \in X$, the functions $f(\cdot, u, \dots, u^{(N-1)})$ and $f(\cdot, v, \dots, v^{(N-1)})$ belong to $C([0, r_1])$ and satisfy the estimate
$$
\|f(\cdot,u,\ldots,u^{(N-1)})-f(\cdot,v,\ldots,v^{(N-1)})\|_{C([0,r_1])}
\le L \|u-v\|_{C^{N-1}([0, r_1])},
$$
where $L>0$ is the Lipschitz constant of $f(r, p_0, \ldots, p_{N-1})$ on the compact set $\{r\in [0, 1], |p_j|\leq 1, j=0, \ldots, N-1   \}$. Consequently, by the estimate in \eqref{Tb1}, we have
\begin{align*}
d(\Phi(u), \Phi(v)) &= \|\Phi(u)-\Phi(v)  \|_{C^{N-1}([0, r_1])}\\
& \le C\, r_1\, \|f(\cdot,u,\ldots,u^{(N-1)})-f(\cdot,v,\ldots,v^{(N-1)})\|_{C([0,r_1])} \\
& \le C L  \, r_1 \|u-v\|_{C^{N-1}([0, r_1])} =(CLr_1)   d(u, v).
\end{align*}
By further shrinking  $r_1 $  such that $C L r_1  < 1$, the mapping $\Phi$ becomes a contraction on the complete metric space $X$.

According to the  contraction mapping principle, $\Phi$ possesses a unique fixed point $u \in X \subset C^N([0, r_1])$. By construction, this fixed point satisfies the homogeneous initial conditions and solves the differential system
$$
\mathcal D_{b_N}\cdots \mathcal D_{b_1}u
=\mathcal D_{b_N}\cdots \mathcal D_{b_1}\left( \Phi( u) \right)= f(r,u,\ldots,u^{(N-1)})
\quad\text{on }(0,r_1).
$$

\end{proof}

We are now ready to prove the existence and uniqueness of solutions with nonhomogeneous initial conditions, specifically those corresponding to the orders in $\Lambda$.

\begin{proof}[\textbf{Proof of Theorem \ref{iodea}}: ]

    First, by Lemma \ref{el}, the function  $$h: = \sum_{j\in \Lambda} \frac{\zeta_j}{j!}r^j$$ satisfies
 \begin{equation*}
 	\left\{ \begin{aligned}
            &   \mathcal D_{b_N} \cdots\mathcal D_{b_1} u   = 0,    \ \ r\in   (0, r_0)\\
 & u^{(j)}(0)   = \zeta_j, \ \ j\in  \Lambda;\\
  &  u^{(j)}(0)   = 0, \ \ j\in A\setminus \Lambda.
            \end{aligned}\right.
            \end{equation*}

\noindent
We seek  a solution  to \eqref{indicial} of the form $u={\tilde u}+h$. Substituting this into the equation yields the following initial value problem 
   \begin{equation*}
   	\left\{
   	\begin{aligned}
         &  \mathcal D_{b_N} \cdots\mathcal D_{b_1}{\tilde u}(r)    =  \tilde{f} (r, {\tilde u}, \ldots, {\tilde u}^{(N-1)}),
   \quad r\in   (0, r_0) \\
& \tilde{u}(0)  = \tilde{u}'(0)   = \cdots = \tilde {u}^{(N-1)}(0),
            \end{aligned}
            \right.
\end{equation*}
 where the nonlinearity is defined as $$ \tilde f(r,{\tilde u}, \ldots, {\tilde u}^{(N-1)}): = f(r, {\tilde u}+h, \ldots, ({\tilde u}+h) ^{(N-1)}).$$
 According to the existence result in Theorem \ref{cont}, this homogeneous initial value problem admits a solution $\tilde{u} \in C^N([0, r_0))$. Consequently, $u = \tilde{u} + h$ is a $C^N$ solution to \eqref{indicial}.

Next, we establish uniqueness. Suppose $u$ and $v$ are both $C^{N}$ solutions to the system \eqref{indicial}. Then $w:  = u - v$  and its derivatives up to order $N-1$ vanish at the origin. Moreover, $w$ satisfies
  \begin{align*}
  		\mathcal D_{b_N} \cdots\mathcal D_{b_1}   {w}  & =   {f}(r,  {  {u},       \dots,   {u}^{(N-1)}} ) -    {f}(r,  {  {v},     \dots,   {v}^{(N-1)}} ) .
  \end{align*}
  Taking the absolute value of both sides and applying the Lipschitz continuity of $f$, we obtain that for any $r_1 < r_0$, there exists a constant $L> 0$ such that
    $$ | \mathcal D_{b_N} \cdots\mathcal D_{b_1}   w (r)| \leq L \sum_{j=0}^{N-1} \left |    w^{(j)}  (r)\right |   \ \ \text{for}\ \ r \in (0, r_1). $$
	 It follows from Theorem \ref{gron} with $h(r)=Lr$ that $w  \equiv 0$  on the interval $[0, r_1]$. Since $r_1$ was chosen arbitrarily, we conclude that $u \equiv v$ on $[0, r_0)$. This completes the proof.

\end{proof}

 \vspace{.2in}

\section{Proof of Theorems \ref{thm:manifold_tri} and  \ref{thm:application_manifold_2m-1}}
\label{sec8}

Let $(M, g)$ denote a smooth, $n$-dimensional manifold equipped with a Riemannian metric $g$. For a fixed point ${p} \in M$, a  normal neighborhood $U_p$ of ${p}$ is the diffeomorphic image under the exponential map $\exp_{p}$ of a bounded star-shaped neighborhood of $0 \in T_{{p}}M$ (\cite{Lee}). Any point in $U_p$ is connected to ${p}$ by a unique, length-minimizing geodesic within $U_p$.

The proof of Theorem \ref{thm:manifold_tri} proceeds by reducing the inequality \eqref{eqn:Delta_g^m_inequality} for geodesically radially symmetric functions to an ordinary differential system of the form \eqref{gode1}, which makes Theorem \ref{gron} applicable. This reduction is carried out below in detail using geodesic polar coordinates.

To begin with, we may choose geodesic polar coordinates $(r, \theta^1, \dots, \theta^{n-1})$ in $U_p$, where 
$$r := d_g(\cdot, {p})$$ 
is the geodesic distance from the center that is smooth in $U_p\setminus \{{p}\}$. Here the angular coordinates $ \theta = (\theta^1, \dots, \theta^{n-1})$ are defined on the standard unit sphere $S^{n-1} \subset T_{{p}}M$, and parameterize points on the geodesic sphere $S_{r}({p})$ via the exponential map $ \exp_{{p}}(r\theta)$.
In these coordinates, each point in $ U_p \setminus \{{p}\}$ is represented by $(r,\theta) \in V$, with
\begin{equation*} 
    V:=\{(r, \theta): 0< r<r_\theta, \ \theta\in S^{n-1}\},  
\end{equation*}   
where $r_\theta>0$ denotes the radial bound  in the direction   $\theta$. 
By the Gauss Lemma, the metric $g$ on $U_p \setminus \{{p}\}$ decomposes in the geodesic polar coordinates as
\begin{equation}\label{gme}
	g = dr^2 + \sum_{a,b=1}^{n-1}g_{ab}(r, \theta) d\theta^a d\theta^b = dr^2 + r^2\sum_{a,b=1}^{n-1}\gamma_{ab}(r, \theta) d\theta^a d\theta^b, \quad (r, \theta) \in V,
\end{equation}
where $\gamma(r, \theta) := (\gamma_{ab}(r, \theta))_{a, b=1}^{n-1}$ extends as a smooth family of Riemannian metrics on $S^{n-1}$ defined for $r \in [0, r_\theta)$ such that $\gamma(0, \cdot)$ is the standard round metric on $S^{n-1}$. In particular,   $\gamma$ is uniformly positive definite on $\overline{V}$.

Consequently, we introduce the scale-invariant frame
\begin{equation}
	\label{frame}
	X_0 := \partial_r, \quad X_a := r^{-1}\partial_{\theta^a} \quad \text{for } a=1, \dots, n-1.
\end{equation}
A direct computation shows
\begin{equation*}
	g(X_0, X_0) = 1, \quad g(X_0, X_a) = 0, \quad \text{and} \quad g(X_a, X_b) = \gamma_{ab}(r, \theta).
\end{equation*}
Let $\tilde{g}_{kl} := g(X_k, X_l), 0\le k, l \le n-1$. Then the matrix $(\tilde{g}_{kl})$ takes the block-diagonal form
$$(\tilde{g}_{kl}) =
\begin{pmatrix}
	1 & 0 \\
	0 & \gamma_{ab}(r,\theta)
\end{pmatrix}.$$
Since $\gamma_{ab}(r,\theta)$ is smooth and uniformly positive definite up to $r=0$, it follows that $(\tilde{g}_{kl})$ is also smooth and uniformly positive definite near $r=0$. Consequently, the frame $\{X_j\}$ is uniformly equivalent to an orthonormal frame. With respect to this frame, we define the connection coefficients $\omega_{ij}^k$ by
$$\nabla_{X_i} X_j = \sum_{k=0}^{n-1} \omega_{ij}^k X_k,$$
where $\nabla$ denotes the Levi-Civita connection associated with $g$. Clearly, $ \omega_{ij}^k$ is smooth on $U_p\setminus \{{p}\}$. We first establish estimates for $\omega_{ij}^k$ in the following lemma.

\begin{lemma}
	\label{lem:polar_connection}
	With the above notations, there exists some constant $C > 0$ such that
	$$|\omega_{ij}^k | \leq C r^{-1} \quad \text{on } U_p \setminus \{{p}\}$$
	for all $i, j, k \in \{0, \dots, n-1\}$.
\end{lemma}

\begin{proof}
	Since
	\begin{equation*}
		g(\nabla_{X_i} X_j, X_l) = \sum_{k=0}^{n-1} \omega_{ij}^k \tilde{g}_{kl},
	\end{equation*}
	multiplying both sides by the inverse matrix $(\tilde{g}^{kl})_{n \times n}$ yields
	$$\omega_{ij}^k = \sum_{l=0}^{n-1} \tilde{g}^{kl} g(\nabla_{X_i} X_j, X_l).$$
	Due to the uniform boundedness of $(\tilde{g}^{kl})_{n \times n}$, it suffices to show that the terms $g(\nabla_{X_i} X_j, X_l)$ are all $O(r^{-1})$. By the Koszul formula,
	\begin{align*}
		2 g(\nabla_{X_i} X_j, X_l) &= \underbrace{X_i \left(g(X_j, X_l) \right) + X_j \left( g(X_i, X_l) \right) - X_l \left( g(X_i, X_j) \right)}_{\text{part I}} \\
		&\quad + \underbrace{g([X_i, X_j], X_l) - g([X_i, X_l], X_j) - g([X_j, X_l], X_i)}_{\text{part II}}.
	\end{align*}
	
	We first analyze the derivative terms in part I. Since $\gamma$ is smooth at $r=0$, $\partial_r \gamma_{ab}$ and $\partial_{\theta^c} \gamma_{ab}$ (where $c=1, \dots, n-1$) are all bounded. Consequently, the radial derivatives $X_0 \left( g(X_a, X_b) \right) = \partial_r \gamma_{ab}$ are $O(1)$, and the angular derivatives $X_c \left( g(X_a, X_b) \right) = r^{-1} \partial_{\theta^c} \gamma_{ab}$ are $O(r^{-1})$. Thus, all terms in part I are at most $O(r^{-1})$.
	
	For terms in part II, the only non-vanishing commutators for the frame elements are
	$$[X_0, X_a] = \partial_r(r^{-1}\partial_{\theta^a}) - r^{-1}\partial_{\theta^a}\partial_r = -r^{-2}\partial_{\theta^a} = -r^{-1}X_a,$$
	while $[X_a, X_b] = 0$ for all $a,b=1, \dots, n-1$. Therefore, all bracket terms are either $0$ or of the form $g(\pm r^{-1} X_a, X_l)$. Given that $g(X_i, X_j)$ is bounded, those terms are all $O(r^{-1})$. Hence, all terms in part II are $O(r^{-1})$.
	
	Combining the two parts, we have $g(\nabla_{X_i} X_j, X_l) = O(r^{-1})$. This completes the proof.
	
\end{proof}

\vspace{.05in}

\begin{lemma}
	\label{lem:polar_covariant_estimate}
	With the preceding notations, for each positive integer $k \le N$, there exists a constant $C_k > 0$ such that for every vector-valued function $f \in C^N(U_p\setminus \{{p}\}, \mathbb{R}^\nu)$, 
	$$|\nabla_g^k f | \le C_k \sum_{\alpha+|\beta| \le k} r^{\alpha-k} \left |\partial_r^\alpha\partial_\theta^\beta f \right | \quad \text{on } U_p \setminus \{{p}\}.$$
	
\end{lemma}

\begin{proof}
	Since the frame $\{X_0,\dots,X_{n-1}\}$ is uniformly equivalent to an orthonormal frame, there exists $C>0$ such that
	\[
	|\nabla_g^k f| \le C \sum_{i_1,\dots,i_k=0}^{n-1} \big| (\nabla_g^k f)(X_{i_1},\dots,X_{i_k}) \big|.
	\]
	Thus it suffices to estimate the components
	\[
	T_I^{(k)} := (\nabla_g^k f)(X_{i_1},\dots,X_{i_k}).
	\]
	We claim that for each multi-index $I = (i_1, \ldots, i_k)$ of length $k$,
	\begin{equation}
		\label{eqn:component_expansion}
		T_I^{(k)} = \sum_{\alpha+|\beta|\le k} A_{I,\alpha,\beta}(r,\theta)\, r^{\alpha-k}\,\partial_r^\alpha \partial_\theta^\beta f,
	\end{equation}
	where $A_{I,\alpha,\beta}$ are smooth and uniformly bounded functions. This claim immediately implies the desired estimate.
	
	We prove \eqref{eqn:component_expansion} by induction on $k$. For $k=1$, the components of the differential are simply the directional derivatives:
	$$(\nabla_g f)(X_0) = X_0 f = \partial_r f, \qquad (\nabla_g f)(X_a) = X_a f = r^{-1}\partial_{\theta^a} f.$$
	Both expressions match the form of \eqref{eqn:component_expansion} with constant scalar coefficients, satisfying the base case.
	
	Assume the claim holds for some $k\ge 1$. For $j\in\{0,\dots,n-1\}$, the definition of the covariant derivative gives
	\begin{align*}
		T_{jI}^{(k+1)} &:= (\nabla_g^{k+1} f)(X_j, X_{i_1}, \dots, X_{i_k}) \\
		&= X_j \left ( T_I^{(k)} \right ) - \sum_{s=1}^k (\nabla_g^k f) \left (X_{i_1}, \dots, \nabla_{X_j}X_{i_s}, \dots, X_{i_k}\right ) \\
		&= \underbrace{ X_j \left ( T_I^{(k)} \right ) }_{\text{Derivative Term}} - \underbrace{ \sum_{s=1}^k \sum_{l=0}^{n-1} \omega_{j i_s}^l  T_{i_1, \dots, l, \dots, i_k}^{(k)} }_{\text{Connection Terms}}.
	\end{align*}
	
	For the derivative term, if $j=0$, then $X_0=\partial_r$, and differentiating \eqref{eqn:component_expansion} gives
	\begin{align*}
	&	X_0 \left ( T_I^{(k)} \right ) = \partial_r \left ( \sum_{\alpha+|\beta| \leq k} A_{I,\alpha,\beta}(r,\theta) r^{\alpha-k} \partial_r^\alpha \partial_\theta^{\beta} f \right ) \\
		& = \sum_{\alpha+|\beta| \leq k} \Big[ (\partial_r A_{I,\alpha,\beta}) r^{\alpha-k} \partial_r^\alpha \partial_\theta^\beta f + A_{I,\alpha,\beta}(\alpha-k) r^{\alpha-k-1}\partial_r^\alpha \partial_\theta^{\beta} f + A_{I,\alpha,\beta} r^{\alpha-k} \partial_r^{\alpha+1} \partial_\theta^{\beta} f \Big].
	\end{align*}
	By consolidating terms and shifting the indices of summation as needed to absorb the higher derivative, we can rewrite the expression using the original dummy variables $\alpha$ and $\beta$ to obtain
	$$X_0 \left ( T_I^{(k)} \right ) = \sum_{\alpha+|\beta|\le k+1} \widetilde A_{I,\alpha,\beta}(r,\theta)\, r^{\alpha-(k+1)}\, \partial_r^{\alpha} \partial_\theta^\beta f,$$
	where $\widetilde A_{I,\alpha,\beta}$ are smooth and bounded. 
	
	If $j=a$ for $a=1, \dots, n-1$, then $X_a=r^{-1}\partial_{\theta^a}$, and similar differentiation yields
	\begin{align*}
		X_a \left ( T_I^{(k)} \right ) & = r^{-1}\partial_{\theta^a} \left ( \sum_{\alpha+|\beta| \leq k} A_{I,\alpha,\beta}(r,\theta) r^{\alpha-k} \partial_r^\alpha \partial_\theta^\beta f \right ) \\
		& = \sum_{\alpha+|\beta| \leq k} \left [ (\partial_{\theta^a} A_{I,\alpha,\beta}) r^{\alpha-(k+1)} \partial_r^\alpha \partial_\theta^\beta f + A_{I,\alpha,\beta} r^{\alpha-(k+1)} \partial_r^\alpha \partial_\theta^{\beta+e_a} f \right ].
	\end{align*}
	Again, by shifting the multi-index in the second term to absorb $e_a$ (the $a$-th standard basis vector) and relabeling the dummy variables back to $\alpha$ and $\beta$, we find
	$$X_a \left ( T_I^{(k)} \right ) = \sum_{\alpha+|\beta|\le k+1} \widetilde A_{I,\alpha,\beta}(r,\theta)\, r^{\alpha-(k+1)}\, \partial_r^\alpha \partial_{\theta}^{\beta} f,$$
	with smooth, bounded coefficients $\widetilde A_{I,\alpha,\beta}$. Hence, the desired structure is preserved for the derivative term.
	
	For the connection terms, Lemma \ref{lem:polar_connection} implies
	$$\omega_{j i_s}^{l} = r^{-1} B_{j i_s}^{l}(r,\theta),$$
	where $B_{j i_s}^{l}$ is smooth and bounded. Combining this with the inductive expression for $T_{i_1,\dots,l,\dots,i_k}^{(k)}$, we obtain a linear combination of terms of the form
	$$\widetilde A_{I,\alpha,\beta}(r,\theta)\, r^{\alpha-(k+1)}\, \partial_r^\alpha \partial_\theta^\beta f,$$
	where $\widetilde A_{I,\alpha,\beta}$ is again smooth and bounded. Therefore, the desired structure is preserved for the connection terms as well. Altogether, the claim holds for $k+1$, completing the induction.
\end{proof}

Applying Lemma \ref{lem:polar_covariant_estimate} to a geodesically radially symmetric function $u$ yields the following simple estimate immediately.

\begin{lemma}
	\label{lemma:covariant_bound}
	Let $(M, g)$ be a Riemannian manifold, ${p} \in M$, and $N\in \mathbb Z^+$. Let $U_p$ be a normal neighborhood of $p$, and $r = d_g(\cdot, {p})$ be the geodesic distance from $p$. Then for each nonnegative integer $j\le N$, there exists a constant $C > 0$ such that, for every function $u = \phi(r) \in C^N(U_p \setminus \{p\}, \mathbb{R}^\nu)$ that is geodesically radially symmetric in $U_p$ about $p$, we have
	\begin{equation}
		\label{eqn:nabla_g_bound}
		|\nabla^{j}_g u| \leq C \sum_{k=1}^j |\phi^{(k)}(r)| \, r^{k-j} \quad \text{on} \,\,\, U_p \setminus \{{p}\}.
	\end{equation}
\end{lemma}

\vspace{.1in}

Next, we consider the Laplace-Beltrami operator $\Delta_g$. Let $|g| = \det(g_{ij})$ denote the determinant of the metric tensor. In geodesic polar coordinates $(r, \theta)\in V$, the Laplacian of a function $u$ is given by (cf. \cite{Petersen})
\begin{align*}
	\Delta_g u &
	= \frac{1}{\sqrt{|g|}} \frac{\partial}{\partial r} \left( \sqrt{|g|}  \frac{\partial u}{\partial r} \right) + \frac{1}{\sqrt{|g|}} \sum_{a,b=1}^{n-1} \frac{\partial}{\partial \theta^a} \left( \sqrt{|g|}  g^{ab} \frac{\partial u}{\partial \theta^b} \right)\\
	& = \frac{\partial^2 u}{\partial r^2} + \frac{1}{\sqrt{|g|}} \frac{\partial \sqrt{|g|}}{\partial r} \frac{\partial u}{\partial r} + \frac{1}{\sqrt{|g|}} \sum_{a,b=1}^{n-1} \frac{\partial}{\partial \theta^a} \left( \sqrt{|g|}  g^{ab} \frac{\partial u}{\partial \theta^b} \right) .
\end{align*}

Since the metric on the geodesic spheres $S_r$ scales as $g_{ab} = r^2\gamma_{ab}$ for $1 \le a, b \le n-1$, the determinant satisfies $\sqrt{|g|} = r^{n-1}\sqrt{|\gamma|}$. A direct computation then yields
\begin{equation}
	\frac{1}{\sqrt{|g|}} \frac{\partial \sqrt{|g|}}{\partial r} = \frac{1}{r^{n-1}\sqrt{|\gamma|}} \frac{\partial (r^{n-1}\sqrt{|\gamma|})}{\partial r} = \frac{n-1}{r} + \frac{1}{ \sqrt{|\gamma|}} \frac{\partial \sqrt{|\gamma|}}{\partial r}.
\end{equation}
Recall that the family $\gamma(r, \cdot)$ is smooth and uniformly positive definite for $r\in[0, r_\theta)$. By the expansion of Jacobi fields in normal coordinates, the metric components $g_{ij}$ and the volume density $\sqrt{|g|}$ are even functions of $r$ near the origin. Specifically, $\sqrt{|\gamma|} = 1 + O(r^2)$, which implies that the function
\begin{equation*}
	\frac{1}{\sqrt{|\gamma|}} \frac{\partial \sqrt{|\gamma|}}{\partial r} = \frac{\partial}{\partial r} \log \sqrt{|\gamma|}
\end{equation*}
is an odd smooth function with respect to $r$ near the origin. It follows that we can write
\begin{equation}\label{hh}
	\frac{1}{\sqrt{|\gamma|}} \frac{\partial \sqrt{|\gamma|}}{\partial r} = r\Psi(r, \theta)
\end{equation}
for some function $\Psi$ that is smooth on $ \overline{V} $.
Geometrically, the leading term is determined by the Ricci curvature, with $\Psi(0, \theta) = -\frac{1}{3} \Ric_{{p}}(\partial_r, \partial_r)$.

Furthermore, the angular part of the Laplacian simplifies as
\begin{equation}
	\label{spLap}
	\frac{1}{\sqrt{|g|}} \sum_{a,b=1}^{n-1} \frac{\partial}{\partial \theta^a} \left( \sqrt{|g|}  g^{ab} \frac{\partial u}{\partial \theta^b} \right) = \frac{1}{r^2}  \frac{1}{\sqrt{|\gamma|}}  \sum_{a,b=1}^{n-1}\frac{\partial}{\partial \theta^a} \left( \sqrt{|\gamma|}\gamma^{ab}\frac{\partial u}{\partial \theta^b} \right)  := \frac{1}{r^2} \Delta_{\gamma(r,\theta)} u,
\end{equation}
where $\Delta_{\gamma(r,\theta)}$ denotes the Laplace-Beltrami operator on the unit sphere $S^{n-1}$ with respect to the metric $\gamma(r, \theta)$. Consequently, we obtain the decomposition

\begin{equation}
	\label{eqn:Riem_Delta_u}
	\Delta_g u  = \underbrace{\left (\frac{\partial^2 u}{\partial r^2} + \frac{n-1}{r} \frac{\partial u}{\partial r} \right )}_{ \Delta^{\text{rad}} u } + r \Psi (r, \theta) \frac{\partial u}{\partial r} + \frac{1}{r^2}\,\Delta_{\gamma(r,\theta)} u,
\end{equation}
where $$\Delta^{\text{rad}} u=\frac{\partial^2 u}{\partial r^2} + \frac{n-1}{r} \frac{\partial u}{\partial r}$$ denotes the radial part of the flat Euclidean Laplacian.

\vspace{.1in}

For a geodesically radial function $u = \phi(r)$, we have
$$	\Delta^{\text{rad}}  u =  \mathcal{D}_{n-1}\mathcal{D}_0      \phi(r),$$
where the operator $ \mathcal{D}_b$ for $ b\in \mathbb C$ is defined as in \eqref{1st}. Repeated application of \eqref{eqn:Riem_Delta_u} to such functions yields the following decomposition for higher powers of the Laplace--Beltrami operator.  

\begin{lemma}
	\label{lemma:Delta_g^m}
	Let $(M, g)$ be a Riemannian manifold, $m\in \mathbb Z^+$ and $p \in M$. Let $U_p$ be a normal neighborhood of $p$, and let $r = d_g(\cdot, p)$ be the geodesic distance from $p$. Suppose $u =\phi(r) \in C^{2m}(U_p \setminus \{p\}, \mathbb{R}^\nu)$ is geodesically radially symmetric in $U_p$ about $p$. Then
	\begin{equation}
		\label{eqn:Delta_g_and_flat}
		\Delta_g^m u = \left( \mathcal{D}_{n-1}\mathcal{D}_0\right)^m \phi(r) + \sum_{k=1}^{2m-1}  E_{m,k}(r, \theta)  \frac{\phi^{(k)}(r)}{r^{2m-k-2}} \quad \text{on} \,\,\, U_p\setminus \{ p\},
	\end{equation}
	where each scalar coefficient function $E_{m,k}(r, \theta)$ extends smoothly to $\overline{V}$.
\end{lemma}

\begin{proof}
	We proceed by induction on $m$.
	The base case $m=1$ follows directly from \eqref{eqn:Riem_Delta_u}. Since $u=\phi(r)$ is a radial function, the angular derivatives all vanish. Thus 	
	\begin{equation}
		\label{eqn:Delta_g_1}
		\Delta_g u =\mathcal{D}_{n-1}\mathcal{D}_0 \phi(r)  + r\Psi(r,\theta) \phi'(r) = \mathcal{D}_{n-1}\mathcal{D}_0 \phi(r) + E_{1,1}(r,\theta) \frac{\phi'(r)}{r^{-1}},
	\end{equation}
	where $E_{1,1}:=\Psi$ is a smooth function on $\overline{V}$.
	
	Now, assume the expansion \eqref{eqn:Delta_g_and_flat} holds for $\Delta_g^m$. We will show that it holds for $\Delta_g^{m+1}$. By the inductive hypothesis,
	\begin{equation}
		\label{eqn:Delta_m+1_decomp}
		\Delta_g^{m+1} u = \Delta_g \left ( \Delta_g^{m} u \right ) = \Delta_g \big(\left( \mathcal{D}_{n-1}\mathcal{D}_0\right)^m \phi(r)  \big ) + \sum_{k=1}^{2m-1} \Delta_g \left ( E_{m,k}(r,\theta) \frac{\phi^{(k)}(r)}{r^{2m-k-2}} \right ) ,
	\end{equation}
	where $E_{m,k}$ are all smooth functions on $\overline{V}$.
	Recall that by Lemma \ref{ele2} there exist constants $a_1, \dots, a_{2m}$ such that
	\begin{equation*}
		\left( \mathcal{D}_{n-1}\mathcal{D}_0\right)^m \phi(r)= \sum_{i=1}^{2m} a_i \frac{\phi^{(i)}(r)}{r^{2m-i}}.
	\end{equation*}
	(The sum starts at $i=1$ because of the initial $\mathcal{D}_0$ operator). Applying \eqref{eqn:Delta_g_1} to $\left( \mathcal{D}_{n-1}\mathcal{D}_0\right)^m \phi(r)$ yields
	\begin{align*}
		\Delta_g \left( \left( \mathcal{D}_{n-1}\mathcal{D}_0\right)^m \phi(r) \right) 
		= \left( \mathcal{D}_{n-1}\mathcal{D}_0\right)^{m+1} \phi(r) + E_{1,1}(r,\theta)  \sum_{i=1}^{2m+1} \tilde{a}_i \frac{\phi^{(i)}(r)}{r^{2m-i}}
	\end{align*}
	for some constants $\tilde{a}_1, \dots, \tilde{a}_{2m+1}$. Hence, \eqref{eqn:Delta_m+1_decomp} becomes
	\begin{align*}
		\Delta_g^{m+1} u - \left( \mathcal{D}_{n-1}\mathcal{D}_0\right)^{m+1} \phi(r) & = E_{1,1}(r,\theta) \sum_{i=1}^{2m+1} \tilde{a}_i \frac{\phi^{(i)}(r)}{r^{2m-i}}  + \sum_{k=1}^{2m-1} \Delta_g \left( E_{m,k}(r,\theta) \frac{\phi^{(k)}(r)}{r^{2m-k-2}} \right) \\
		& =  \sum_{i=1}^{2m+1} \tilde{a}_i E_{1,1}(r,\theta)\frac{\phi^{(i)}(r)}{r^{2(m+1)-i-2}} + \sum_{k=1}^{2m-1} \Delta_g \left( E_{m,k}(r,\theta) \frac{\phi^{(k)}(r)}{r^{2m-k-2}} \right) .
	\end{align*}
	
	It remains to show that each term $ \Delta_g \left( E_{m,k}(r,\theta) \frac{\phi^{(k)}(r)}{r^{2m-k-2}} \right)$ can be expressed as a linear combination of the form $$E_j(r,\theta) \frac{\phi^{(j)}(r)}{r^{2(m+1)-j-2}} = E_j(r,\theta) \frac{\phi^{(j)}(r)}{r^{2m-j}},$$ where $j=1, \ldots, 2m+1$ and $E_j(r,\theta)$ are smooth functions on $\overline{V}$. By the product rule for Laplacian and applying \eqref{eqn:Delta_g_1} to $\frac{\phi^{(k)}(r)}{r^{2m-k-2}} $, for each $k=1, \ldots, 2m-1$, we have
	\begin{align}
		\label{eqn:Delta_g_E_induction}
		& \quad	\Delta_g \left( E_{m,k} \frac{\phi^{(k)}(r)}{r^{2m-k-2}} \right) \nonumber \\ &= E_{m,k} \Delta_g \left( \frac{\phi^{(k)}(r)}{r^{2m-k-2}}\right) + \frac{\phi^{(k)}(r)}{r^{2m-k-2}} \Delta_g E_{m,k}  + 2\left\langle \nabla_g E_{m,k}, \nabla_g \left(\frac{\phi^{(k)}(r)}{r^{2m-k-2}} \right)\right \rangle_g \nonumber \\
		&= E_{m,k} \mathcal{D}_{n-1}\mathcal{D}_0 \left(\frac{\phi^{(k)}(r)}{r^{2m-k-2}}\right) + E_{m,k} E_{1,1} \frac{1}{r^{-1}}  \frac{d}{dr}\left(\frac{\phi^{(k)}(r)}{r^{2m-k-2}}\right)\\
		&\quad + \frac{\phi^{(k)}(r)}{r^{2m-k-2}}  \Delta_g E_{m,k}  + 2 \left (\partial_r E_{m,k} \right ) \frac{d}{dr}\left(\frac{\phi^{(k)}(r)}{r^{2m-k-2}}\right). \nonumber
	\end{align}
	In the last step, we used the facts that  $\frac{\phi^{(k)}(r)}{r^{2m-k-2}}$ is a purely radial function, the radial vector field $\partial_r$ has unit norm, and $\partial_r$ is orthogonal to the geodesic spheres.
	
	For the first term of \eqref{eqn:Delta_g_E_induction}, using $\mathcal{D}_{n-1}\mathcal{D}_0 = \partial_r^2 + \frac{n-1}{r}\partial_r$, we obtain
	\begin{equation*}
		E_{m,k}\mathcal{D}_{n-1}\mathcal{D}_0 \left(\frac{\phi^{(k)}(r)}{r^{2m-k-2}}\right) = E_{m,k} \sum_{j=k}^{k+2} c_j \frac{\phi^{(j)}(r)}{r^{2m-j}}
	\end{equation*}
	for some constants $c_k, c_{k+1}, c_{k+2}$. Similarly, the second and fourth terms can be written as
	\begin{align*}
		E_{m,k}E_{1,1} \frac{1}{r^{-1}} \frac{d}{dr}\left(\frac{\phi^{(k)}(r)}{r^{2m-k-2}}\right) &= r^2E_{m,k}E_{1,1} \sum_{j=k}^{k+1} \tilde{c}_j \frac{\phi^{(j)}(r)}{r^{2m-j}},\\
		2\left (\partial_r E_{m,k} \right ) \frac{d}{dr}\left(\frac{\phi^{(k)}(r)}{r^{2m-k-2}}\right)  &= \left ( 2r \partial_r E_{m,k} \right ) \sum_{j=k}^{k+1} \tilde c_j \frac{\phi^{(j)}(r)}{r^{2m-j}}
	\end{align*}
	for some constants $\tilde{c}_k, \tilde{c}_{k+1}$.
	
	It remains to consider the third term. Applying \eqref{eqn:Riem_Delta_u} to $E_{m,k}$ gives
	\begin{align*}
		\frac{\phi^{(k)}(r)}{r^{2m-k-2}}\Delta_g E_{m,k}
		& =  \left (\partial_r^2 E_{m,k} + \frac{n-1}{r}\partial_r E_{m,k} + r \Psi \partial_r E_{m,k}+ \frac{1}{r^2}\Delta_{\gamma(r,\theta)}E_{m,k} \right ) \frac{\phi^{(k)}(r)}{r^{2m-k-2}}\\
		& =  \left ( r^2\partial_r^2 E_{m,k} + (n-1) r\partial_r E_{m,k} + r^3 \Psi \partial_r E_{m,k}+ \Delta_{\gamma(r,\theta)}E_{m,k} \right ) \frac{\phi^{(k)}(r)}{r^{2m-k}}.
	\end{align*}
	Thus each contribution has the required form, with all the resulting coefficients smooth on $\overline{V}$. Hence the decomposition \eqref{eqn:Delta_g_and_flat} holds with $m$ replaced by $m+1$. This completes the proof by induction.	
\end{proof}

The decomposition in Lemma~\ref{lemma:Delta_g^m} can be seen explicitly in the constant-curvature model spaces. The following example shows that, near the pole, the Laplace-Beltrami operator separates into the Euclidean radial Laplacian and curvature-dependent lower-order terms.

\begin{example}
	\upshape
	Consider the action of the Laplace-Beltrami operator on a radial function
	$u=\phi(r)$ in the standard simply connected space forms.
	
	For the standard unit sphere $\mathbb{S}^n$, the operator is given by $$\Delta_{\mathbb{S}^n} \phi = \phi'' + (n-1)\cot(r) \phi'.$$
	Expanding $\cot(r)$ near $r=0$ yields
	$$	\Delta_{\mathbb{S}^n} \phi = \underbrace{\phi''(r) + \frac{n-1}{r}\phi'(r)}_{\Delta^{\text{rad}} u} - \frac{n-1}{3}r\phi'(r) + O(r^3)\phi'(r).
	$$
	
	Similarly, for the hyperbolic space $\mathbb{H}^n$, one has
	$$\Delta_{\mathbb{H}^n} \phi = \phi'' + (n-1)\coth(r) \phi'.$$
	Expanding $\coth(r)$ near $r=0$ yields
	$$ 		\Delta_{\mathbb{H}^n} \phi = \underbrace{\phi''(r) + \frac{n-1}{r}\phi'(r)}_{\Delta^{\text{rad}} u} + \frac{n-1}{3}r\phi'(r) + O(r^3)\phi'(r).
	$$
	Thus, in both cases, the Euclidean radial Laplacian appears as the principal
	singular part of the operator. The curvature contributes lower-order terms of
	size $O(r)\phi'(r)$, illustrating that the deviation of the metric from the
	Euclidean metric affects only the admissible error terms and not the leading
	singular structure.
\end{example}

Finally, to prove Theorem \ref{thm:manifold_tri}, we use the following observation: all odd-order radial derivatives of a smooth, locally radially symmetric function vanish at the center.

\begin{lemma}
	\label{odd}
	Let $(M, g)$ be a Riemannian manifold, $m\in \mathbb Z^+$ and ${p} \in M$. Let $U_p$ be a normal neighborhood of ${p}$, and let $r = d_g(\cdot, p)$ be the geodesic distance from $p$. Suppose a vector-valued function $u\in C^{2m}(U_p, \mathbb{R}^\nu)$ is locally geodesically radially symmetric about $p$. Then all odd-order radial derivatives of $u$ at $p$ up to order $2m-1$ vanish:
	\begin{equation*}
		\partial_r^{2j-1}u(p) = 0, \qquad j=1,\ldots,m.
	\end{equation*}
\end{lemma}

\begin{proof}
	Fix a unit vector $\theta\in S^{n-1} \subset T_{{p}}M$, and consider the geodesic $\gamma_{\theta}(t)=\exp_{{p}}(t\theta)$, $|t|<r_\theta.$
	Define
	\[
	\widetilde\phi(t):=u(\gamma_{\theta}(t)).
	\]
	Since $u\in C^{2m}(U_p, \mathbb{R}^\nu)$ and $\gamma_{\theta}$ is smooth, we have
	$\widetilde\phi\in C^{2m}((-r_\theta,r_\theta), \mathbb{R}^\nu). $
	Moreover, by local geodesic radial symmetry,
	\[
	\widetilde\phi(t)
	=
	u(\exp_{{p}}(t\theta))
	=
	\phi(d_g({p},\exp_{{p}}(t\theta)))
	=
	\phi(|t|).
	\]
	Hence $\widetilde\phi$ is an even function. Therefore,
	\[
	\widetilde\phi^{(q)}(0)=0
	\]
	for every odd integer $q\le 2m-1$. Since the $q$-th radial derivative of $u$ at ${p}$ is identified with $\phi^{(q)}(0)=\widetilde\phi^{(q)}(0)$, we obtain
	\[
	\partial_r^{2j-1}u({p})=0,
	\qquad j=1,\ldots,m.
	\]
	This proves the lemma.
	
\end{proof}

\begin{proof}[\textbf{Proof of Theorem \ref{thm:manifold_tri}}: ]
	
	Let $u\in C^{2m} (U_p, \mathbb{R}^\nu) $ satisfy the inequality \eqref{eqn:Delta_g^m_inequality} and be geodesically radially symmetric in $U_p$ about $p$. 
	To establish the theorem, it suffices to show that if $\partial^{2j}_{r} u(p) = 0$ for all $j = 0, \dots, m-1$, then $u\equiv 0$ on $U_p$.
	
    First,  it follows from Lemma \ref{odd} that all the radial derivatives of $u$ up to order $2m-1$ vanish at $p$. 
	Let $\phi(r)$ be the associated radial profile such that $u = \phi(r)$ for $r \in [0, r_\theta)$. Then 
	\[
	\phi^{(j)}(0) = 0, \qquad j = 0, \dots, 2m-1.
	\]
	We claim that these vanishing jets up to order $2m-1$, under the constraint of \eqref{eqn:Delta_g^m_inequality}, imply that $u \equiv 0$ in $U_p$.
	
	By Lemma \ref{lemma:covariant_bound}, the covariant derivatives of $u$ satisfy the bound
	\[
	|\nabla^{j}_g u| \leq C_0 \sum_{k=1}^j |\phi^{(k)}(r)| \, r^{k-j} \ \ \text{on}\ U_p\setminus \{p\}
	\]
	for $j\ge 1$. Thus, there exists a constant $C_1 > 0$ such that
	\begin{equation}
		\label{eqn:nabla_j_u_bound_Riem}
		\sum_{j=0}^{2m-1} |\nabla_g ^{j} u| \leq C_0 \sum_{j=0}^{2m-1}\sum_{k=0}^{j}\frac{| \phi^{(k)}(r)|}{r^{j-k}} \leq C_1\sum_{k=0}^{2m-1}\frac{| \phi^{(k)}(r)|}{r^{2m-1-k}} \ \ \text{on}\ U_p\setminus \{p\}.
	\end{equation}
	
	On the other hand, by Lemma \ref{lemma:Delta_g^m},
	\[
	\Delta_g^m u = \left( \mathcal{D}_{n-1}\mathcal{D}_0\right)^{m} \phi(r)   + \sum_{k=1}^{2m-1} E_{m,k}(r, \theta) \frac{ \phi^{(k)}(r)}{r^{2m-k-2}}  \ \ \text{on}\ U_p\setminus \{p\}, 
	\]with
     $$|E_{m,k}(r, \theta)| \leq C_2  \ \ \text{on} \ V $$ for a  constant $C_2 > 0$. Therefore,  
	\allowdisplaybreaks
	\begin{align*}
		\left | \left (\mathcal{D}_{n-1} \mathcal{D}_0 \right)^m \phi (r) \right | & \leq \left | \Delta_g^m u \right | + C_2\sum_{k=1}^{2m-1} \frac{\left | \phi^{(k)}(r) \right |}{r^{2m-k-2}} \\
		& \leq C \sum_{j=0}^{2m-1} | \nabla_g^j u | + C_2\sum_{k=1}^{2m-1} \frac{\left | \phi^{(k)}(r) \right |}{r^{2m-k-2}} \\
		& \leq C \cdot C_1 \sum_{k=0}^{2m-1}\frac{| \phi^{(k)}(r)|}{r^{2m-1-k}}  + C_2\sum_{k=1}^{2m-1} \frac{\left | \phi^{(k)}(r) \right |}{r^{2m-k-2}} \\
		&= \left(CC_1 r + C_2 r^2\right) \sum_{k=0}^{2m-1} \frac{\left | \phi^{(k)}(r) \right |}{r^{2m-k}}
	\end{align*}
for $r\in (0,   r_0)$,	where  $ r_0: = \sup_{\theta\in S^{n-1}}  r_\theta<\infty $.	By Theorem \ref{gron}, we know $\phi \equiv 0$ for $r\in [0,   r_0)$, and equivalently, $u \equiv 0$ in $ U_p$.

\end{proof}

\begin{remark}
	\label{remark:h(r)_over_r}
\upshape

The above estimates show that the conclusion of Theorem \ref{thm:manifold_tri} still holds if condition \eqref{eqn:Delta_g^m_inequality} is replaced by the weaker inequality 

$$|\Delta_g^m u| \leq \frac{h(r)}{r} \sum_{j=0}^{2m-1} |\nabla_g^{j} u| \quad \text{on} \,\,\, U_p \setminus \{p\},$$
where $h  $ is a non-negative continuous function near $0$ with $h(0) = 0$. Indeed, applying this weaker bound essentially replaces the $C_1r$ term in the final estimate with $C_1h(r)$, yielding an overall coefficient bounded by $(C_1h(r) + C_2 r^2)$. Since this still vanishes at $r = 0$, Theorem \ref{gron} applies just as before.
	
\end{remark}

\hspace{.05in}

We next consider the following quasilinear Poisson system near $p \in M$:
$$
\left\{
\begin{aligned}
	&\Delta_g^m u = f(\cdot, u, \nabla_g u, \dots, \nabla_g^{2m-1}u) && \text{near } p,\\
	&\nabla_g^j u(p) = \zeta_j, && j=0,\dots,2m-1.
\end{aligned}
\right.
$$
The local existence theory for such systems in the Euclidean setting was established by the first and third authors in \cite{PZ}. More precisely, for every vector-valued function $f \in C^{1,\alpha}$ (where $0 < \alpha < 1$), there exist infinitely many local $C^{2m,\alpha}$ solutions to the above Poisson system with arbitrarily prescribed admissible initial jets. On general Riemannian manifolds, the case $m=1$ was treated by the first and second authors in \cite{PY}. With suitable modifications, their method extends to the case $m>1$. As an immediate consequence of Theorem \ref{thm:manifold_tri}, we obtain Theorem \ref{thm:application_manifold_2m-1} regarding the local uniqueness property for such systems under an additional radial symmetry assumption.

\begin{proof}[\textbf{Proof of Theorem  \ref{thm:application_manifold_2m-1}}: ]
	 Let $w := u - v$. Since $u$ and $v$ both solve \eqref{eqn:Delta^m_IVP_Riem}, we have 
 $$\Delta_g^m w = f(\cdot, u, \nabla_g u, \dots, \nabla_g^{2m-1}u) - f(\cdot, v, \nabla_g v, \dots, \nabla_g^{2m-1}v).$$
Because $f$ is locally Lipschitz and $u,v \in C^{2m}$, it follows that for every $\tilde U_p \Subset U_p$, there exists a constant $C>0$ such that
\[
|\Delta_g^m w | \le C \sum_{j=0}^{2m-1} |\nabla_g^j w |
\quad \text{in }\  \tilde U_p.
\]
Moreover, by assumption  $w$ is geodesically radially symmetric in $U_p$ about $p$, and  all even-order radial derivatives of $w$ up to order $2m-2$  vanish at $p$. It then follows from Theorem \ref{thm:manifold_tri} that $w \equiv 0$ in $\tilde U_p$. That is, $u \equiv v$ in $\tilde U_p$.  Since $\tilde U_p$ is arbitrary, $u \equiv v$ in $U_p$.

\end{proof}

\vspace{0.2cm}

In the following example, we construct a pair of smooth nonradial functions $u$ and $v$ such that  the difference $u-v$ is   geodesically radially symmetric and the jets of $u$ and $v$ at the center agree to arbitrary order. In particular, according to Theorem  \ref{thm:application_manifold_2m-1}, there is no equation of the form \eqref{eqn:Delta^m_IVP_Riem} that is satisfied simultaneously by these two functions.

\begin{example}
    \upshape
Consider the Poincar\'e metric on the unit disc $\mathbb D$ in $\mathbb C$:
$$ ds^2=\frac{4\left(dx^2+dy^2\right)}{\left(1-x^2-y^2\right)^2} $$
According to classical complex analysis theory, the hyperbolic distance from $z=x+iy\in \mathbb D$ to any point $a\in \mathbb D$ is given by 
\[
d_{\mathbb D}(z,a)
=
2\tanh^{-1}
\left|
\frac{z-a}{1-\overline{a}z}
\right|.
\]
Let $\phi\in C^\infty([0, 1))$ be a function that is flat at $0$. For instance, take $\phi(r)= e^{-\frac{1}{|r|^2}}$ on $(0, 1)$, and extend it continuously to $r=0$. Define 
$$ u(z) = x  +\phi\left(\left|\frac{z-a}{1-\overline{a}z} \right| \right)$$
and 
  $$v(z) = x.$$
Then neither $u$ nor $v$ is locally geodesically radially symmetric about $a$. However, their difference
$$u(z)-v(z) = \phi\circ \tanh\left(\frac{1}{2}d_{\mathbb D}(z,a)\right)  $$ is geodesically radially symmetric in $\mathbb D$ about $a$. Moreover, since \(\phi\) is flat at \(0\), the function \(u-v\) vanishes to
infinite order at \(a\). Hence the jets of \(u\) and \(v\) agree to arbitrary
order at \(a\).

\end{example}

\vspace{.2in}

 \section{Proof of Theorem \ref{thm:local-kahler-form-rigidity}}\label{cas}

Let \((M,\omega)\) be a K\"ahler manifold of complex dimension $n$. In this section  we study the local aspects   of the Calabi--Yau equation on $M$. We begin by analyzing a    local rigidity phenomenon    under a radial symmetry assumption. We then   establish the  existence of abundant  local potentials   by prescribing jets up to order $2$ at a fixed point $p\in M$. Altogether, these results yield    Theorem \ref{thm:local-kahler-form-rigidity}. Throughout the section, we choose local holomorphic coordinates \(z=(z^1,\ldots,z^n)\) near \(p\), and write
\[
\omega=\sqrt{-1}\,g_{i\bar j}\,dz^i\wedge d\bar z^j.
\]
Here and in what follows we use the Einstein summation convention, so repeated indices are summed over $\{1, \ldots, n\}$.
 
  The proof of the rigidity part  of Theorem \ref{thm:local-kahler-form-rigidity} is based on linearizing  equation \eqref{calabi} along the segment joining the two potential functions. As will be shown, the resulting equation for their difference satisfies a linear elliptic equation. Under the assumption that  the difference of the potential functions is locally geodesically radial about \(p\), the equation reduces,  with the aid of the following lemma, to a singular radial ODE of the type covered by Theorem~\ref{gron}.

\begin{lemma}\label{lem:distance-expansion}
Let \((M,\omega)\) be a K\"ahler manifold, and   \(g\) be the Riemannian metric induced by \(\omega\). Fix \(p\in M\), and let  $r =d_g(\cdot, p) $ be the geodesic distance from $p$. Choose holomorphic normal coordinates \(z=(z^1,\ldots,z^n)\) centered at \(p\), and write $ \xi=\frac{z}{|z|}$. Then, away from \(p\) and for \(z\) sufficiently small, we have
\begin{equation}\label{ri}
    r_i=\frac{\bar\xi^i}{2}+O(r),
\end{equation}
and
\begin{equation}\label{rij}
    r_{i\bar j}
=
\frac{1}{2r}\delta_{ij}
-
\frac{1}{4r}\bar\xi^i\xi^j
+
O(1).
\end{equation}
Here \(r_i=\partial_i r\), \(r_{i\bar j}=\partial_i\partial_{\bar j}r\), and \(O(1)\) denotes a term uniformly bounded as \(r\to 0\).
\end{lemma}

\begin{proof}
In holomorphic normal coordinates, \(r(z)=|z|+O(|z|^2)\). Therefore
\[
r_i=\partial_i r=\frac{\bar \xi^i}{2}+O(r),
\]
which proves \eqref{ri}.

It remains to prove \eqref{rij}. In geodesic polar coordinates \((r,\theta)\) for the Riemannian metric $g$ induced by \(\omega\), the metric has the form
\[
g=dr^2+r^2\gamma_{ab}(r,\theta)\,d\theta^a d\theta^b,
\]
where \(\gamma_{ab}(r,\theta)\) is smooth and uniformly positive for \(r\) small, and \(\gamma_{ab}(0,\theta)\) is the standard metric on \(S^{2n-1}\). 
Hence
\[
\nabla^2 r(\partial_r,\cdot)=0,
\qquad
\nabla^2 r(\partial_{\theta^a},\partial_{\theta^b})
=\frac12 \partial_r (r^2\gamma_{ab}) =
 r  \gamma_{ab}(0, \theta)  + O(r^2).
\]
See, for instance, \cite[pp. 49]{Petersen}. This gives the intrinsic expansion
\begin{equation}\label{exp1}
    \nabla^2 r
=
\frac1r\bigl(g-dr\otimes dr\bigr)+O(1).
\end{equation}
 Taking the \((1,\bar 1)\)-component of \eqref{exp1} and making use of  \eqref{ri} further give
\[
r_{i\bar j}
=
\frac{1}{2r}\delta_{ij}
-
\frac{1}{r}r_i r_{\bar j}
+
O(1)  = \frac{1}{2r}\delta_{ij}
-
\frac{1}{4r}\bar\xi^i\xi^j
+
O(1).
\]
This proves the lemma.

\end{proof}

\medskip

\begin{proof}[\textbf{Proof of  Theorem \ref{thm:local-kahler-form-rigidity} part 2:}]  
Define
\[
F(\varphi):=\log\frac{\omega_\varphi^n}{\omega^n}, \ \  \text{where}\ \  
\omega_\varphi:=\omega+\sqrt{-1}\,\partial\bar\partial\varphi.
\]
Then equation \eqref{calabi} is equivalent to \[F(\varphi)=f.\] Let \(\varphi_0\) and \(\varphi_1\) be two solutions and set \(u:=\varphi_1-\varphi_0\). Letting
\[\varphi_t:=\varphi_0+tu, \qquad 0\leq t\leq 1,\]  the fundamental theorem of calculus gives
\begin{equation}\label{fun}
0=F(\varphi_1)-F(\varphi_0)
=\int_0^1 D_{\varphi_t}F(u)\,dt .
\end{equation}
The linearization  $  D_{\varphi_t}F $ is computed as follows. Since $$\omega_{\varphi_t+su}
=\omega_{\varphi_t}+s\sqrt{-1}\,\partial\bar\partial u,$$ we have
\begin{equation*}
 D_{\varphi_t}F (u)
=
\left.\frac{d}{ds}\right|_{s=0}
\log\frac{
(\omega_{\varphi_t}+s\sqrt{-1}\,\partial\bar\partial u)^n
}{\omega^n}.
\end{equation*}
Let \(g_t\) be the K\"ahler metric associated to \(\omega_{\varphi_t}\).  In local holomorphic coordinates, $ g_{t,i\bar j}=g_{i\bar j}+(\varphi_t)_{i\bar j}$. Hence  
\begin{equation*} 
D_{\varphi_t} F (u)
 = \left.\frac{d}{ds}\right|_{s=0}
\log\det
(g_{t, i\bar j}+s\,u_{i\bar j})  = g_t^{i\bar j}u_{i\bar j}
\left(=
\Delta_{\omega_{\varphi_t}}u\right),
\end{equation*}
where \(g_t^{i\bar j}\) denotes the inverse matrix of \(g_{t,i\bar j}\) and the last step  has used the matrix identity $\left.\frac{d}{ds}\right|_{s=0}
\log\det(A+sB) = \operatorname{tr}(A^{-1}B)$ for two matrices $A$ and $B$ with $A$ invertible. Combining this with \eqref{fun} gives
\[
0=\left(\int_0^1 g_t^{i\bar j}\,dt\right)u_{i\bar j}.
\]
Set \(a^{i\bar j}:=\int_0^1 g_t^{i\bar j}\,dt\). Since each \(g_t^{i\bar j}\) is positive definite, \(a^{i\bar j}\) is positive definite, and \(u\) satisfies
\begin{equation}\label{linear-eq}
a^{i\bar j}u_{i\bar j}=0.
\end{equation}

Assume now that \(u\) is geodesically radial with respect to the induced Riemannian metric   \(g\) in a normal neighborhood $U_p$ of \(p\), so that \(u(z)=\phi(r(z))\) for some single variable function $\phi$, where \(r(z)=d_g(z, p)\). All derivatives \(u_{i\bar j}\) are taken with respect to local holomorphic coordinates.  By the chain rule in holomorphic coordinates,  
\[
u_i=\phi'(r)r_i  \qquad \text{and}\qquad u_{i\bar j}
=
\phi''(r)r_i r_{\bar j}
+
\phi'(r)r_{i\bar j}.
\]
Substituting this into \eqref{linear-eq}, we obtain
\begin{equation}\label{en1}
a^{i\bar j}r_i r_{\bar j}\,\phi''(r)
+
a^{i\bar j}r_{i\bar j}\,\phi'(r)
=
0.
\end{equation}
Since \(r\) is the distance function with respect to \(g\), \(|\nabla r|_g=1\)  for \(r\in (0, r_\theta)\), where $r_\theta$ is the radial bound in the direction $\theta$. By the uniform ellipticity of \(a^{i\bar j}\), we have \(0<c\leq a^{i\bar j}r_i r_{\bar j}\leq C\). Therefore, after dividing \eqref{en1} by \(a^{i\bar j}r_i r_{\bar j}\), we get
\begin{equation}\label{en3}
\phi''(r)+\widetilde a(r,\theta)\phi'(r)=0,
 \ \  \text{where}\ \  
\widetilde a(r,\theta):=
\frac{a^{i\bar j}r_{i\bar j}}
{a^{i\bar j}r_i r_{\bar j}}.
\end{equation}
Here \(\widetilde a\) is smooth on $0< r<r_\theta,\theta\in S^{2n-1}$.

We next determine the leading behavior of \(\widetilde a\) as \(r\to 0\). 
Indeed, it follows from \eqref{ri} and \eqref{rij} that \[
a^{i\bar j}r_i r_{\bar j}
=
\frac14 a^{i\bar j}(p)\bar\xi^i\xi^j
+
O(r),
\]
and
\[
a^{i\bar j}r_{i\bar j}
=
\frac{1}{2r}\sum_i a^{i\bar i}(p)
-
\frac{1}{4r}a^{i\bar j}(p)\bar\xi^i\xi^j
+
O(1) = \frac{1}{r}\left(\frac{1}{2}\sum_i a^{i\bar i}(p)
-
\frac{1}{4}a^{i\bar j}(p)\bar\xi^i\xi^j
+
O(r)\right).
\]
Consequently,
\begin{equation}\label{atil-leading}
\widetilde a(r,\theta)
=
\frac1r
\left(
\frac{2\sum_i a^{i\bar i}(p)}
{a^{i\bar j}(p)\bar\xi^i\xi^j}
-1+
O(r)
\right)
.
\end{equation}
Since \(a^{i\bar j}(p)\) is positive definite, all the eigenvalues $\lambda_i$ of $(a^{i\bar j}(p))$ are positive. We have
\[
\frac{\sum_i a^{i\bar i}(p)}
{a^{i\bar j}(p)\bar\xi^i\xi^j} \ge \frac{\sum_i \lambda_i}{\max_{i} \lambda_i}
\geq 1.
\]
Hence 
\begin{equation}\label{ct}
c_\theta:=
\frac{2\sum_i a^{i\bar i}(p)}
{a^{i\bar j}(p)\bar\xi^i\xi^j}
-1\ge 1.
\end{equation}
Combined with \eqref{atil-leading}, we have for each fixed \(\theta\in S^{2n-1}\), there exists a continuous function \(h_\theta\) on \([0,r_\theta)\), with \(h_\theta(0)=0\)  such that
\[
\widetilde a(r,\theta)=\frac{c_\theta-h_\theta(r)}{r}.
\]
Therefore \eqref{en3} can be rewritten as
\begin{equation}\label{radial-ode}
\phi''(r)+\frac{c_\theta}{r}\phi'(r)
=
h_\theta(r)\frac{\phi'(r)}{r}.
\end{equation}
Finally, recalling the operator \(\mathcal D_\alpha:=\frac{d}{dr}+\frac{\alpha}{r}\), we have $\mathcal D_{c_\theta}\mathcal D_0\phi
=
\phi''(r)+\frac{c_\theta}{r}\phi'(r). $
Thus \eqref{radial-ode} becomes
\[
\mathcal D_{c_\theta}\mathcal D_0\phi
=
h_\theta(r)\frac{\phi'(r)}{r}.
\]
Since \(c_\theta>0\) by \eqref{ct}, condition \eqref{bjN} is satisfied. Therefore Theorem~\ref{gron} applies to give $\phi\equiv 0$.

\end{proof}

Next we investigate the existence of ample local K\"ahler forms in the class of $\omega$ with the
prescribed Ricci form in a neighborhood of $p$. 
Let $\text{Herm}(n)$ be the collection of  Hermitian matrices of size $n$. \(H=(H_{i\bar j})\in \text{Herm}(n)\) is called an admissible Hessian for $\omega$ at \(p\) if
 \begin{equation*} 
   \det(g_{i\bar j}(p)+H_{i\bar j})
=
e^{f(p)}\det(g_{i\bar j}(p))\ \ \text{and}\ \   g_{i\bar j}(p)+H_{i\bar j}>0.
 \end{equation*}
Equivalently, \(H\) is admissible if the Hermitian form
\[
\omega_H(p)
:=
\sqrt{-1}\bigl(g_{i\bar j}(p)+H_{i\bar j}\bigr)
\,dz^i\wedge d\bar z^j
\]
is positive and satisfies the  Calabi constraint at $p$: 
\[
\omega_H(p)^n=e^{f(p)}\omega(p)^n.
\]
Define the set $$\Sigma_p:=\{H\in \text{Herm}(n): H   \text{ is admissible at } p\}.$$ Then $\Sigma_p$ is a hypersurface in $\text{Herm}(n) $ and  $\text{dim}_{\mathbb R}\Sigma_p = n^2-1$. Theorem \ref{thm:local-kahler-form-rigidity}  part $1$ is a direct consequence of the following

\begin{thm}\label{cale} Let \((M,\omega)\) be a smooth Kähler manifold,   \(p\in M\), and   \(f\in C^\infty\) in a neighborhood of \(p\). Let \(H=(H_{i\bar j})\) be an admissible Hessian for \(\omega\) at \(p\). Then, after possibly shrinking to a sufficiently small neighborhood of \(p\), there exists a smooth local solution \(\varphi\) of
\[
(\omega+\sqrt{-1}\,\partial\bar\partial\varphi)^n
=
e^f\omega^n
\]
such that
\[
\varphi_{i\bar j}(p)=H_{i\bar j}.
\]
Moreover, the value \(\varphi(p)\) and the differential \(d\varphi(p)\) may be prescribed arbitrarily.
\end{thm}

\medskip

The argument relies  on  constructing a contraction map in a closed subset of  H\"older spaces. The resulting fixed point then yields a local solution of  the fully nonlinear Calabi equation. To ensure that the potential attains the prescribed second-order jet \(H\), we introduce a finite-dimensional family of quadratic functions realizing these jets as approximate solutions, and subsequently adjust the parameters via a finite-dimensional inverse function theorem.

For \( 0 < \rho <1 \), $0<\alpha<1$ and $k\in \mathbb Z^+\cup\{0\}$, we will define the following  scaled H\"older norm on \(B_\rho\), the coordinate ball of radius $\rho$ with respect to the local coordinates $(z_1, \ldots, z_n)$, by
\[
\|u\|_{C^{k,\alpha}_\rho }
:=
\sum_{j=0}^k \rho^j\|\nabla^j u\|_{C }
+
\rho^{k+\alpha}[\nabla^k u]_{C^\alpha },
\]
where the semi-norm $ [v]_{C^\alpha}$ of a   function $v\in C(B_\rho)$ is defined by \[
[v]_{C^\alpha(B_\rho)}
:=
\sup_{\substack{x,y\in B_\rho,\\ x\neq y}}
\frac{|v(x)-v(y)|}{|x-y|^\alpha} 
\]
if it is finite. Throughout the proof,  $C$ denotes  a constant dependent only possibly  on the given data $\omega$, $ f$ and $H$, which may vary from line to line.   


 \begin{proof}[\textbf{Proof of Theorem \ref{cale}: }]
Without loss of generality, let \(p=0\).  For each $M \in \text{Herm}(n)$ satisfying $ g_{i\bar j}(0)+M_{i\bar j}>0$, define
\begin{equation*} 
\mathcal F(z,M)
:=
\log\det(g_{i\bar j}(z)+M_{i\bar j}) - \log\det(g_{i\bar j}(z))- f(z).
\end{equation*}
Then the Calabi equation \eqref{calabi} reduces to \(\mathcal F(\cdot,\partial\bar\partial\varphi)=0\).
For every \(K\in\Sigma (:=\Sigma_0)\), set
\[
Q_K(z):=K_{i\bar j}z^i\bar z^j.
\]
Then \((Q_K)_{i\bar j}=K_{i\bar j}\). We seek a solution of the form
\[
\varphi_K=Q_K+w_K,
\]
where $w_K$ is a smooth function on $B_\rho$ with $ w_K|_{\partial B_\rho}=0$. Thus \(w:= w_K\) should solve
\begin{equation}\label{eq:wK-equation}
\mathcal F(z,K+\partial\bar\partial w)=0
\qquad\text{in }B_\rho.
\end{equation}

Rewrite $ \mathcal F(z,K+\partial\bar\partial w) $ as the sum of the following three components: 
\begin{equation*}
    \mathcal F(z,K+\partial\bar\partial w) = \mathcal F(z,K) +L_K(w) + N_K(w),
\end{equation*}
where $L_K$ is the linearization of $\mathcal F$ in the matrix variable at $K$, and $N_K$ is the quadratic remainder after subtracting the linear part of
\(\mathcal F\) in the matrix variable. Namely, 
\begin{equation}\label{aij}
    L_Kw:=A_K^{i\bar j}(z)w_{i\bar j}, \ \ \text{with}\ \ A_K^{i\bar j}(z)
:=
D_{M_{i\bar j}} \mathcal F(z,K)
=
(g(z)+K)^{i\bar j},
\end{equation}  
and 
\[
N_K(w)
:=
\mathcal F(z,K+\partial\bar\partial w)
-
\mathcal F(z,K)
-
A_K^{i\bar j}(z)w_{i\bar j}.
\]
Thus,  \eqref{eq:wK-equation} can  be rewritten as
\begin{equation}\label{eq:fixed-point-equation}
L_Kw=- \mathcal F(\cdot,K)-N_K(w).
\end{equation}
Note that since \(g(0)+K>0\), after shrinking \(\rho\) if necessary, the operator 
$L_K $ is uniformly elliptic on \(B_\rho\) and  for \(K\) in a small neighborhood of \(H\) inside \(\Sigma\). 

We next estimate the scaled H\"older norms of $\mathcal F(\cdot,K) $ and $ N_K(w)$.  The admissibility condition \(K\in\Sigma\) ensures that  \(\mathcal F(0,K)=0\). By the smoothness of  the data,
\begin{equation}\label{eq:RK-scaled}
\| \mathcal F(\cdot,K)\|_{C^\alpha_\rho }
\leq
C\rho^\alpha
\end{equation}
for \(K\) near \(H\).
For $
N_K(w)$, Taylor's formula in the matrix variable gives
\[
N_K(w)
=
\int_0^1(1-s)\,
D_{M_{i\bar j}M_{k\bar \ell}}\mathcal F
\bigl(z,K+s\partial\bar\partial w\bigr)
w_{i\bar j}w_{k\bar \ell}\,ds .
\]
Since the matrices involved remain in a fixed compact subset of the admissible cone, the second derivatives $D_{M_{i\bar j}M_{k\bar \ell}}\mathcal F $ are uniformly bounded. Hence by the scaled H\"older product estimate, 
\begin{equation}\label{eq:NK-quadratic}
\|N_K(w)\|_{C^\alpha_\rho}
\leq
C\|\partial\bar\partial w\|_{C^\alpha_\rho}^2 \le  
C\rho^{-4}\|w\|_{C^{2,\alpha}_\rho}^2,
\end{equation}
where the last inequality made use of the fact  \(
\|\partial\bar\partial w\|_{C^\alpha_\rho}
\leq
C\rho^{-2}\|w\|_{C^{2,\alpha}_\rho} 
\)     by  definition.

 Let \(T_K\) denote the inverse operator of the linear Dirichlet problem
\[
L_Kw=h,
\qquad
w|_{\partial B_\rho}=0 
\]
for $h\in  C^\alpha(B_\rho)$. By the scaled Schauder estimate (see \cite[Chapter 6]{GT}), $T_Kh\in  C^{2, \alpha}(B_\rho)$  with
\begin{equation}\label{eq:scaled-schauder}
\|T_Kh\|_{C^{2,\alpha}_\rho}
\leq
C\rho^2\|h\|_{C^\alpha_\rho},
\end{equation}
where \(C\) is independent of sufficiently small \(\rho\) and of \(K\) near \(H\).
Define
\[
\mathcal T_K(w):=T_K\bigl(- \mathcal F(\cdot,K)-N_K(w)\bigr).
\]
We shall show that, for \(\rho>0\) sufficiently small and some constant  $M_0$ sufficiently large, \(\mathcal T_K\) maps the set
\[
\mathcal B_\rho
:=
\left\{
w\in C^{2, \alpha}(B_\rho): w|_{\partial B_\rho}=0\ \ \text{and}\ \ 
\|w\|_{C^{2,\alpha}_\rho}\leq M_0\rho^{2+\alpha}
\right\}
\]
into itself and is a contraction on \(\mathcal B_\rho\). 

Indeed,  for \(w\in\mathcal B_\rho\), by \eqref{eq:RK-scaled}-\eqref{eq:scaled-schauder},
\[
\begin{aligned}
\|\mathcal T_K(w)\|_{C^{2,\alpha}_\rho}
&\leq
C\rho^2
\left(
\| \mathcal F(\cdot,K)  \|_{C^\alpha_\rho}
+
\|N_K(w)\|_{C^\alpha_\rho}
\right)  \\
&\leq
C\rho^2
\left(
\rho^\alpha
+
\rho^{-4}\|w\|_{C^{2,\alpha}_\rho}^2
\right)  \\
&\leq
C\rho^{2+\alpha}
+
C\rho^{-2}(M_0\rho^{2+\alpha})^2  \\
&=
C(1 + M_0^2\rho^\alpha)\rho^{2+\alpha}.
\end{aligned}
\]
Choosing \(M_0\) sufficiently large (for instance, $M_0\ge 2C$) and then taking \(\rho>0\) sufficiently small (say, $\rho\le M_0^{-\frac{2}{\alpha}}$), we get
\[
\|\mathcal T_K(w)\|_{C^{2,\alpha}_\rho}
\leq
 M_0\rho^{2+\alpha}.
\]
Thus \(\mathcal T_K\) maps \(\mathcal B_\rho\) into itself.

Next, for \(w_1,w_2\in\mathcal B_\rho\),    set   \(w_t:=w_2+t(w_1-w_2)\). Then by  the fundamental theorem of calculus and \eqref{aij},
\[
\begin{aligned}
N_K(w_1)-N_K(w_2)
&=
\mathcal F(z,K+\partial\bar\partial w_1)-\mathcal F(z,K+\partial\bar\partial w_2)
-
A_K^{i\bar j}(z) (w_1-w_2)_{i\bar j} \\
&=
\int_0^1
\left[
D_{M_{i\bar j}}\mathcal F(z,K+\partial\bar\partial w_t)
-
D_{M_{i\bar j}}\mathcal F(z,K)
\right]
  (w_1-w_2)_{i\bar j}\,dt .
\end{aligned}
\]
Since the matrices remain in a compact subset of the admissible cone,
\[
\left|D_{M_{i\bar j}}
\mathcal F(z,K+\partial\bar\partial w_t)
-
D_{M_{i\bar j}}\mathcal F (z,K)
\right|
\leq
C|\partial\bar\partial w_t|.
\]
Hence, using the scaled H\"older product estimate,
  \begin{equation*}\label{eq:NK-lip-ddbar}
\begin{split}
   \|N_K(w_1)-N_K(w_2)\|_{C^\alpha_\rho}
\leq&
C
\left(
\|\partial\bar\partial w_1\|_{C^\alpha_\rho}
+
\|\partial\bar\partial w_2\|_{C^\alpha_\rho}
\right)
\|\partial\bar\partial(w_1-w_2)\|_{C^\alpha_\rho}\\
\le & C\rho^{-4}
\left(
\|w_1\|_{C^{2,\alpha}_\rho}
+
\|w_2\|_{C^{2,\alpha}_\rho}
\right)
\|w_1-w_2\|_{C^{2,\alpha}_\rho}. 
\end{split}
\end{equation*}
 Together with \eqref{eq:scaled-schauder} and linearity of the operator $T_K$, we obtain
\[
\begin{aligned}
\|\mathcal T_K(w_1)-\mathcal T_K(w_2)\|_{C^{2,\alpha}_\rho}&= \|  T_K\left(N_k(w_1) -N_k(w_2)\right)\|_{C^{2,\alpha}_\rho} \\ 
&\leq
C\rho^2
\|N_K(w_1)-N_K(w_2)\|_{C^\alpha_\rho} \\
&\leq
C\rho^2\rho^{-4}
\left(
\|w_1\|_{C^{2,\alpha}_\rho}
+
\|w_2\|_{C^{2,\alpha}_\rho}
\right)
\|w_1-w_2\|_{C^{2,\alpha}_\rho} \\
&\leq
C\rho^{-2}(2M_0\rho^{2+\alpha})
\|w_1-w_2\|_{C^{2,\alpha}_\rho} \\
&=
CM_0\rho^\alpha
\|w_1-w_2\|_{C^{2,\alpha}_\rho}.
\end{aligned}
\]
Shrinking \(\rho>0\) sufficiently small again so that \(CM_0\rho^\alpha<1\), we conclude that \(\mathcal T_K\) is a contraction on \(\mathcal B_\rho\). By the contraction mapping principle, \(\mathcal T_K\) has a unique fixed point \(w_K\in\mathcal B_\rho\) satisfying \begin{equation}\label{eq:wK-bound}
\|w_K\|_{C^{2,\alpha}_\rho}
\leq
C\rho^{2+\alpha}.
\end{equation}  In particular, this $w_K$ satisfies \eqref{eq:fixed-point-equation}, which is equivalent to  \eqref{eq:wK-equation}. By elliptic bootstrapping, \(w_K\) is smooth on $B_\rho$. 
Thus \(\varphi_K:=Q_K+w_K\) is a smooth local solution to \eqref{calabi} for every \(K\in\Sigma\) close to \(H\). Its complex Hessian at \(0\) is
\[
(\varphi_K)_{i\bar j}(0)
=
K_{i\bar j}+(w_K)_{i\bar j}(0).
\]

 \vspace{0.2cm}
Next we adjust \(K\) so that the complex Hessian of $\varphi_K$ at the origin is exactly \(H\). Define an operator 
\begin{equation}\label{E}
\begin{aligned}
E_\rho:\Sigma &\to \Sigma \\
K &\mapsto (\varphi_K)_{i\bar j}(0)
= K+\partial\bar\partial w_K(0).
\end{aligned}
\end{equation}
The image indeed lies in \(\Sigma\), since \(\varphi_K\) solves the equation and therefore
 $
\mathcal F(0,E_\rho(K))=0.$
 Choose a local coordinate chart on \(\Sigma\) centered at \(H\), and identify \(H\) with \(0\). We use this chart to identify a sufficiently small neighborhood of \(H\) in \(\Sigma\) with a ball in \(T_H\Sigma\). Under this identification, we suppress the chart from the notation and write points of \(\Sigma\) simply as \(K\). In these coordinates, it suffices to find \(K_\rho\) such that $E_\rho(K_\rho)=0,$
or equivalently, by \eqref{E}, 
\[
K_\rho=-\partial\bar\partial w_{K_\rho}(0).
\]
Define
\[
R_\rho(K):=-\partial\bar\partial w_K(0).
\]
We will prove that \(R_\rho\) is a contraction on a fixed small ball $B_\delta\subset T_H\Sigma$ in these local coordinates. 
To this end, we first establish the necessary estimates. We  claim that the map \(\partial\bar\partial w_K(0)\) depends \(C^1\)-smoothly on \(K\). Moreover, for every \(\dot K\in T_K\Sigma\), if $$
\dot w_K:=D_Kw_K[\dot K], $$
then
\begin{equation}\label{ome}
\left|
\partial\bar\partial \dot w_K(0)
\right|
\leq
C\rho^\alpha |\dot K|.
\end{equation}
Assuming this claim, it follows from  \eqref{eq:wK-bound} and \eqref{ome} that 
\[
\|R_\rho\|_{C(B_\delta)}\leq C\rho^\alpha,
\qquad
\|D_KR_\rho\|_{C(B_\delta)}\leq C\rho^\alpha.
\]
Choosing \(\rho>0\) sufficiently small so that $ C\rho^\alpha<\min\left\{\frac{\delta}{2},\frac12\right\},$
we conclude that \(R_\rho\) maps \(B_\delta\) into itself and is a contraction. Hence, by the contraction mapping principle, there exists \(K_\rho\in B_\delta\) such that
\[
K_\rho=R_\rho(K_\rho).
\]
Therefore
\[
E_\rho(K_\rho)
=
K_\rho+\partial\bar\partial w_{K_\rho}(0)
=
K_\rho-R_\rho(K_\rho)
=
0.
\]
Recalling that $H$ is identified with $0$, this yields
$$E_\rho(K_\rho)=H.  $$ 
In particular, for \(\varphi:=\varphi_{K_\rho}\), we achieve the prescribed Hessian
\(
\varphi_{i\bar j}(0)=H_{i\bar j}.
\)

Now it remains to prove the claim. We first show that the dependence of \(w_K\) on \(K\) is smooth in the \(C^{2,\alpha}\)-topology. Indeed, let \(
\mathcal X_\rho
:=
\{w\in C^{2,\alpha}({B_\rho}): w|_{\partial B_\rho}=0\}
\), and  consider the map
 \begin{equation*}
    \begin{aligned}
        \mathcal P:\Sigma\times \mathcal X_\rho&\to C^\alpha(B_\rho) \\
 (K,w) &\mapsto \mathcal F(z,K+\partial\bar\partial w). 
    \end{aligned}
\end{equation*}
 At a solution \(w=w_K\), the derivative in the \(w\)-variable is
\[
D_w\mathcal P(K,w_K)[v]
=
G_K^{i\bar j}v_{i\bar j},
\]
where  
\begin{equation}\label{gij}
    G_{K,i\bar j}
:=
g_{i\bar j}+K_{i\bar j}+(w_K)_{i\bar j},
\end{equation}
and  \(G_K^{i\bar j}\) is the inverse matrix of \(G_{K,i\bar j}\).
Then
\( D_w\mathcal P(K,w_K)
\)
  is an isomorphism \(
\mathcal X_\rho\to C^\alpha(B_\rho)
\) 
by the elliptic Dirichlet theory. Hence the implicit function theorem applied to $ \mathcal P(K, w) =0$ gives the smooth dependence
$$
K\mapsto w_K
 $$ as a map into $\mathcal X_\rho.
$
Since the evaluation map
 $
w\mapsto \partial\bar\partial w(0)$
 is a bounded linear map from \(C^{2,\alpha}(B_\rho)\) to the finite-dimensional space of Hermitian matrices, it follows that
$$
K\mapsto \partial\bar\partial w_K(0)
$$
is smooth.

Differentiating \(\mathcal P(K, w_K)=0\) in the direction \(\dot K\in T_K\Sigma\), we obtain $
G_K^{i\bar j}\bigl(\dot K_{i\bar j}+(\dot w_K)_{i\bar j}\bigr)=0,$
or equivalently, $ \dot w_K$ satisfies the linear elliptic equation
\begin{equation}\label{eq:dotw-equation}
G_K^{i\bar j}(\dot w_K)_{i\bar j} = - G_K^{i\bar j}  \dot K_{i\bar j}.
\end{equation}
Since \(\dot K\in T_K\Sigma\), differentiating the constraint \(\mathcal F(0,K)=0\) gives
\[
(g(0)+K)^{i\bar j}\dot K_{i\bar j}=0.
\]
Using this  and \eqref{gij}, together with \(g(z)-g(0)=O(\rho)\) and \eqref{eq:wK-bound}, we get
\begin{equation*} 
\begin{split}
    \|G_K^{i\bar j}\dot K_{i\bar j}\|_{C^\alpha_\rho}
=&   \|  ((g+K+  \partial\bar\partial w_K)^{i\bar j}- \left(g(0)+K\right)^{i\bar j} )\dot K_{i\bar j}\|_{C^\alpha_\rho}  \\
\le &    \| \left(g(0)+K+  (g-g(0))+ \partial\bar\partial w_K\right)^{i\bar j}- \left(g(0)+K\right)^{i\bar j}) \|_{C^\alpha_\rho}|\dot K|   \\
\le& C    \|   g-g(0) + \partial\bar\partial w_K  \|_{C^\alpha_\rho} |\dot K|\\ 
\le &
C\rho^\alpha |\dot K|.
\end{split}
\end{equation*}
Here   the third line  has also used the matrix identity $ A^{-1}-B^{-1} = B^{-1}(B-A)A^{-1}$ for two invertible matrices $A$ and $B$.  In addition, since $w_K\in \mathcal X_\rho$, $ \dot w_K|_{\partial B_\rho}=0$ as well. Applying the scaled Schauder estimate \eqref{eq:scaled-schauder} to the equation \eqref{eq:dotw-equation}, we obtain
\begin{equation*}
\|\dot w_K\|_{C^{2,\alpha}_\rho}
\leq
C\rho^{2+\alpha}|\dot K|,
\end{equation*}
 from which and the fact that $\left|
\partial\bar\partial \dot w_K(0)
\right|\le \rho^{-2}  \|\dot w_K\|_{C^{2,\alpha}_\rho}$, the inequality \eqref{ome} follows.

Finally, the \(0\)-jet and \(1\)-jet of $\varphi$ may be prescribed independently. Indeed, adding a real constant and the real part of a   linear function does not change \(\partial\bar\partial\varphi\). Hence it does not change the Calabi equation or the prescribed complex Hessian, but it can adjust \(\varphi(p)\) and \(d\varphi(p)\) arbitrarily.

 \end{proof}

	\vspace{.2in}

\section{Weighted divergence-form operators}\label{last}
In  this section, we extend the preceding results for (powers of) the Laplacian   to a class of weighted divergence-form elliptic operators with an isolated singularity at the origin. For simplicity of exposition, we restrict our attention to functions defined on the ball $B_{r_0} \subset \mathbb R^n (n\ge 2) $ of radius $r_0$ centered at $0$. The results, however, remain valid on any star-shaped domain about $0$, that is, every other point in the domain can reach $0$ via a straight line segment lying entirely within the domain. 

\vspace{.05in}

  Given $(c_1, c_2) \in \mathbb{C}^2$, $u \in C^2(B_{r_0})$, and $r = |x|$, we define the operator $\Delta_{(c_1, c_2)}$ as
 $$\Delta_{(c_1, c_2)} u := r^{-c_2} \diver \left( r^{c_2-c_1} \nabla (r^{c_1} u) \right).$$
In the case when $(c_1, c_2)=(0, 0)$, it is reduced to the standard Laplacian $\Delta$.  
In general, expanding the above expression by  a direct calculation yields

\begin{equation*}
	\Delta_{(c_1, c_2)} u = \Delta u + \frac{c_1 + c_2}{r^2} (x \cdot \nabla u) + \frac{c_1(c_2 + n - 2)}{r^2} u.
\end{equation*}

\noindent
For a radial function $u(r)$, it reduces to
\begin{equation}
\label{eqn:Delta_c1_c2}
\Delta^{\text{rad}}_{(c_1, c_2)} u = u'' +\frac{c_1+c_2+n-1}{r}u' +\frac{c_1(c_2+n-2)}{r^2}u,
\end{equation}
where $\Delta^{\text{rad}}_{(c_1, c_2)}$ is the radial component of $\Delta_{(c_1, c_2)}$.
Operators of the type $\Delta_{(c_1, c_2)}$ are, in general, non-selfadjoint and lack translation or conformal invariance, except in certain specific cases (see Example \ref{ex4}).   Given $m \in \mathbb{Z}^+$ and $(c_1, \ldots, c_{2m}) \in \mathbb{C}^{2m}$, we similarly define the iterated operator
$$  \Delta_{(c_{2m-1}, c_{2m})}\cdots \Delta_{(c_1, c_2)}.$$

The following example demonstrates in detail that, when $m=2$,  the radial component of $\Delta_{(c_3,c_4)}\Delta_{(c_1,c_2)} $ is an Euler-type operator of order $4$.

\begin{example}
	\upshape
	Given $(c_1, c_2, c_3, c_4)\in \mathbb C^4$, a direct computation gives
\begin{equation*}
\Delta_{(c_3,c_4)}\Delta_{(c_1,c_2)}u
= \Delta^2 u +\frac{d_4}{r^{2}}\,(x\cdot\nabla)\,\Delta u +\frac{d_3}{r^{2}}\,\Delta u +\frac{d_2}{r^{4}}\,(x\cdot\nabla)^2u +\frac{d_1}{r^{4}}\,(x\cdot\nabla u) +\frac{d_0}{r^{4}}\,u,
\end{equation*}
where
\begin{align*}
        d_0& =  c_1 (c_3-2)\,(c_2+n-2)\,(c_4+n-4); \\
d_1 & =(c_1+c_2)(c_3-2)(c_4+n-4)
      +c_1(c_2+n-2)(c_3+c_4-4);\\
d_2 &  = (c_1+c_2)(c_3+c_4-4);\\
d_3 & = n(c_1+c_3)+c_1c_2+c_3c_4+2(c_2-c_3);\\
d_4 & = c_1+c_2+c_3+c_4.
\end{align*}
     In particular, noting $x\cdot \nabla = r\partial_r$, its radial component becomes
\begin{equation*}
\big(\Delta_{(c_3,c_4)}\Delta_{(c_1,c_2)}\big)^{\mathrm{rad}}u
=u^{(4)}(r)
+\frac{e_3}{r}\,u^{(3)}(r) +\frac{e_2}{r^{2}}\,u''(r)
+\frac{e_1}{r^{3}}\,u'(r)
+\frac{e_0}{r^{4}}\,u(r),
\end{equation*}
where
\begin{align*}
    e_0 &  = d_0;\\
    e_1 & = -(n-1)(n-3)  + (n-1)(d_3-d_4) + d_2 + d_1;\\
e_2 &= (n-1)(n-3) + (n-1)d_4 + d_3 + d_2;\\
e_3 & = 2(n-1) + d_4.
\end{align*}

\end{example}

In general, a direct computation shows that the radial component  of  $ \Delta_{(c_{2m-1}, c_{2m})}\cdots \Delta_{(c_1, c_2)}$ is a $2m$-th order Euler-type operator. Conversely, as established below, every $2m$-th order Euler-type operator can be realized as the radial component of  such a composition. Thus, the radial components of iterated divergence-form operators of the form $ \Delta_{(c_{2m-1}, c_{2m})}\cdots \Delta_{(c_1, c_2)}$ are in correspondence with Euler-type differential operators.

\begin{lemma}\label{ele3}
  For every $(a_1, \dots, a_{2m}) \in \mathbb{C}^{2m}$, there exists  some constant  $(c_1, \dots, c_{2m}) \in \mathbb{C}^{2m}$, such that the $2m$-th order Euler-type differential operator    $$u^{(2m)} + \frac{a_1}{r}u^{(2m-1)} + \cdots + \frac{a_N}{r^{2m}} u$$
   corresponds to the radial component of the operator $\Delta_{(c_{2m-1}, c_{2m})}\cdots \Delta_{(c_1, c_2)}$.
   \end{lemma}

\begin{proof}
First, according to Lemma \ref{ele2}, there exists  $(b_1, \ldots, b_{2m}) \in \mathbb{C}^{2m}$ such that the general Euler-type operator can be decomposed as a product of first-order singular operators:
$$ u^{(2m)} + \frac{a_1}{r}u^{(2m-1)} + \cdots + \frac{a_{2m}}{r^{2m}} u =\mathcal{D}_{b_{2m}} \cdots \mathcal{D}_{b_1} u. $$

For each $j=1, \ldots, m$, we seek pairs $(c_{2j-1}, c_{2j})$ such that the radial component of $\Delta_{(c_{2j-1}, c_{2j})}$ coincides with the second-order factor $\mathcal{D}_{b_{2j}} \mathcal{D}_{b_{2j-1}}$:
\begin{equation}\label{bc2}
    \Delta^{\text{rad}}_{(c_{2j-1}, c_{2j})} = \mathcal{D}_{b_{2j}} \mathcal{D}_{b_{2j-1}}.
\end{equation}
Using the expansion in \eqref{eqn:Delta_c1_c2}, the operator is given by
$$ \Delta^{\text{rad}}_{(c_{2j-1}, c_{2j})} u =  u'' +\frac{c_{2j-1}+c_{2j}+n-1}{r}u' +\frac{c_{2j-1}(c_{2j}+n-2)}{r^2}u. $$

\noindent
On the other hand, a direct expansion of the product $\mathcal{D}_{b_{2j}}\mathcal{D}_{b_{2j-1}}$ yields
 $$  \mathcal D_{b_{2j}}\mathcal D_{b_{2j-1}} u =   u''+ \frac{b_{2j-1}+b_{2j}}{r}u'+\frac{b_{2j-1}(b_{2j}-1)}{r^2}u. $$

 \noindent
To ensure \eqref{bc2}, we equate the corresponding coefficients:
\begin{equation*}
\begin{cases}
    b_{2j-1}+b_{2j}   = c_{2j-1}+c_{2j}+n-1;\\
	b_{2j-1}(b_{2j}-1)   = c_{2j-1}(c_{2j}+n-2).
\end{cases}
	\end{equation*}

\noindent
This system admits two possible solutions for the parameters $(b_{2j-1}, b_{2j})$:
\begin{equation}
	\label{b2}
	(b_{2j-1}, b_{2j}) = (c_{2j-1}, c_{2j} + n - 1) \quad \text{or} \quad (b_{2j-1}, b_{2j}) = (c_{2j} + n - 2, c_{2j-1} + 1).
\end{equation}
Equivalently, one may determine the parameters $(c_{2j-1}, c_{2j})$ as
\begin{equation}\label{c2}
	(c_{2j-1}, c_{2j}) = (b_{2j-1}, b_{2j} - n + 1) \quad \text{or} \quad (c_{2j-1}, c_{2j}) = (b_{2j} - 1, b_{2j-1} - n + 2).
\end{equation}

With either choice, the $2m$-th order Euler-type operator is realized precisely as the radial component of the iterated divergence-form operator $\Delta_{(c_{2m-1}, c_{2m})} \cdots \Delta_{(c_1, c_2)}$. This completes the proof.

\end{proof}

 Using argument analogous to those used in the proof   of Theorem \ref{thm:manifold_tri} and Remark \ref{remark:h(r)_over_r}, one obtains corresponding uniqueness results for these iterated weighted divergence-form operators.

\begin{thm}
	\label{lpdeg}
	Let $r_0>0$ and $m \in \mathbb{Z}^+$, and $(c_1, \ldots, c_{2m}) \in \mathbb{C}^{2m}$. Let $u : B_{r_0}  \to \mathbb{C}^{\nu}$ be a vector-valued $C^{2m}$ function  satisfying the differential inequality
	\begin{equation}
		\label{pdeg}
		\left| \Delta_{(c_{2m-1}, c_{2m})} \cdots \Delta_{(c_1, c_2)} u(x) \right| \leq \frac{h(|x|)}{|x| } \cdot\sum_{j=0}^{2m-1} |\nabla^j u(x)| \quad \text{on } \ B_{r_0}  \setminus \{0\}
		\end{equation}
		for some nonnegative function $h\in C([0, r_0))$ with  $  h(0) =0$. If $u$ is radially symmetric  in $B_{r_0}$ about the origin and
		\begin{equation}
			\label{bjm}
			2j - \mathrm{Re}(c_{2j-1}) - 2m < 2 \quad \text{and} \quad 2j - \mathrm{Re}(c_{2j}) - 2m < n, \ \ j=1, \ldots, m,
			\end{equation}
			  then exactly one of the following holds:
			\begin{itemize}
							\item There exists some $j_0 \in \{0, \dots, m-1\}$ such that $\partial_r^{2j_0} u(0) \neq 0$;
				\item $u \equiv 0$ in $B_{r_0}$.
			\end{itemize}
		\end{thm}

		\begin{cor} \label{gthm:application_f_x_u_nabla_u_2m-1}
			Let $r_0>0$ and $m \in \mathbb{Z}^+$, and let $(c_1, \ldots, c_{2m}) \in \mathbb{C}^{2m}$ satisfy \eqref{bjm}. Suppose $u, v: B_{r_0}  \to \mathbb{C}^{\nu}$ are two vector-valued $C^{2m}$ solutions to the nonlinear Poisson system
			\begin{equation*}
				\Delta_{(c_{2m-1}, c_{2m})} \cdots \Delta_{(c_1, c_2)} u(x) = f(x, u(x), \nabla u(x), \dots, \nabla^{2m-1} u(x))   \quad \text{on } \ B_{r_0}  \setminus \{0\},
				\end{equation*}
				with prescribed initial data $$\nabla^j u(0) = \nabla^j v(0) = \zeta_j,  \quad j = 0, \dots, 2m-1.$$
				
				\noindent
			Here the vector-valued function $f$ is continuous in all variables and locally Lipschitz continuous in $(u, \ldots, \nabla^{2m-1} u)$. If the difference $u - v$ is radially symmetric  in $B_{r_0} $ about $0$, then $u \equiv v$ in $B_{r_0} $.
			\end{cor}

\begin{proof}[\textbf{Proofs of Theorem \ref{lpdeg} and Corollary \ref{gthm:application_f_x_u_nabla_u_2m-1}:}]
	If $u$ is radially symmetric about $0$ with profile function $\phi$, we may choose $(b_1, \dots, b_N)$ according to \eqref{b2} to write
	$$\Delta_{(c_{2m-1}, c_{2m})}\cdots \Delta_{(c_1, c_2)}u(x)=\mathcal D_{b_{2m}} \cdots\mathcal D_{b_1} \phi(r).$$
	 Applying the estimate in Lemma \ref{lemma:covariant_bound} to the right-hand side of the inequality \eqref{pdeg}, 
	we further reduce  \eqref{pdeg}  to
$$| \mathcal D_{b_{2m}} \cdots\mathcal D_{b_1} \phi(r)|\le  \frac{h(r)}{r}\cdot C\sum_{j=0}^{2m-1}\frac{|\phi^{(j)}(r)|}{r^{2m-1 -j}} =  Ch(r)\cdot\sum_{j=0}^{2m-1}\frac{|\phi^{(j)}(r)|}{r^{2m -j}}.$$
   It remains  to show that   \eqref{bjm} is equivalent to \eqref{bjN}, as the rest of the proof of Theorem \ref{lpdeg} then follows directly from  Theorem \ref{gron}. The proof of Corollary \ref{gthm:application_f_x_u_nabla_u_2m-1}  is similar to that of Theorem  \ref{thm:application_manifold_2m-1} and is therefore omitted.

For any $j =1, \ldots, m,$ in the first case of \eqref{b2} where
  $ b_{2j-1} = c_{2j-1}$ and $ b_{2j} = c_{2j}+n-1 $,   \eqref{bjm} implies
	\begin{align*}
		&2j-\mathrm{Re}(c_{2j-1})-2m     =  2j-\mathrm{Re}(b_{2j-1})-2m<2  \Rightarrow  (2j-1)-\mathrm{Re}(b_{2j-1})-2m<1; \\
		&2j-\mathrm{Re}(c_{2j})-2m   =  2j-\left (\mathrm{Re}(b_{2j})-n+1 \right )-2m<n  \Rightarrow 2j-\mathrm{Re}(b_{2j})-2m<1.
	\end{align*}
If   $  b_{2j-1} = c_{2j}+n-2$ and $b_{2j} =c_{2j-1}+1$ in \eqref{b2}, then \eqref{bjm} implies
	\begin{align*}
	&	2j-\mathrm{Re}(c_{2j-1})-2m     =  2j-\left (\mathrm{Re}(b_{2j})-1 \right )-2m<2   \Rightarrow  2j-\mathrm{Re}(b_{2j})-2m<1; \\
	&	2j-\mathrm{Re}(c_{2j})-2m    =  2j-\left (\mathrm{Re}(b_{2j-1})-n+2 \right )-2m<n   \Rightarrow (2j-1)-\mathrm{Re}(b_{2j-1})-2m<1.
	\end{align*}

\noindent
In either case \eqref{bjN} is satisfied with $N=2m$. The proof is complete.

\end{proof}

\begin{example}
	\label{ex4}
	\upshape 	
        In the case when $c_2=-c_1\in \mathbb C$, $\Delta_{(c_1, c_2)}$ reduces to a Hardy-type Schr\"odinger operator  	with a (possibly complex) inverse-square potential:
 	$$\Delta_{(c_1, -c_1)} u = \Delta u + \frac{c_1(n - 2 - c_1)}{r^2} u.$$
 	
 	\noindent
 	If, in addition, that $c_1 = -c_2 = \frac{n-2}{2}$, the operator becomes:$$\Delta_{\left(\frac{n-2}{2}, -\frac{n-2}{2}\right)} u = \Delta u + \frac{(n-2)^2}{4r^2} u,$$
 	which is the conformally invariant inverse-square perturbation of the Euclidean Laplacian.
 	
    Moreover, note that condition \eqref{bjm} is automatically satisfied provided that $\mathrm{Re}(c_{2j-1}) > -2$ and $\mathrm{Re}(c_{2j}) > -n$ for $j = 1, \dots, m$. Consequently, the conclusions of Theorem \ref{lpdeg} and Corollary \ref{gthm:application_f_x_u_nabla_u_2m-1} apply to the operator $\Delta^m_{(c_1, -c_1)}$ as long as $-2 < \mathrm{Re}(c_1) < n$.
 	 	 		\end{example}
 		
\vspace{.1in}

Finally, we investigate the existence  of $C^N$ radial solutions to the weighted divergence-form operator $\Delta_{(c_{2m-1}, c_{2m})}\cdots \Delta_{(c_1, c_2)}$, subject to compatible initial conditions. For a solution to achieve $C^N$ regularity, Lemma \ref{odd} implies that all odd-order jets at the origin must vanish. Furthermore, following the observations in the proof of Theorem \ref{iode} and Remark \ref{re}, the even-order jets must also vanish, with the exception of those in the set $\Gamma$ defined as in \eqref{ga} below, the subset of the even integer indicial roots of the operator's radial component in the range $\{0, \dots, 2m-2\}$. At these specific orders, the initial jets may be prescribed arbitrarily.

\begin{thm}
	\label{pe}
	Let $r_0>0$ and  $m \in \mathbb{Z}^+$, and let $(c_1, \ldots, c_{2m}) \in \mathbb{C}^{2m}$ satisfy \eqref{bjm}. Define the index set $\Gamma$ as
	\begin{equation}
		\label{ga}
		\Gamma = \big( \left \{2j-2-c_{2j-1} : j=1, \ldots, m \right \} \cup \left \{2j-n -c_{2j} : j=1, \ldots, m \right \} \big) \cap I,
	\end{equation}
	where $I := \{ 0, 2, \ldots, 2m-2\}$. Given $\zeta_j \in \mathbb{C}$ for each $j \in \Gamma$, there exists a $C^{2m}$ radial solution on $B_{r_0} $ to the following nonlinear system 
   \begin{equation*}
	\left \{
	\begin{aligned}
		&\Delta_{(c_{2m-1}, c_{2m})}\cdots \Delta_{(c_1, c_2)}  u  =f(r, u, \partial_r  u, \dots, \partial_r^{2m-1} u) \quad \text{on } \ B_{r_0}  \setminus \{0\}; \\
		&\partial_r^{j} u(0)= \zeta_j, \quad \quad j\in  \Gamma;\\
		&\partial_r^{j} u (0)= 0, \quad \quad j\in  \{0, \ldots, 2m-1\} \setminus \Gamma,
	\end{aligned}
	\right.
\end{equation*}
		where the vector-valued function $f$ is continuous in all its variables. If, in addition, $f$ is locally Lipschitz continuous in the variables $(u, \partial_r  u, \dots, \partial_r^{(2m-1)} u)$, then this solution is the unique $C^{2m}$ radial solution to the system.
	\end{thm}

In the  special case where the operator consists of powers of the Laplacian, a straightforward computation yields the following:

\begin{example}
	\label{ex5}
    \upshape For the operator $\Delta^m$, where $c_j = 0$ for all $j = 1, \dots, 2m$, a direct computation shows that the corresponding set $\Gamma$ is precisely
    $$ \Gamma = \{ 0, 2, \ldots, 2m-2 \}. $$
    In particular, this implies that initial data  for $ \Delta^m$ can be freely prescribed at every even order up to $2m-2$.
  \end{example}

\begin{proof}[\textbf{Proof of Theorem \ref{pe}}: ]
 Suppose $u$ is radial about the origin. By Lemma \ref{odd}, the odd-order derivatives must vanish: $$u^{(j)}(0) =0, j=1, 3,\ldots, 2m-1.$$  Moreover, $u$ satisfies
  $$\Delta_{(c_{2m-1}, c_{2m})}\cdots \Delta_{(c_1, c_2)}  u  = \mathcal D_{b_{2m}} \cdots\mathcal D_{b_1}  u   = f(r, u, \dots, u^{(n-1)}), $$
  where the coefficients $(b_{2j-1}, b_{2j})$ are related to $(c_{2j-1}, c_{2j})$ by \eqref{b2}, or equivalently \eqref{c2}.

 Let $\Lambda$ be the index set defined as in \eqref{Lam} with respect to $(b_1, \dots, b_{2m})$.  We shall show that $$\Gamma = \Lambda\cap \{0, 2,\ldots,2m-2  \}. $$
 Indeed, by the definition of $\Gamma$ in \eqref{ga}, every $\gamma\in \Gamma$ satisfies
    $$\gamma = -c_{2j-1}+2j-2 \quad \text{or} \quad  -c_{2j}+2j -n$$
    for some $j=1, \ldots, m$. Substituting these into \eqref{c2}, we obtain
    $$ \gamma  +b_j-j+1 =0$$
    for some $j=1, \ldots, 2m-1 $.This establishes that $\Gamma \subset \Lambda \cap \{0, 2, \dots, 2m-2\}$.  The reverse inclusion is proved  similarly.

  Consequently, the proof reduces to establishing the existence of $C^N([0, r_0))$ solutions to the following initial value problem:
    \begin{equation*}
         \begin{cases}
             &\mathcal D_{b_{2m}} \cdots\mathcal D_{b_1}  u(r)  = f(r, u, \dots, u^{(n-1)}),\ \ r\in (0, r_0),\\
             &u^{(j)} (0) = \zeta_j, \quad j\in \Lambda\cap \{0, 2,\ldots,2m-2  \}, \\
             &u^{(j)} (0) = 0, \quad j\in \{0, \ldots, 2m-1\}\setminus \left( \Lambda\cap \{0, 2,\ldots,2m-2  \}\right).
        \end{cases}
    \end{equation*}
    The existence of such a solution is then guaranteed by Theorem \ref{iodea}.

\end{proof}

  As an immediate consequence of Theorem \ref{pe} and Example \ref{ex5}, existence and uniqueness hold for radial solutions to the initial value problem for $\Delta^m$, as stated below.

 \begin{cor}\label{pel}
	Let $r_0>0$ and $m \in \mathbb{Z}^+$. For any $(\zeta_0, \dots, \zeta_{m-1}) \in \mathbb{C}^m$, there exists a $C^{2m}(B_{r_0} )$ radial solution to the system
\begin{equation*}
	\left \{
	\begin{aligned}
		&\Delta^{m}  u  = f(r, u, \partial_r  u, \dots, \partial_r^{2m-1} u) \quad \text{on } \ B_{r_0} ; \\
		&\partial_r^{2j} u(0)= \zeta_j, \quad j= 0, \ldots, m-1; \\
		&\partial_r^{2j+1} u (0)= 0, \quad j= 0, \ldots, m-1,
	\end{aligned}
	\right.
\end{equation*}
		where the vector-valued function $f$ is continuous with respect to all its arguments. If, in addition, $f$ is locally Lipschitz continuous with respect to the variables $(u, \partial_r  u, \dots, \partial_r^{2m-1} u)$, then this is the unique $C^{2m}$ radial solution.
	
 \end{cor}

 More generally, Theorem \ref{pe} can be  applied to give existence of $C^{2m}$ solutions to  a nonlinear singular system of the following form
\begin{equation} \label{h}
    \Delta^{m}  u   + \sum_{\substack{j,k\ge 0\\0\le j+2k\le 2m-1}} \frac{d_{j, k}}{r^{2m-2k}} (x\cdot \nabla)^j\Delta^k u  = f(r, u, \partial_r  u, \dots, \partial_r^{(2m-1)} u),
\end{equation}
where $ (d_{j, k}) $ are  complex constants  and $f$ is continuous in all its variables. Indeed, noting that $ x\cdot \nabla = r\partial_r$,    the left hand side  of \eqref{h} can be rewritten as
$$ \Delta^{m}  u   + \sum_{\substack{j,k\ge 0\\0\le j+2k\le 2m-1}} \frac{\tilde d_{j, k}}{r^{2m-2k-j}}\partial_r^j\Delta^k u   $$
for some  $(\tilde d_{j, k})$ determined by $ (d_{j, k}) $. In particular, its radial component becomes
$$ \partial_r^{2m}  u
  +   \frac{a_{1}}{r } \partial_r^{2m-1} u +\cdots + \frac{a_{2m}}{r^{2m}}u $$
for some constants $(a_1, \ldots,a_{2m})\in \mathbb C^{2m}$. It then follows from  Lemma \ref{ele3}   that there exists $(c_1, \ldots, c_{2m})\in \mathbb C^{2m}$ such that the radial component of the left hand side  of \eqref{h} can be written as
\begin{equation}\label{rse}
   \big( \Delta^{m}   + \sum_{\substack{j,k\ge 0\\0\le j+2k\le 2m-1}} \frac{d_{j, k}}{r^{2m-2k}} (x\cdot \nabla)^j\Delta^k \big)^{\text{rad}} = \left( \Delta_{(c_{2m-1}, c_{2m})}\cdots \Delta_{(c_1, c_2)}\right)^{\text{rad}}.
\end{equation}

In order to state the main result in this direction, we introduce the sets
\begin{equation}\label{Qj}
   Q_{(d_{j, k})  }: = \left\{ (c_1, \ldots, c_{2m})\in\mathbb C^{2m}: \eqref{rse} \ \text{holds}    \right\}
\end{equation}
and
\begin{equation}\label{b3}
    I_3:  =  \left\{ (c_1, \ldots, c_{2m})\in\mathbb C^{2m}: 2j - \mathrm{Re}(c_{2j-1}) - 2m < 2, 2j - \mathrm{Re}(c_{2j}) - 2m < n, j=1, \ldots, m\right\},
    \end{equation}
which is the set of $(c_1, \ldots, c_{2m})$ satisfying \eqref{bjm}.

\medskip

\begin{cor}\label{re2}
   Let $r_0>0$ and $m \in \mathbb{Z}^+$ and   $\left (d_{j, k} \right )$ be   complex constants, $j, k\ge 0$, $0\le j+2k\le 2m-1$.  Suppose
   $$Q_{(d_{j, k})  }\cap I_3 \ne \emptyset, $$
   where $ Q_{(d_{j, k})  }$ and $I_3$ are the sets defined in \eqref{Qj} and \eqref{b3}, respectively. Let  $\Gamma$ be defined by \eqref{ga} with respect to some $(c_1, \ldots, c_{2m})\in Q_{(d_{j, k})  }\cap I_3$. Then given $\zeta_j \in \mathbb{C}$ for each $j \in \Gamma$, there exists a $C^{2m}(B_{r_0} )$ radial solution
to the following system  
 \begin{equation*}
	\left \{
	\begin{aligned}
		&\Delta^{m}  u   + \sum_{\substack{j,k\ge 0\\0\le j+2k\le 2m-1}} \frac{d_{j, k}}{r^{2m-2k}} (x\cdot \nabla)^j\Delta^k u  = f(r, u, \partial_r  u, \dots, \partial_r^{(2m-1)} u) \quad \text{on } \ B_{r_0}  \setminus \{0\};\\
        &\partial_r^{j} u(0)= \zeta_j, \quad \quad j\in  \Gamma;\\
		&\partial_r^{j} u (0)= 0, \quad \quad j\in  \{0, \ldots, 2m-1\}\setminus \Gamma,
	\end{aligned}
	\right.
\end{equation*}
	where the vector-valued function $f$ is continuous with respect to all its arguments. 	  If  in addition  $f$ is locally Lipschitz continuous in the variables $(u, \partial_r  u, \dots, \partial_r^{(2m-1)} u)$, then this solution is the unique $C^{2m}$ radial solution to the problem.

\end{cor}

\vspace{.2in}

\end{document}